\documentclass{amsart}
\usepackage[utf8]{inputenc}

\usepackage[utf8]{inputenc}
\usepackage{floatrow}
\usepackage{amssymb}
\usepackage{amsmath}
\usepackage{mathrsfs}
\usepackage{dsfont}
\usepackage{verbatim}
\usepackage{stmaryrd, MnSymbol}
\usepackage{enumerate, xspace}
\usepackage{changebar}
\usepackage{hyperref}
\usepackage{comment}
\usepackage{mathtools}
\mathtoolsset{showonlyrefs,showmanualtags}
\usepackage{tikz}
\usepackage{youngtab}
\usetikzlibrary{graphs,shapes}
\usepackage[edges]{forest}
\usepackage{algorithm2e}
\usepackage{cancel}
\usepackage{here}
\usepackage{amsthm}
\usepackage{caption}
\newcommand{\xc}[2]{ \xcancel{\color{#1}{#2}} }

\newtheorem{theorem}{Theorem}[section]
\newtheorem{lemma}{Lemma}[section]
\newtheorem{corollary}{Corollary}[section]
\newtheorem{proposition}{Proposition}[section]

\newtheorem{remark}{Remark}[section]

\numberwithin{equation}{section}

\def\R{{\mathbb{R}}}
\def\N{{\mathbb{N}}}
\def\Z{{\mathbb{Z}}}

\title{Relationships between two linearizations of the box-ball system : Kerov-Kirillov-Reshetikhin bijection and slot configuration}
\author{Matteo Mucciconi, Makiko Sasada, Tomohiro Sasamoto and Hayate Suda}

\begin{document}



\begin{abstract}
    The box-ball system (BBS), which was introduced by Takahashi and Satsuma in 1990, is a soliton cellular automaton.  Its dynamics can be linearized by a few methods, among which the best known is the Kerov-Kirillov-Reshetikhin (KKR) bijection using rigged partitions.  Recently a new linearization method in terms of “slot configurations” was introduced by Ferrari-Nguyen-Rolla-Wang, but its relations to existing ones have not been clarified. In this paper we investigate this issue and clarify the relation between the two linearizations. For this we introduce a novel way of describing the BBS dynamics using a carrier with seat numbers. We show that the seat number configuration also linearizes the BBS and reveals explicit relations between the KKR bijection and the slot configuration. In addition, by using these explicit relations, we also show that even in case of finite carrier capacity the BBS can be linearized via the slot configuration.
\end{abstract}

\maketitle

\section{Introduction}

We consider the box-ball system (BBS) introduced by Takahashi-Satsuma \cite{TS} and a class of its generalizations BBS$(\ell)$ introduced in \cite{TM}, which are cellular automata. The states of the system are configurations of particles (balls) on the half line $\N = \{ 1,2, \dots\}$ denoted by $\eta \in \Omega = \left\{0,1 \right\}^{\N}$, and we will assume that the site $x=0$ is always vacant (box). 
The dynamics of the BBS$(\ell)$ can be described in terms of a \emph{carrier} with capacity $\ell$. At each time step the carrier enters the system empty from the leftmost site ($x = 0$) and starts travelling to the right. It visits each site $x$ of the lattice updating its local state as follows:
\begin{itemize}
    \item if there is a ball at site $x$ and the carrier is not full, then the carrier picks up the ball; 
    \item if the site $x$ is empty and the carrier is not empty, then the carrier puts down a ball;
    \item otherwise, the carrier just {passes} through. 
\end{itemize}
The above updating rules can be summarized by recording in a function $W_{\ell} :\mathbb{Z}_{\ge 0} \to \{ 0,1,\dots, \ell\}$ the number of balls transported by the carrier as it visits each site of the lattice $\N$, i.e., we recursively define $W_{\ell} :\mathbb{Z}_{\ge 0} \to \{ 0,1,\dots, \ell\}$ as $W_{\ell}(0) = 0$ and 
\begin{equation}
    W_{\ell}(x)=W_{\ell}(x-1)+\min\{\eta(x),\ell - W_\ell(x-1)\} - \min\{1-\eta(x), W_{\ell}(x-1)\}.
\end{equation}
Then by using $W_{\ell}$, the one step time evolution of the BBS$(\ell)$ is described by the operator 
\begin{equation}
    T_{\ell} : \Omega \to \Omega
\end{equation}
which acts on states as 
\begin{equation}
    T_{\ell}\eta(x)=\eta(x)+W_{\ell}(x-1)-W_{\ell}(x).
\end{equation}
The original model, namely the BBS introduced by \cite{TS}, corresponds to the case with infinite capacity carrier, that is $\ell=\infty$. Throughout this paper, by abuse of notation, for any function $f : \Omega \to \R$ we will denote $f\left(T_{\ell}\eta\right)$ by $T_{\ell}f$ and often omit the variable $\eta$. Also, we denote $T_{\infty}$ by $T$.

The BBS has been widely studied from the viewpoint of the integrable system. In particular, the BBS can be obtained via a {\it certain discretization} of the Korteweg-de Vries equation (KdV equation)
    \begin{align}
        \partial_t u + 6u \partial_{x} u + \partial^3_{x}u = 0.
     \end{align}
As the KdV equation is known to be a soliton equation, the BBS also exhibits solitonic behavior; {indeed, this property is a consequence of the solitonic nature of a certain discretized KdV equation \cite{TTMS}.} A $k$-soliton of a given ball configuration $\eta$ is a certain substring of $\eta$ consisting of $k$ ``1"s and ``0"s. If distances of solitons are large enough, then a $k$-soliton is identified as consecutive $k$ ``1"s followed by $k$ ``0"s -- here we look at the ball configuration from left to right.  Even when the distance between some solitons is small or solitons are interacting, we can identify them precisely, via the Takahashi-Satsuma (TS) algorithm, which is recalled in Appendix A. For example, Figure \ref{fig:timeevolution} shows the time evolution of the configuration $\eta = 111000010010\dots$, which includes one $3$-soliton and two $1$-solitons, and the distance of these solitons is large enough in $\eta, T^{4}\eta$ but they interact in $T\eta, T^{2}\eta, T^{3}\eta$. In addition, it is also known that the BBS can be obtained via the zero temperature limit of {a} certain spin chain model, and the BBS inherits the symmetry of the model before taking the limit, see \cite{IKT} for details. 
Thus, despite the simple description of the dynamics, the BBS is considered an important model in mathematical physics since it has the properties of both classical and quantum integrable systems.

\begin{figure}[H]
\footnotesize
\setlength{\tabcolsep}{5pt}
\begin{center}
\renewcommand{\arraystretch}{2}
\begin{tabular}{rccccccccccccccccccccccc}
    $x$  & 1 &  2 &  3 &  4 &  5 &  6 & 7 &  8 & 9 & 10 & 11 & 12 & 13 & 14 & 15 & 16 & 17 & 18 & 19 & 20 & 21 & 22 & $\dots$ \\
    \hline\hline $\eta(x)$ & \color{red}{1} & \color{red}{1} & \color{red}{1} & \color{red}{0} & \color{red}{0} & \color{red}{0} & 0 & \color{blue}{1} & \color{blue}{0} & 0 & \color{green}{1} & \color{green}{0} & 0 & 0 & 0 & 0 & 0 & 0 & 0 & 0 & 0 & 0 & $\dots$ \\
 \hline 
 $T\eta(x)$ & 0 & 0 & 0 & \color{red}{1} & \color{red}{1} & \color{red}{1} & \color{red}{0} & \color{red}{0} & \color{blue}{1} & \color{blue}{0} & \color{red}{0} & \color{green}{1} & \color{green}{0} & 0 & 0 & 0 & 0 & 0 & 0 & 0 & 0 & 0 & $\dots$  \\
 \hline
 $T^2\eta(x)$ & 0 & 0 & 0 & 0 & 0 & 0 & \color{red}{1} & \color{red}{1} & \color{blue}{0} & \color{blue}{1} & \color{red}{1} & \color{red}{0} & \color{green}{1} & \color{green}{0} & \color{red}{0} & \color{red}{0} & 0 & 0 & 0 & 0 & 0 & 0 & $\dots$  \\
 \hline
 $T^3\eta(x)$ & 0 & 0 & 0 & 0 & 0 & 0 & 0 & 0 & \color{blue}{1} & \color{blue}{0} & 0 & \color{red}{1} & \color{green}{0} & \color{green}{1} & \color{red}{1} & \color{red}{1} & \color{red}{0} & \color{red}{0} & \color{red}{0} & 0 & 0 & 0 & $\dots$  \\
 \hline
 $T^4\eta(x)$ & 0 & 0 & 0 & 0 & 0 & 0 & 0 & 0 & 0 & \color{blue}{1} & \color{blue}{0} & 0 & \color{green}{1} & \color{green}{0} & 0 & 0 & \color{red}{1} & \color{red}{1} & \color{red}{1} & \color{red}{0}  & \color{red}{0} & \color{red}{0} & $\dots$ \\
 \hline
   \end{tabular}
\end{center}
\caption{How the ball configuration $\eta = 1110000100\dots$ evolves with time, where $T^{n}\eta$ is recursively defined as $T^{n}\eta = T \left( T^{n-1}\eta \right)$, $T^{0}\eta = \eta$. The $3$-soliton and {two} $1$-solitons identified by the TS algorithm are colored by red, blue and green, respectively.}\label{fig:timeevolution}
\end{figure}

For some classical integrable systems such as {the} KdV equation, the initial value problems are explicitly solved via the linearization of their dynamics.
The BBS{,} as well as BBS$(\ell)${,} are clearly non-linear dynamics, yet they are also known to be linearized by the Kerov-Kirillov-Reshetikhin (KKR) bijection \cite{KOSTY} using the language of rigged Young diagrams and also by a procedure called $10$-elimination \cite{MIT}. A relation between the two linearizations was studied in \cite{KS}. Recently, another linearization using the notion of the {\it slot configuration} and the {\it $k$-slots} has been introduced in \cite{FNRW}. The latter linearization is known to be useful to study the randomized BBS and its {\it generalized hydrodynamics} \cite{CS,FG,FNRW}, where generalized hydrodynamics is a relatively new theory of hydrodynamics for integrable systems, see the review \cite{D} for details. The aim of this paper is to introduce a new algorithm which also linearizes the BBS dynamics and reveals relations between the KKR bijection and the slot configuration. In a forthcoming paper \cite{S}, the relation between the $10$-elimination and the slot configuration will be considered. 

To describe the explicit relation between these linearizations, we introduce two novel ways to encode the ball configuration $\eta \in \Omega$  : 
\begin{enumerate}
    \item[(i)] In Section \ref{subseq:seat}, we introduce a {\it carrier with seat numbers} and the corresponding {\it seat number configuration} $\eta^{\sigma}_{k} \in \Omega$, for any $k \in \N$ and $\sigma\in\{ \uparrow, \downarrow\}$. We will show that the seat number configuration is a {\it sequential} generalization of the slot configuration, namely $\eta^{\sigma}_k(x)$ depends only on $(\eta(y))_{0 \le y \le x}$ but contains full information of the slot configuration, see Proposition \ref{prop:seat_slot} for details. 
    \item[(ii)] In Section \ref{sec:algo}, we introduce a new algorithm to produce a growing sequence of pairs of interlacing Young diagrams $\left(\mu^{\uparrow}(x),\mu^{\downarrow}(x)\right)_{x \in \N}$ as well as refined riggings $\left(\mathbf{J}^{\sigma}(x)\right)_{x \in \N}$ for each $\sigma\in\{ \uparrow, \downarrow\}$ from any ball configuration $\eta \in \Omega$. This procedure turns out to be useful in order to establish connections between the KKR bijection and the seat number configuration, see Proposition \ref{prop:seat_KKR} for details.  In particular, we give an intuitive meaning {to} the KKR bijection, which was a purely combinatorial object, by means of the carrier with seat numbers.
\end{enumerate}
As a result, we obtain {an} explicit relation between the KKR bijection and the slot configuration, an open problem addressed in \cite{FNRW}. Our results reveal that  the slot configuration can be defined independently of the notion of solitons, see Section \ref{subseq:seat} and Proposition \ref{prop:seat_slot} for details. In addition, we will see that the slot configuration is more related to ``energy functions" than solitons, see Proposition \ref{prop:seat_KKR} and the discussion following it. We also explain an interpretation of our result. First, we note that the original slot configuration is defined via the notion of solitons identified by the TS algorithm. In this sense, the slot configuration sees the BBS as a classical integrable system. On the other hand, the linearization property of the rigged configuration obtained via the KKR bijection is closely related to a combinatorial $R$ matrix which satisfies the Yang-Baxter equation, that is, in this formalism the BBS is treated as a quantum integrable system \cite{IKT}. Therefore, the present result can be considered as a new way to connect two different perspectives (classical and quantum) on the BBS.

\subsection{Outline}

The rest of the paper is organized as follows. In Section \ref{sec:main}, we first define the seat number configuration $(\eta^{\sigma}_{k})_{k \in \N}$. Then we explain how the original BBS is linearized by simple observations on seat numbers, see Theorem \ref{thm:seat_ISM}. In the subsequent subsection, we briefly summarize the relationships between other linearizations and the seat number configuration, where the main results in this direction are Propositions \ref{prop:seat_KKR} and \ref{prop:seat_slot}. Finally, we state the relation between the KKR bijection and the slot configuration in Theorem \ref{thm:2}. As a direct consequence of Theorem \ref{thm:2}, we show that the BBS$(\ell)$ can be linearized by the slot decomposition for any $\ell < \infty$ as well as the seat number configuration, see Theorem \ref{thm:slot_finite}. 
Some possible extensions for other models and applications to generalized hydrodynamics of our results are also discussed at the end of Section \ref{subsec:relation}.
In Section \ref{sec:seat}, we describe some basic properties of seat number configurations and we give a proof of Theorem \ref{thm:seat_ISM}. In Section \ref{sec:KKR}, first we recall the definition of rigged configurations and of the KKR bijection. Then, we introduce the {\it interlacing Young diagrams} algorithm, and prove Proposition \ref{prop:seat_KKR} by using this algorithm. In Section \ref{sec:slot}, first we recall the definition of the slot configuration and the corresponding slot decomposition. Then, we prove Proposition \ref{prop:seat_slot}, Theorem \ref{thm:2} and Theorem \ref{thm:slot_finite}.

\section{Main results}\label{sec:main}

In this section we introduce a carrier with seat numbers, and corresponding seat number configuration. Unlike the exisiting methods (KKR bijection and the slot configuration), the seat number configuration can always be defined for any $\eta \in \Omega$, and linearize the dynamics of the BBS starting from $\eta$. When {$\eta$ satisfies $\sum_{x \in \N} \eta(x) < \infty$}, we can obtain {an} explicit relation between the KKR bijection / slot configuration and the seat number configuration. As a result, we {determine} relationships between the KKR bijection and the slot configuration. 

\subsection{Seat number configuration}\label{subseq:seat}

Now, we consider a situation in which the seats of the carrier, introduced in the previous section, are indexed by $\N$ and the refined update rule of such a carrier is given as follows:
\begin{itemize}
    \item if there is a ball at site $x$, namely $\eta(x)=1$, then the carrier picks the ball and puts it at the empty seat with the smallest seat number;  
    \item if the site $x$ is empty, namely $\eta(x)=0$, and if there is at least one occupied seat, then the carrier puts down the ball at the occupied seat with the smallest seat number;
    \item otherwise, the carrier just {passes} through. 
\end{itemize}
The above rule can also be summarized by recording in functions $\mathcal{W}=\left(\mathcal{W}_{k}\right)_{k \in \N}, \ \mathcal{W}_{k} : \Z_{\ge 0} \to \{0,1\}$ whether the seat No. $k$ is occupied at site $x$, i.e., we recursively define $\mathcal{W}$ as $\mathcal{W}_{k}(0) = 0$ for any $k \in \N$ and 
\begin{align}\label{eq:dynamics}
\mathcal{W}_{k}(x) & =\mathcal{W}_{k}(x-1) + \eta(x)(1-\mathcal{W}_{k}(x-1))\prod_{\ell=1}^{k-1}\mathcal{W}_{\ell}(x-1) \\  & -(1-\eta(x))\mathcal{W}_{k}(x-1)\prod_{\ell=1}^{k-1}(1-\mathcal{W}_{\ell}(x-1)).
 \end{align}
Then, $\mathcal{W}_{k}(x)=0$ means that the seat No.$k$ of the carrier passing at site $x$ is empty, while $\mathcal{W}_{k}(x)=1$ means that the seat No. $k$ of the carrier is occupied. We call $\mathcal{W}=(\mathcal{W}_{k})_k$ the carrier with seat numbers.
It is obvious by definition that for the classical carrier $W_{\ell}$ with capacity $\ell \in \N \cup \{\infty\}$,
\begin{align}\label{eq:cap_seat}
W_{\ell}(x)=\sum_{k=1}^{\ell} \mathcal{W}_{k}(x)
\end{align}
holds, and we see that $\mathcal{W}$ is a refinement of $W_{\ell}$. Now, we say that a site $x$ is $(k,\uparrow)$-seat if $\eta(x)=1$ and the ball picked at $x$ sits at the seat No. $k$. In the same way, we say that a site $x$ is $(k,\downarrow)$-seat if $\eta(x)=0$ and the ball seated at No. $k$ is put down at $x$. 
Then by using this notion, we define the seat number configuration $\eta^{\sigma}_{k} \in \Omega$, $k \in \N, \sigma \in \left\{ \uparrow,\downarrow \right\}$, as
        \begin{align}
            \eta^{\uparrow}_{k}(x) &:= \begin{dcases} 1 \ & \text{if} \ x \ \text{is a} \ (k,\uparrow)\text{-seat} \\ 0 \ & \text{otherwise} \end{dcases} \\
            &=\mathbf{1}_{\{\mathcal{W}_{k}(x)> \mathcal{W}_{k}(x-1)\}} \\ 
            &= \eta(x)(1-\mathcal{W}_{k}(x-1))\prod_{\ell=1}^{k-1}\mathcal{W}_{\ell}(x-1) \label{eq:kup_alt}, \\
            \eta^{\downarrow}_{k}(x) &:= \begin{dcases} 1 \ & \text{if} \ x \ \text{is a} \ (k,\downarrow)\text{-seat} \\ 0 \ & \text{otherwise} \end{dcases} \\
            &=\mathbf{1}_{\{\mathcal{W}_{k}(x)< \mathcal{W}_{k}(x-1)\}} \\ 
            &= (1-\eta(x))\mathcal{W}_{k}(x-1)\prod_{\ell=1}^{k-1}(1-\mathcal{W}_{\ell}(x-1)) \label{eq:kdown_alt},
        \end{align}
    where the third equalities in \eqref{eq:kup_alt} and \eqref{eq:kdown_alt} are consequences of \eqref{eq:dynamics}.  For later use, we note that  
        \begin{align}
            \eta^{\uparrow}_{k}(x) - \eta^{\downarrow}_{k}(x) &= \mathbf{1}_{\{\mathcal{W}_{k}(x)> \mathcal{W}_{k}(x-1)\}} - \mathbf{1}_{\{\mathcal{W}_{k}(x)< \mathcal{W}_{k}(x-1)\}} \\
            &= \mathcal{W}_{k}(x) - \mathcal{W}_{k}(x-1),
        \end{align}
    and thus we obtain  
        \begin{align}\label{eq:lem:1_1}
            \mathcal{W}_{k}(x) = \sum_{y=1}^{x}(\eta^{\uparrow}_{k}(y) -\eta^{\downarrow}_{k}(y)).
        \end{align}
Observe that there is at most one ball getting in and out at each site. Hence, if a ball gets in or out at site $x$, i.e., $W_{\infty}(x - 1) \neq W_{\infty}(x)$, then the seat number of $x$, that is the $(k,\sigma)$ satisfying $\eta^{\sigma}_k(x)=1$, is uniquely determined. On the other hand, if site $x$ is empty and any seat is vacant at $x-1$, i.e., $W_{\infty}(x - 1) = W_{\infty}(x) = 0$, then we call such $x$ a {\it record}, following \cite{FNRW}. {We note that the operator $T : \Omega \to \Omega$ can be regarded as a flip operator of $0$s and $1$s except for records, i.e., 
    \begin{align}\label{eq:rec_T}
        T\eta(x) 
        = \begin{dcases}
            1 - \eta(x) \ & \ \text{if } x \text{ is not a record}, \\
            \eta(x) \ & \ \text{otherwise}.
        \end{dcases}
    \end{align}}For later use, we define $r(x) \in \{0,1 \}$ as the function such that $r(x) = 1$ if and only if $x$ is a record. 
From the above observations, we see that any site of given ball configuration can be distinguished either as a $(k,\sigma)$-seat for some $k, \sigma$ or a record. In {formulas}, we have 
    \begin{align}
        \sum_{k \in \N} \eta^{\uparrow}_{k}(x) = \eta(x), \qquad
        r(x) + \sum_{k \in \N} \eta^{\downarrow}_{k}(x) = 1 - \eta(x)
    \end{align}
for any $x$. Hence, it is obvious that $r(x)$ is given by
    \begin{align}
        r(x) = 1- \sum_{k \in \N}(\eta^{\uparrow}_{k}(x)+ \eta^{\downarrow}_{k}(x)) \label{eq:record}.
    \end{align} 
Figure \ref{fig:seatnumber_conf} shows the values of $ \mathcal{W}_{k}(\cdot)$ and the seat number configuration for the ball configuration $\eta = 11001110110001100\dots$. Note that the same specific ball configuration will be repeatedly used throughout this paper to facilitate comparison of multiple methods.

\begin{figure}[H]
\footnotesize
\setlength{\tabcolsep}{5pt}
\renewcommand{\arraystretch}{2}
\begin{center}
\begin{tabular}{rcccccccccccccccccccc}
    $x$ & 0 & 1 &  2 &  3 &  4 &  5 &  6 & 7 &  8 & 9 & 10 & 11 & 12 & 13 & 14 & 15 & 16 & 17 & 18 & 19 \\
 \hline\hline $\eta(x)$ & & 1 & 1& 0 & 0 & 1 & 1 & 1 & 0 & 1 & 1 & 0 & 0 & 0 & 1 & 1 & 0 & 0 & 0 & 0 \\
   \hline
   \hline
   $W_{4}(x)$ & 0 & 1 & 2 & 1 & 0 & 1 & 2 & 3 & 2 & 3 & 4 & 3 & 2 & 1 & 2 & 3 & 2 & 1 & 0 & 0\\
   \hline
    $\mathcal{W}_{1}(x)$ & 0 & 1 & 1 & 0 & 0 & 1 & 1 & 1 & 0 & 1 & 1 & 0 & 0 & 0 & 1 & 1 & 0 & 0 & 0 & 0\\
    \hline
    $\eta^{\uparrow}_{1}(x)$ & & 1 & 0 & 0 & 0 & 1 & 0 & 0 & 0 & 1 & 0 & 0 & 0 & 0 & 0 & 0 & 0 & 0 & 0& 0\\
    \hline
    $\eta^{\downarrow}_{1}(x)$ & & 0 & 0& 1 & 0 & 0 & 0 & 0 & 1 & 0 & 0 & 1 & 0 & 0 & 0 & 0 & 1 & 0 & 0& 0\\
    \hline
    $\mathcal{W}_{2}(x)$ & 0  & 0 & 1 & 1 & 0 & 0 & 1 & 1 & 1 & 1 & 1 & 1 & 0 & 0 & 0 & 1 & 1 & 0 & 0& 0\\
    \hline
    $\eta^{\uparrow}_{2}(x)$ & & 0 & 1 & 0 & 0 & 0 & 1 & 0 & 0 & 0 & 0 & 0 & 0 & 0 & 0 & 1 & 0 & 0 & 0& 0\\
    \hline
    $\eta^{\downarrow}_{2}(x)$ & & 0 & 0 & 0 & 1 & 0 & 0 & 0 & 0 & 0 & 0 & 0 & 1 & 0 & 0 & 0 & 0 & 1 & 0& 0\\
    \hline
    $\mathcal{W}_{3}(x)$ & 0 &  0 & 0 & 0 & 0 & 0 & 0 & 1 & 1 & 1 & 1 & 1 & 1 & 0 & 0 & 0 & 0 & 0 & 0& 0\\
    \hline
    $\eta^{\uparrow}_{3}(x)$ & &  0 & 0 & 0 & 0 & 0 & 0 & 1 & 0 & 0 & 0 & 0 & 0 & 0 & 0 & 0 & 0 & 0 & 0& 0\\
    \hline
    $\eta^{\downarrow}_{3}(x)$ & &  0 & 0 & 0 & 0 & 0 & 0 & 0 & 0 & 0 & 0 & 0 & 0 & 1 & 0 & 0 & 0 & 0 & 0& 0\\
    \hline
    $\mathcal{W}_{4}(x)$ & 0 &  0 & 0 & 0 & 0 & 0 & 0 & 0 & 0 & 0 & 1 & 1 & 1 & 1 & 1 & 1 & 1 & 1 & 0& 0\\
    \hline
    $\eta^{\uparrow}_{4}(x)$ & &  0 & 0 & 0 & 0 & 0 & 0 & 0 & 0 & 0 & 1 & 0 & 0 & 0 & 0 & 0 & 0 & 0 & 0& 0\\
    \hline
    $\eta^{\downarrow}_{4}(x)$ & &  0 & 0 & 0 & 0 & 0 & 0 & 0 & 0 & 0 & 0 & 0 & 0 & 0 & 0 & 0 & 0 & 0 & 1& 0\\
    \hline
    $r(x)$ & & 0 & 0 & 0 & 0 & 0 & 0 & 0 & 0 & 0 & 0 & 0 & 0 & 0 & 0 & 0 & 0 & 0 & 0 & 1\\
   \end{tabular}
\end{center}
\caption{Seat number configurations and records.}\label{fig:seatnumber_conf}
\end{figure}

Now we observe the relationship between the seat number configuration and the solitons identified by the TS algorithm. As we will see in Section \ref{sec:seat}, for any $\eta \in \Omega_{< \infty}$, 
the total number of $(k,\uparrow)$-seats is the same as that of $(k,\downarrow)$-seats for each $k \in \N$, where $\Omega_{< \infty} \subset \Omega$ is the set of all finite ball configurations
\begin{align}
        \Omega_{< \infty} := \left\{ \eta \in \Omega \ ; \ \sum_{x \in \N} \eta(x) < \infty \right\}.
\end{align} 
Moreover, the total number of $(k,\sigma)$-seats is conserved in time for each $k \in \N$ and $\sigma \in \{\uparrow, \downarrow\}$, that is, $\sum_{x \in \N
} \eta^{\uparrow}_{k}(x) = \sum_{x \in \N} \eta^{\downarrow}_{k}(x) = \sum_{x \in \N} T\eta^{\uparrow}_{k}(x)$. This relation will be established below in \eqref{eq:flip}, see also Remark \ref{rem:conserv}. 
On the other hand, for a configuration where all entries are $``0"$ except for $k$ consecutive ``$1$"'s, i.e. there is only one soliton and its size is $k$, we can easily see that such $k$-soliton is composed by one of each $(\ell,\sigma)$-seat for $ 1 \le \ell \le k$ and $\sigma \in \left\{ \uparrow, \downarrow \right\}$, see Figure \ref{fig:onlyone} for example. For any configuration in $\Omega_{< \infty}$, we will show that such {a} relation between the seat number configuration and solitons is also valid, that is, any $k$-soliton is composed by one of each $(\ell,\sigma)$-seat for $ 1 \le \ell \le k$ and $\sigma \in \left\{ \uparrow, \downarrow \right\}$, see Proposition \ref{prop:seat_slot} and Section \ref{sec:slot} for details. Hence, for any $\eta \in \Omega_{<\infty}$ we have the formula
\begin{align}
    \left| \left\{ k  \text{-solitons in }  \eta  \right\} \right| = \sum_{x \in \N} \left( \eta^{\uparrow}_{k}(x) - \eta^{\uparrow}_{k+1}(x) \right) = \sum_{x \in \N} \left( \eta^{\downarrow}_{k}(x) - \eta^{\downarrow}_{k+1}(x) \right).
\end{align}
In addition, if $x$ is a record, then by following the TS algorithm, all solitons in $\left( \eta(y)\right)_{1 \le y \le x}$ can be identified independently of $\left( \eta(y)\right)_{y \ge x + 1}$, and we claim that the following equation
    \begin{align}\label{eq:soliton=seat}
        &\left| \left\{ k  \text{-solitons in }  \left( \eta(y)\right)_{1 \le y \le x} \right\} \right| = \sum_{y = 1}^{x} \left( \eta^{\uparrow}_{k}(y) - \eta^{\uparrow}_{k+1}(y) \right) = \sum_{y = 1}^{x} \left( \eta^{\downarrow}_{k}(y) - \eta^{\downarrow}_{k+1}(y) \right)
    \end{align}
holds for any $k \in \N$, while for general $x \in \N$ \eqref{eq:soliton=seat} may not hold. Since any element of $\eta \in \Omega_{< \infty}$ consists of records except for a finite number of sites,  $\eta^{\sigma}_{k}(\cdot) - \eta^{\sigma}_{k+1}(\cdot)$ can be considered as the local density of $k$-solitons for each $\sigma \in \left\{ \uparrow, \downarrow \right\}$. Note that
the above claim will be justified by Proposition \ref{prop:seat_slot}.

\begin{figure}[H]
\footnotesize
\setlength{\tabcolsep}{5pt}
\renewcommand{\arraystretch}{2}
\begin{center}
\begin{tabular}{rccccccccccc}
 $\eta(x)$ & \dots & 1 & 1 & 1 & 1 & 0 & 0 & 0 & 0 &  \dots  \\
 \hline 
 $\eta^{\uparrow}_{1}(x)$ & \dots & 1 & 0 & 0 & 0 & 0 & 0 & 0 & 0 &  \dots  \\
 \hline
 $\eta^{\downarrow}_{1}(x)$ & \dots & 0 & 0 & 0 & 0 & 1 & 0 & 0 & 0 &  \dots  \\
 \hline
$\eta^{\uparrow}_{2}(x)$ & \dots & 0 & 1 & 0 & 0 & 0 & 0 & 0 & 0 &  \dots  \\
 \hline
 $\eta^{\downarrow}_{2}(x)$ & \dots & 0 & 0 & 0 & 0 & 0 & 1 & 0 & 0 &  \dots  \\
 \hline
 $\eta^{\uparrow}_{3}(x)$ & \dots & 0 & 0 & 1 & 0 & 0 & 0 & 0 & 0 &  \dots  \\
 \hline
 $\eta^{\downarrow}_{3}(x)$ & \dots & 0 & 0 & 0 & 0 & 0 & 0 & 1 & 0 &  \dots  \\
 \hline
 $\eta^{\uparrow}_{4}(x)$ & \dots & 0 & 0 & 0 & 1 & 0 & 0 & 0 & 0 &  \dots  \\
 \hline
 $\eta^{\downarrow}_{4}(x)$ & \dots & 0 & 0 & 0 & 0 & 0 & 0 & 0 & 1 &  \dots  \\
   \end{tabular}
\end{center}
\caption{For the case where only one $4$-soliton is included in $\eta$.}\label{fig:onlyone}
\end{figure}

When we consider {a} general ball configuration $\eta \in \Omega$, the TS algorithm may not work because there can be infinite number of balls, and thus we may not be able to identify solitons in $\eta$. However, since the construction of the seat number configuration is sequential, i.e., the value of $\eta^{\sigma}_{k}(x)$ can be determined by $\left(\eta(y)\right)_{1 \le y \le x}$ for any $k \in \N$ and $\sigma \in \left\{\uparrow, \downarrow\right\}$, we can always define $\eta^{\sigma}_{k}(\cdot)$ for any $\eta \in \Omega$. Therefore, motivated by the above discussion, to study the dynamical behavior of the BBS for general  $\eta \in \Omega$, we define $m^{\sigma}_{k} : \Z_{\ge 0} \to \Z_{\ge 0}$ as $m^{\sigma}_{k}(0) := 0$, and
            \begin{align}
                m^{\sigma}_{k}(x) := \sum_{y = 1}^{x} \left( \eta^{\sigma}_{k}(y) - \eta^{\sigma}_{k+1}(y) \right), \label{eq:updowndiff} 
            \end{align}
for any $k \in \N$, $x \in \N$ and $\sigma \in \left\{\uparrow,\downarrow \right\}$. Note that from \eqref{eq:lem:1_1} and \eqref{eq:updowndiff}, we get
    \begin{align}
        m^{\uparrow}_{k}(x) - m^{\downarrow}_{k}(x) &= \mathcal{W}_{k}(x) - \mathcal{W}_{k+1}(x) \in \left\{-1,0,1 \right\} \label{eq:diff1}, \\
        m^{\sigma}_{k}(x+1) - m^{\sigma}_{k}(x) &= \eta^{\sigma}_{k}(x+1) - \eta^{\sigma}_{k+1}(x+1) \in \left\{-1,0,1 \right\} \label{eq:diff2},
    \end{align}
for any $k \in \N$, $x \in \Z_{\ge 0}$ and $\sigma \in \left\{\uparrow,\downarrow \right\}$.
We then introduce the $j$-th leftmost matching point $\tau_{k}(j)$ as 
            \begin{align}
                \tau_{k}(j) &:= \min\left\{ x \in \Z_{\ge 0} \ ; \ m^{\sigma}_{k}(x) \ge j \ \text{for both} \ \sigma \in \left\{\uparrow,\downarrow \right\} \right\} \\
                &= \min\left\{ x \in \Z_{\ge 0} \ ; \ m^{\uparrow}_{k}(x) = m^{\downarrow}_{k}(x) = j \right\} \label{eq:defoftau},
            \end{align}
for any $k, j \in \N$, where the second equality in \eqref{eq:defoftau} is a consequence of \eqref{eq:diff1} and \eqref{eq:diff2}. In terms of solitons, $\tau_{k}(j)$ is the site where the $j$-th $k$-soliton is identified by the TS algorithm, see Proposition \ref{prop:tau_soliton} for details. For example, in Figure \ref{fig:tauands}, the ball configuration $\eta$ contains one $4$-soliton colored in \textcolor{brown}{brown}, two $2$-solitons colored in \textcolor{red}{red} and \textcolor{green}{green}, and one $1$-soliton colored in \textcolor{blue}{blue}, and one can check that the {right}most component of each soliton $x = 4,9,17,18$ are $\tau_{2}(1), \tau_{1}(1), \tau_{2}(2), \tau_{4}(1)$, respectively.  Indeed, the following proposition justifies the above observation and its proof will be presented in Subsection \ref{proof:match}.

    \begin{proposition}\label{prop:match}
        Suppose that $\eta \in \Omega$ and $\tau_{k}(j) < \infty$ for some $k \in \N$ and $j \in \N$. Then, 
        \begin{align}
            x \ge \tau_{k}(j) \ \text{if and only if} \ m^{\sigma}_{k}(x) \ge j \ \text{for both} \ \sigma \in \left\{\uparrow,\downarrow \right\}.
        \end{align}
    \end{proposition}

\begin{figure}
\footnotesize
\setlength{\tabcolsep}{3pt}
\begin{center}
\renewcommand{\arraystretch}{2}
\begin{tabular}{rcccccccccccccccccccc}
    $x$ &  0 & 1 &  2 &  3 &  4 &  5 &  6 & 7 &  8 & 9 & 10 & 11 & 12 & 13 & 14 & 15 & 16 & 17 & 18 & 19 \\
 \hline\hline $\eta(x)$ &  & \color{red}{1} & \color{red}{1} & \color{red}{0} & \color{red}{0} & \color{brown}{1} & \color{brown}{1} & \color{brown}{1} & \color{blue}{0} & \color{blue}{1} & \color{brown}{1} & \color{brown}{0} & \color{brown}{0} & \color{brown}{0} & \color{green}{1} & \color{green}{1} & \color{green}{0} & \color{green}{0} & \color{brown}{0} & 0 \\
   \hline
   \hline
   $\eta^{\uparrow}_{1}(x)$ &  & 1 & 0 & 0 & 0 & 1 & 0 & 0 & 0 & 1 & 0 & 0 & 0 & 0 & 1 & 0 & 0 & 0 & 0& 0\\
    \hline
    $\eta^{\downarrow}_{1}(x)$ &  & 0 & 0& 1 & 0 & 0 & 0 & 0 & 1 & 0 & 0 & 1 & 0 & 0 & 0 & 0 & 1 & 0 & 0& 0\\
    \hline
    $\eta^{\uparrow}_{2}(x)$ &  & 0 & 1 & 0 & 0 & 0 & 1 & 0 & 0 & 0 & 0 & 0 & 0 & 0 & 0 & 1 & 0 & 0 & 0& 0\\
    \hline
    $\eta^{\downarrow}_{2}(x)$ &  & 0 & 0 & 0 & 1 & 0 & 0 & 0 & 0 & 0 & 0 & 0 & 1 & 0 & 0 & 0 & 0 & 1 & 0& 0\\
    \hline
    $\eta^{\uparrow}_{3}(x)$ &  & 0 & 0 & 0 & 0 & 0 & 0 & 1 & 0 & 0 & 0 & 0 & 0 & 0 & 0 & 0 & 0 & 0 & 0& 0\\
    \hline
    $\eta^{\downarrow}_{3}(x)$ &  & 0 & 0 & 0 & 0 & 0 & 0 & 0 & 0 & 0 & 0 & 0 & 0 & 1 & 0 & 0 & 0 & 0 & 0& 0\\
    \hline
    $\eta^{\uparrow}_{4}(x)$ &  & 0 & 0 & 0 & 0 & 0 & 0 & 0 & 0 & 0 & 1 & 0 & 0 & 0 & 0 & 0 & 0 & 0 & 0& 0\\
    \hline
    $\eta^{\downarrow}_{4}(x)$ &  & 0 & 0 & 0 & 0 & 0 & 0 & 0 & 0 & 0 & 0 & 0 & 0 & 0 & 0 & 0 & 0 & 0 & 1 & 0 \\
    \hline 
    $r(x)$ &  & 0 & 0 & 0 & 0 & 0 & 0 & 0 & 0 & 0 & 0 & 0 & 0 & 0 & 0 & 0 & 0 & 0 & 0 & 1\\
    \hline
    $\xi_{1}(x)$ & 0 & 0 & 1 & 1 & 2 & 2 & 3 & 4 & 4 & 4 & 5 & 5 & 6 & 7 & 7 & 8 & 8 & 9 & 10 & 11 \\
    \hline
    $\xi_{2}(x)$ & 0 & 0 & 0 & 0 & 0 & 0 & 0 & 1 & 1 & 1 & 2 & 2 & 2 & 3 & 3 & 3 & 3 & 3 & 4 & 5\\
    \hline
    $\xi_{3}(x)$ & 0 & 0 & 0 & 0 & 0 & 0 & 0 & 0 & 0 & 0 & 1 & 1 & 1 & 1 & 1 & 1 & 1 & 1 & 2 & 3\\
    \hline
    $\xi_{4}(x)$ & 0 & 0 & 0 & 0 & 0 & 0 & 0 & 0 & 0 & 0 & 0 & 0 & 0 & 0 & 0 & 0 & 0 & 0 & 0 & 1 \\
    \hline
    $m^{\uparrow}_{1}(x)$ & 0 & 1 & 0 & 0 & 0 & 1 & 0 & 0 & 0 & \textcolor{blue}{1} & 1 & 1 & 1 & 1 & 2 & 1 & 1 & 1 & 1 & 1 \\
    \hline
    $m^{\downarrow}_{1}(x)$ & 0 & 0 & 0 & 1 & 0 & 0 & 0 & 0 & 1 & \textcolor{blue}{1} & 1 & 2 & 1 & 1 & 1 & 1 & 2 & 1 & 1 & 1 \\
    \hline
    $m^{\uparrow}_{2}(x)$ & 0 & 0 & 1 & 1 & \textcolor{red}{1} & 1 & 2 & 1 & 1 & 1 & 1 & 1 & 1 & 1 & 1 & 2 & 2 & \textcolor{green}{2} & 2 & 2\\
    \hline
    $m^{\downarrow}_{2}(x)$ & 0 & 0 & 0 & 0 & \textcolor{red}{1} & 1 & 1 & 1 & 1 & 1 & 1 & 1 & 2 & 1 & 1 & 1 & 1 & \textcolor{green}{2} & 2 & 2\\
    \hline
    $m^{\uparrow}_{3}(x)$ & 0 & 0 & 0 & 0 & 0 & 0 & 0 & 1 & 1 & 1 & 0 & 0 & 0 & 0 & 0 & 0 & 0 & 0 & 0 & 0\\
    \hline
    $m^{\downarrow}_{3}(x)$ & 0 & 0 & 0 & 0 & 0 & 0 & 0 & 0 & 0 & 0 & 0 & 0 & 0 & 1 & 1 & 1 & 1 & 1 & 0 & 0\\
    \hline
    $m^{\uparrow}_{4}(x)$ & 0 & 0 & 0 & 0 & 0 & 0 & 0 & 0 & 0 & 0 & 1 & 1 & 1 & 1 & 1 & 1 & 1 & 1 & \textcolor{brown}{1} & 1 \\
    \hline
    $m^{\downarrow}_{4}(x)$ & 0 & 0 & 0 & 0 & 0 & 0 & 0 & 0 & 0 & 0 & 0 & 0 & 0 & 0 & 0 & 0 & 0 & 0 & \textcolor{brown}{1} & 1 \\
   \end{tabular}
\end{center}
\caption{How the value of the functions $m^{\sigma}_{k}$ and $\xi_{k}$ change for the ball configuration $\eta = 11001110110001100\dots$. The solitons identified by the TS algorithm and leftmost matching points of $m^{\sigma}_{k}$ are highlighted in color, and one can see that for each $k \in \N$, the rightmost component of a $k$-soliton is a leftmost matching point of $m^{\sigma}_{k}$, respectively. {We note that the functions $\xi_{k}(x), k \in \N, x \in \Z_{\ge 0}$ are defined immediately after Proposition \ref{prop:match}.} }\label{fig:tauands}
\end{figure}

Now we introduce a way to determine the {\it effective position} of $\tau_{k}(\cdot)$. First, we introduce the functions $\xi_{k}(x)$ counting the total number of $(\ell,\uparrow), (\ell,\downarrow)$-seats satisfying $\ell \ge k + 1$ and records up to $x$ as $\xi_{k}(0) := 0$ and
 \begin{align}
                \xi_{k}(x) 
                &:= \sum_{\ell \ge k + 1} \sum_{y = 1}^{x} \left( \eta^{\uparrow}_{\ell}(y) + \eta^{\downarrow}_{\ell}(y) \right) + \sum_{y = 1}^{x} r(x)\\
                &= x - \sum_{\ell = 1}^{k} \sum_{y = 1}^{x} \left( \eta^{\uparrow}_{\ell}(y) + \eta^{\downarrow}_{\ell}(y) \right)
            \end{align}
for any $k \in \N$, $x \in \Z_{\ge 0}$ and $\sigma \in \left\{\uparrow,\downarrow \right\}$. Figure \ref{fig:tauands} also shows an example of $\xi_{k}\left(\cdot\right)$. Then, the effective position of $\tau_{k}(j)$ is defined as $\xi_{k}\left( \tau_{k}(j) \right)$ for any $j \in \N$. We explain the reason of the definition of the effective position from the viewpoint of solitons. First we recall that each $\tau_{k}(j)$ corresponds to a $k$-soliton as pointed out above. Next, we note that the function $\xi_{k}\left(\cdot\right)$ is a non-decreasing function, but constant on sites included in $\ell$-solitons with $\ell \le k$. {In particular, if $\gamma \subset \N$ is a $k$-soliton, then the rightmost component of $\gamma$ is $\tau_{k}(j)$ for some $j \in \N$, and we have $\xi_{k}\left(x\right) = \xi_{k}\left(\tau_{k}(j)\right)$ for any $x \in \gamma$. Thus, by associating $\tau_{k}\left(j\right)$'s to each $k$-soliton, we can consider $k$-solitons as points via $\xi_{k}\left(\cdot\right)$, and the function $\xi_{k}\left( \tau_{k}( \ \cdot \ ) \right)$ can be regarded as measuring certain distances between $\tau_{k}\left(j\right)$'s ignoring $\ell$-seats with $\ell \le k$ between them.}
 Now, we claim that at each time step, $\xi_{k}\left( \tau_{k}(\cdot) \right)$ will be shifted by $k$, i.e., $T\xi_{k}\left( T\tau_{k}(\cdot) \right) = \xi_{k}\left( \tau_{k}(\cdot) \right) + k$, and in this sense we say that $\xi_{k}(\cdot)$ gives the {\it effective positions of $k$-solitons} considering the effect of the interaction between other solitons. In addition, if there are no {interactions} between solitons, then the effective positions are essentially equivalent to the original positions of the solitons in $\eta$. In Figure \ref{fig:effective position} we give an example of $\xi_{k}\left(\cdot\right)$ for the ball configuration $\eta = 111000010010\dots$, which is the same configuration used in Figure \ref{fig:timeevolution}, and it can be seen that the effective positions are shifted by $k$ at one step time evolution, while the components of the solitons are not linearly changed in time due to the interaction. Hereafter, we will justify the above claims not from the viewpoint of solitons, but rather from the viewpoint of the seat number configuration.

\begin{figure}
\footnotesize
\setlength{\tabcolsep}{4pt}
\begin{center}
\renewcommand{\arraystretch}{2}
\begin{tabular}{rcccccccccccccccccccccccc}
    $x$ & 0 & 1 &  2 &  3 &  4 &  5 &  6 & 7 &  8 & 9 & 10 & 11 & 12 & 13 & 14 & 15 & 16 & 17 & 18 & 19 & 20 & 21 & 22 & $\dots$ \\
    \hline\hline $\eta(x)$ &  & \color{red}{1} & \color{red}{1} & \color{red}{1} & \color{red}{0} & \color{red}{0} & \color{red}{0} & 0 & \color{blue}{1} & \color{blue}{0} & 0 & \color{green}{1} & \color{green}{0} & 0 & 0 & 0 & 0 & 0 & 0 & 0 & 0 & 0 & 0 & $\dots$ \\ 
    \hline
    $\xi_{1}(x)$ & 0 & 0 & 1 & 2 & 2 & 3 & 4 & 5 & 5 & 5 & 6 & 6 & 6 & 7 & 8 & 9 & 10 & 11 & 12 & 13 & 14 & 15 & 16 & $\dots$ \\ 
    \hline
    $\xi_{3}(x)$ & 0 & 0 & 0 & 0 & 0 & 0 & 0 & 1 & 1 & 1 & 2 & 2 & 2 & 3 & 4 & 5 & 6 & 7 & 8 & 9 & 10 & 11 & 12 & $\dots$ \\ 
    \hline
    \\
 \hline 
 $T\eta(x)$ & & 0 & 0 & 0 & \color{red}{1} & \color{red}{1} & \color{red}{1} & \color{red}{0} & \color{red}{0} & \color{blue}{1} & \color{blue}{0} & \color{red}{0} & \color{green}{1} & \color{green}{0} & 0 & 0 & 0 & 0 & 0 & 0 & 0 & 0 & 0 & $\dots$  \\
 \hline
    $T\xi_{1}(x)$ & 0 & 1 & 2 & 3 & 3 & 4 & 5 & 5 & 6 & 6 & 6 & 7 & 7 & 7 & 8 & 9 & 10 & 11 & 12 & 13 & 14 & 15 & 16 & $\dots$ \\ 
    \hline
    $T\xi_{3}(x)$ & 0 & 1 & 2 & 3 & 3 & 3 & 3 & 3 & 3 & 3 & 3 & 3 & 3 & 3 & 4 & 5 & 6 & 7 & 8 & 9 & 10 & 11 & 12 & $\dots$ \\
    \hline \\
 \hline
 $T^2\eta(x)$ & & 0 & 0 & 0 & 0 & 0 & 0 & \color{red}{1} & \color{red}{1} & \color{blue}{0} & \color{blue}{1} & \color{red}{1} & \color{red}{0} & \color{green}{1} & \color{green}{0} & \color{red}{0} & \color{red}{0} & 0 & 0 & 0 & 0 & 0 & 0 & $\dots$  \\
 \hline
    $T^2\xi_{1}(x)$ & 0 & 1 & 2 & 3 & 4 & 5 & 6 & 6 & 7 & 7 & 7 & 8 & 8 & 8 & 8 & 9 & 10 & 11 & 12 & 13 & 14 & 15 & 16 & $\dots$ \\ 
    \hline
    $T^2\xi_{3}(x)$ & 0 & 1 & 2 & 3 & 4 & 5 & 6 & 6 & 6 & 6 & 6 & 6 & 6 & 6 & 6 & 6 & 6 & 7 & 8 & 9 & 10 & 11 & 12 & $\dots$ \\
    \hline \\
 \hline
 $T^3\eta(x)$ & & 0 & 0 & 0 & 0 & 0 & 0 & 0 & 0 & \color{blue}{1} & \color{blue}{0} & 0 & \color{red}{1} & \color{green}{0} & \color{green}{1} & \color{red}{1} & \color{red}{1} & \color{red}{0} & \color{red}{0} & \color{red}{0} & 0 & 0 & 0 & $\dots$  \\
 \hline
    $T^3\xi_{1}(x)$ & 0 & 1 & 2 & 3 & 4 & 5 & 6 & 7 & 8 & 8 & 8 & 9 & 9 & 9 & 9 & 10 & 11 & 11 & 12 & 13 & 14 & 15 & 16 & $\dots$ \\ 
    \hline
    $T^3\xi_{3}(x)$ & 0 & 1 & 2 & 3 & 4 & 5 & 6 & 7 & 8 & 8 & 8 & 9 & 9 & 9 & 9 & 9 & 9 & 9 & 9 & 9 & 10 & 11 & 12 & $\dots$ \\
    \hline \\ 
 \hline
 $T^4\eta(x)$ & & 0 & 0 & 0 & 0 & 0 & 0 & 0 & 0 & 0 & \color{blue}{1} & \color{blue}{0} & 0 & \color{green}{1} & \color{green}{0} & 0 & 0 & \color{red}{1} & \color{red}{1} & \color{red}{1} & \color{red}{0}  & \color{red}{0} & \color{red}{0} & $\dots$ \\
 \hline
    $T^4\xi_{1}(x)$ & 0 & 1 & 2 & 3 & 4 & 5 & 6 & 7 & 8 & 9 & 9 & 9 & 10 & 10 & 10 & 11 & 12 & 12 & 13 & 14 & 14  & 15 & 16 & $\dots$ \\ 
    \hline
    $T^4\xi_{3}(x)$ & 0 & 1 & 2 & 3 & 4 & 5 & 6 & 7 & 8 & 9 & 9 & 9 & 10 & 10 & 10 & 11 & 12 & 12 & 12 & 12 & 12 & 12 & 12 & $\dots$ \\
   \end{tabular}
\end{center}
\caption{At each time step, the effective positions of $k$-solitons are shifted by $k$.}\label{fig:effective position}
\end{figure}

From now on, we return to the viewpoint of the seat number configuration. We introduce $\zeta_{k}(i)$ as the total number of $\tau_{k}$'s located at effective position $i$ in the above sense, i.e., 
    \begin{align}\label{def:zeta}
                \zeta_{k}(i) &:= \left| \left\{ j \in \N \ ; \ \tau_{k}(j) \in \left\{ x \in \Z_{\ge 0} \ ; \ \xi_{k}(x) = i \right\} \right\} \right| .   \end{align}
For the ball configuration in Figure \ref{fig:tauands}, $\zeta$ is given by
    \begin{align}
        \zeta_{k}(i) = \begin{dcases}
                            1 \quad & (k,i) = (4,0), (2,0), (2,3), (1,4), \\
                            0 \quad & \text{otherwise.}
        \end{dcases}
    \end{align}

Now we present one of our main theorems, which claims that the effective position of $\tau_{k}(\cdot)$ is shifted by $k$ in one step time evolution, i.e., the dynamics of the BBS is linearized in terms of $\zeta$.  Before describing the statement, we will give an example. Observe that Figure \ref{fig:seat_linear} shows the seat numbers of $T\eta$, where $\eta$ is the same configuration as in Figure \ref{fig:seatnumber_conf} and \ref{fig:tauands}. From the figure, we see that 
    \begin{align}
        T\zeta_{k}(i) = \begin{dcases}
                            1 \quad & (k,i) = (4,4), (2,2), (2,5), (1,5), \\
                            0 \quad & \text{otherwise.}
        \end{dcases}
    \end{align}
In particular, we have
    \begin{align}
        T\zeta_{k}(i) = \zeta_{k}(i - k),
    \end{align}
for any $k \in \N$ and $i \in \Z_{\ge 0}$, with convention $\zeta_{k}(i) = 0$ for $i < 0$. Our claim is that this relationship holds true for general configurations as well.
\begin{figure}
\footnotesize
\setlength{\tabcolsep}{3pt}
\renewcommand{\arraystretch}{2}
\begin{center}
\begin{tabular}{rcccccccccccccccccccccccc}
    $x$ &  0 & 1 &  2 &  3 &  4 &  5 &  6 & 7 &  8 & 9 & 10 & 11 & 12 & 13 & 14 & 15 & 16 & 17 & 18 & 19 & 20 & 21 & 22 & 23 \\
    \hline\hline $\eta(x)$ &  & \color{red}{1} & \color{red}{1} & \color{red}{0} & \color{red}{0} & \color{brown}{1} & \color{brown}{1} & \color{brown}{1} & \color{blue}{0} & \color{blue}{1} & \color{brown}{1} & \color{brown}{0} & \color{brown}{0} & \color{brown}{0} & \color{green}{1} & \color{green}{1} & \color{green}{0} & \color{green}{0} & \color{brown}{0} & 0 & 0 & 0 & 0 & 0 \\
 \hline $T\eta(x)$ &  & 0 & 0& \color{red}{1} & \color{red}{1} & \color{red}{0} & \color{red}{0} & 0 & \color{blue}{1} & \color{blue}{0} & 0 & \color{brown}{1} & \color{brown}{1} & \color{brown}{1} & \color{green}{0} & \color{green}{0} & \color{green}{1} & \color{green}{1} & \color{brown}{1} & \color{brown}{0} & \color{brown}{0} & \color{brown}{0} & \color{brown}{0} & 0  \\
   \hline
   \hline
    $T\xi_{1}(x)$ & 0 & 1 & 2 & 2 & 3 & 3 & 4 & 5 & 5 & 5 & 6 & 6 & 7 & 8 & 8 & 9 & 9 & 10 & 11 & 11 & 12 & 13 & 14 & 15 \\
    \hline
    $T\xi_{2}(x)$ & 0 & 1 & 2 & 2 & 2 & 2 & 2 & 3 & 3 & 3 & 4 & 4 & 4 & 5 & 5 & 5 & 5 & 5 & 6 & 6 & 6 & 7 & 8 & 9 \\
    \hline
    $T\xi_{3}(x)$ & 0 & 1 & 2 & 2 & 2 & 2 & 2 & 3 & 3 & 3 & 4 & 4 & 4 & 4 & 4 & 4 & 4 & 4 & 5 & 5 & 5 & 5 & 6 & 7 \\
    \hline
    $T\xi_{4}(x)$ & 0 & 1 & 2 & 2 & 2 & 2 & 2 & 3 & 3 & 3 & 4 & 4 & 4 & 4 & 4 & 4 & 4 & 4 & 4 & 4 & 4 & 4 & 4 & 5 \\
    \hline
    $Tm^{\uparrow}_{1}(x)$ & 0 & 0 & 0 & 1 & 0 & 0 & 0 & 0 & 1 & \textcolor{blue}{1} & 1 & 2 & 1 & 1 & 1 & 1 & 2 & 1 & 1 & 1 & 1 & 1 & 1 & 1 \\
    \hline
    $Tm^{\downarrow}_{1}(x)$ & 0 & 0 & 0 & 0 & 0 & 1 & 0 & 0 & 0 & \textcolor{blue}{1} & 1 & 1 & 1 & 1 & 2 & 1 & 1 & 1 & 1 & 2 & 1 & 1 & 1 & 1 \\
    \hline
    $Tm^{\uparrow}_{2}(x)$ & 0 & 0 & 0 & 0 & 1 & 1 & \textcolor{red}{1} & 1 & 1 & 1 & 1 & 1 & 2 & 1 & 1 & 1 & 1 & \textcolor{green}{2} & 2 & 2 & 2 & 2 & 2 & 2 \\
    \hline
    $Tm^{\downarrow}_{2}(x)$ & 0 & 0 & 0 & 0 & 0 & 0 & \textcolor{red}{1} & 1 & 1 & 1 & 1 & 1 & 1 & 1 & 1 & 2 & 2 & \textcolor{green}{2} & 2 & 2 & 3 & 2 & 2 & 2 \\
    \hline
    $Tm^{\uparrow}_{3}(x)$ & 0 & 0 & 0 & 0 & 0 & 0 & 0 & 0 & 0 & 0 & 0 & 0 & 0 & 1 & 1 & 1 & 1 & 1 & 0 & 0 & 0 & 0 & 0 & 0 \\
    \hline
    $Tm^{\downarrow}_{3}(x)$ & 0 & 0 & 0 & 0 & 0 & 0 & 0 & 0 & 0 & 0 & 0 & 0 & 0 & 0 & 0 & 0 & 0 & 0 & 0 & 0 & 0 & 1 & 0 & 0 \\
    \hline
    $Tm^{\uparrow}_{4}(x)$ & 0 & 0 & 0 & 0 & 0 & 0 & 0 & 0 & 0 & 0 & 0 & 0 & 0 & 0 & 0 & 0 & 0 & 0 & 1 & 1 & 1 & 1 & \textcolor{brown}{1} & 1 \\
    \hline
    $Tm^{\downarrow}_{4}(x)$ & 0 & 0 & 0 & 0 & 0 & 0 & 0 & 0 & 0 & 0 & 0 & 0 & 0 & 0 & 0 & 0 & 0 & 0 & 0 & 0 & 0 & 0 & \textcolor{brown}{1} & 1 \\
   \end{tabular}
\end{center}
\caption{Linearization property of the $(k,\sigma)$-seats.}\label{fig:seat_linear}
\end{figure}

\begin{theorem}\label{thm:seat_ISM}
Suppose that $\eta \in \Omega$ and $0 < \left|\left\{ x \in \Z_{\ge 0} \ ; \  \xi_{k}(x) = i \right\}\right| < \infty$ for some $k \in \N$ and $i \in \Z_{\ge 0}$. Then we have $0 < \left|\left\{ x \in \Z_{\ge 0} \ ; \  T\xi_{k}(x) = i \right\}\right| < \infty$ and
\begin{align}\label{eq:seat_ISM}
(T\zeta)_{k}(i)= \zeta_{k}(i-k),
\end{align}
with convention that $\zeta_{k}(i) =0$ for any $i <0$.
\end{theorem}
{In particular, since the function $\xi_{k}( \cdot )$ strictly increases at each record for any $k \in \N$, we have the following corollary of Theorem \ref{thm:seat_ISM}. 
\begin{corollary}
    Suppose that $\eta \in \Omega_{<\infty}$. For any $k \in \N$ and $i \in \Z_{\ge 0}$, we have 
        \begin{align}
(T\zeta)_{k}(i)= \zeta_{k}(i-k),
\end{align}
with convention that $\zeta_{k}(i) =0$ for any $i <0$.
\end{corollary}}

We give the proof of {Theorem \ref{thm:seat_ISM}} in Section \ref{sec:seat}. We emphasize that the proof is self-contained and none of the relations with other linearization methods are used, and the definitions of $\eta^{\sigma}_{k}, m^{\sigma}_{k}, \tau_{k}, \xi_{k}, \zeta_{k}$ are independent of the notion of solitons.
Note that under slightly stronger assumptions than that of Theorem \ref{thm:seat_ISM}, one can reconstruct $\eta$ from $\zeta$ via the relation to the slot decomposition, see Remark \ref{rem:bijective} and \cite[Section 2.2]{CS} for a constructive proof of this claim. We also note that, in fact, their reconstruction algorithm only depends on the value of $\zeta$ and does not require the notion of solitons. Hence, the above results mean that the seat number configuration gives a new linearization method for the BBS. 

\begin{remark}
    From the relation with the rigged configuration obtained by KKR bijection shown in Section \ref{sec:KKR}, under the same assumption as Theorem \ref{thm:seat_ISM}, the above theorem can be generalized to the BBS with capacity $\ell$ as 
    \[
(T_{\ell}\zeta)_{k}(i)= \zeta_{k}(i-(k\wedge \ell)).
\]
We believe that there should be a direct proof of this linearization without using the relation with KKR bijection, but do not pursue it in this paper. 
\end{remark}

\begin{remark}
{The following is one example of a configuration that does not satisfy the assumption of Theorem \ref{thm:seat_ISM} for some $k$ and $i$.
\begin{align}
    \eta\left(x\right) = 
    \begin{dcases}
        1 \ & \ x = 1,2,3,7,8,11, 4n + 1, 4n + 2  \text{ for some } n \ge 3, \\ 
        0 \ & \ \text{otherwise}.
    \end{dcases}
\end{align}
In other words, $\eta = 111000110010(1100)^{\otimes \infty}$, see also Figure \ref{ex:not_satisfy}. In this example, we see that $0 < \left| \left\{ x \in \Z_{\ge 0} \ ; \ \xi_{1}\left(x\right) = i \right\} \right| < \infty$ for any $i \in \Z_{\ge 0}$, and $T\zeta_{1}\left(i\right) = \zeta_{1}\left(i - 1\right)$. However, for $k = 2, 3$, we have
    \begin{align}
        \left| \left\{ x \in \Z_{\ge 0} \ ; \ \xi_{2}\left(x\right) = i \right\} \right| &= 
        \begin{dcases}
            3 \ & \ i = 0,1, \\
            \infty \ & \ i = 2, \\
            0 \ & \ i \ge 3,
        \end{dcases} \\
        \left| \left\{ x \in \Z_{\ge 0} \ ; \ \xi_{3}\left(x\right) = i \right\} \right| &= 
        \begin{dcases}
            \infty \ & \ i = 0, \\
            0 \ & \ i \ge 1.
        \end{dcases}
    \end{align}
On the other hand, we see that $T\tau_{3}(1) = \infty$, and thus we get 
    \begin{align}
        T\zeta_{3}\left(3\right) = 0 \neq \zeta_{3}\left(0\right) = 1.
    \end{align}
From the soliton viewpoint, this is a situation where a 3-soliton overtakes an infinite number of 2-solitons, and thus escapes to infinity at once. 
We note that by direct computation, 
    \begin{align}
        \tau_{2}(j) = T\tau_{2}(j) = 
        \begin{dcases}
            10 \ & \ j = 1, \\
            16 + 4(j - 1) \ & \ j \ge 2,
        \end{dcases}
    \end{align}
and thus $\zeta_{2}\left(2\right) = \infty$ and $T\zeta_{2}\left(4\right) = \infty$ for this $\eta$. Hence \eqref{eq:seat_ISM} formally holds for $k = 2$.} We also note that this configuration also violates the condition described in Remark \ref{rem:infinite}, discussed later in Section \ref{subsec:relation}.

\end{remark}

\begin{figure}
    \footnotesize
    \setlength{\tabcolsep}{3pt}
    \renewcommand{\arraystretch}{2}
    \centering
    \begin{tabular}{rcccccccccccccccccccccc}
    $x$ & 0 & 1 &  2 &  3 &  4 &  5 &  6 & 7 &  8 & 9 & 10 & 11 & 12 & 13 & 14 & 15 & 16 & 17 & 18 & 19 & 20 &  $\dots$ \\
    \hline\hline $\eta(x)$ &  & 1 & 1 & 1 & 0 & 0 & 0 & 1 & 1 & 0 & 0 & 1 & 0 & 1 & 1 & 0 & 0 & 1 & 1 & 0 & 0 & $\dots$ \\ 
    \hline  
    $T\eta(x)$ &  & 0 & 0 & 0 & 1 & 1 & 1 & 0 & 0 & 1 & 1 & 0 & 1 & 0 & 0 & 1 & 1 & 0 & 0 & 1 & 1 & $\dots$ \\ 
    \hline \hline
    $\xi_{1}(x)$ & 0 & 0 & 1 & 2 & 2 & 3 & 4 & 4 & 5 & 5 & 6 & 6 & 6 & 6 & 7 & 7 & 8 & 8& 9 & 9 &10& $\dots$ \\ 
    \hline
    $\xi_{2}(x)$ & 0 & 0 & 0 & 1 & 1 & 1 & 2 & 2 & 2 & 2 & 2 & 2 & 2 & 2 & 2 & 2 & 2 & 2 & 2 & 2 & 2 & $\dots$ \\ 
    \hline
    $\xi_{3}(x)$ & 0 & 0 & 0 & 0 & 0 & 0 & 0 & 0 & 0 & 0 & 0 & 0 & 0 & 0 & 0 & 0 & 0 & 0 & 0 & 0 & 0 &$\dots$ \\ 
    \hline
    $T\xi_{1}(x)$ & 0 & 1 & 2 & 3 & 3 & 4 & 5 & 5 & 6 & 6 & 7 & 7 & 7 & 7 & 8 & 8& 9 & 9 &10 & 10 & 11 & $\dots$ \\ 
    \hline
    $T\xi_{2}(x)$ & 0 & 1 & 2 & 3 & 3 & 3 & 4 & 4 & 4 & 4 & 4 & 4 & 4 & 4 & 4 & 4 & 4 & 4 & 4 & 4 & 4 & $\dots$ \\ 
    \hline
    $T\xi_{3}(x)$ & 0 & 1 & 2 & 3 & 3 & 3 & 3 & 3 & 3 & 3 & 3 & 3 & 3 & 3 & 3 & 3 & 3 & 3 & 3 & 3 & 3 &$\dots$ \\
    \hline
    $m^{\uparrow}_{1}(x)$ & 0 & 1 & 0 & 0 & 0 & 0 & 0 & 1 & 0 & 0 & 0 & 1 & \textcolor{red}{1} & 2 & 1 & 1 & 1 & 2 & 1 & 1 & 1 & \dots \\
    \hline
    $m^{\downarrow}_{1}(x)$ & 0 & 0 & 0 & 0 & 1 & 0 & 0 & 0 & 0 & 1 & 0 & 0 & \textcolor{red}{1} & 1 & 1 & 2 & 1 & 1 & 1 & 2 & 1 & \dots \\
    \hline
    $m^{\uparrow}_{2}(x)$ & 0 & 0 & 1 & 0 & 0 & 0 & 0 & 0 & 1 & 1 & \textcolor{red}{1} & 1 & 1 & 1 & 2 & 2 & \textcolor{red}{2} & 2 & 3 & 3 & \textcolor{red}{3} & \dots \\
    \hline
    $m^{\downarrow}_{2}(x)$ & 0 & 0 & 0 & 0 & 0 & 1 & 0 & 0 & 0 & 0 & \textcolor{red}{1} & 1 & 1 & 1 & 1 & 1 & \textcolor{red}{2} & 2 & 2 & 2 & \textcolor{red}{3} &  \dots \\
    \hline
    $m^{\uparrow}_{3}(x)$ & 0 & 0 & 0 & 1 & 1 & 1 & \textcolor{red}{1} & 1 & 1 & 1 & 1 & 1 & 1 & 1 & 1 & 1 & 1 & 1 & 1 & 1 & 1 & \dots \\
    \hline
    $m^{\downarrow}_{3}(x)$ & 0 & 0 & 0 & 0 & 0 & 0 & \textcolor{red}{1} & 1 & 1 & 1 & 1 & 1 & 1 & 1 & 1 & 1 & 1 & 1 & 1 & 1 & 1 & \dots \\
    \hline
    $Tm^{\uparrow}_{1}(x)$ & 0 & 0 & 0 & 0 & 1 & 0 & 0 & 0 & 0 & 1 & 0 & 0 & \textcolor{red}{1} & 1 & 1 & 2 & 1 & 1 & 1 & 2 & 1 & \dots \\
    \hline
    $Tm^{\downarrow}_{1}(x)$ & 0 & 0 & 0 & 0 & 0 & 0 & 0 & 1 & 0 & 0 & 0 & 1 & \textcolor{red}{1} & 2 & 1 & 1 & 1 & 2 & 1 & 1 & 1 & \dots \\
    \hline
    $Tm^{\uparrow}_{2}(x)$ & 0 & 0 & 0 & 0 & 0 & 1 & 0 & 0 & 0 & 0 & \textcolor{red}{1} & 1 & 1 & 1 & 1 & 1 & \textcolor{red}{2} & 2 & 2 & 2 & \textcolor{red}{3} &  \dots \\
    \hline
    $Tm^{\downarrow}_{2}(x)$ & 0 & 0 & 0 & 0 & 0 & 0 & 0 & 0 & 1 & 1 & \textcolor{red}{1} & 1 & 1 & 1 & 2 & 2 & \textcolor{red}{2} & 2 & 3 & 3 & \textcolor{red}{3} &  \dots \\
    \hline
    $Tm^{\uparrow}_{3}(x)$ & 0 & 0 & 0 & 0 & 0 & 0 & 1 & 1 & 1 & 1 & 1 & 1 & 1 & 1 & 1 & 1 & 1 & 1 & 1 & 1 & 1 & \dots \\
    \hline
    $Tm^{\downarrow}_{3}(x)$ & 0 & 0 & 0 & 0 & 0 & 0 & 0 & 0 & 0 & 0 & 0 & 0 & 0 & 0 & 0 & 0 & 0 & 0 & 0 & 0 & 0 & \dots \\
\end{tabular}
\caption{An example of $\eta$ which does not satisfy the assumption of Theorem \ref{thm:seat_ISM} for $k \ge 2$. $\tau_{k}(j)$'s and $T\tau_{k}(j)$'s are colored by red.}\label{ex:not_satisfy}
\end{figure}

\subsection{Relationships between various linearizations}\label{subsec:relation}

The seat number configuration has a strong advantage that its relation to known linearization methods is clear, and hence it reveals equivalences between them. In Section \ref{sec:KKR}, we will see the relation to the KKR bijection, which gives a {\it sequential} construction of growing sequences of pairs of partitions and {\it riggings}, called {\it rigged configurations}, from given ball configurations, {where a rigging of a partition $\mu$ is a collection of integers assigned to each element of $\mu$.}
The term {\it sequential} means that starting from $\eta$, we will sequentially construct rigged configurations $(\mu(x),\mathbf{J}(x))$ for $x=0,1,\dots$ in such a way that $(\mu(x),\mathbf{J}(x))$ is a function of $\eta(1),\dots,\eta(x)$ and $(\mu,\mathbf{J}) = \lim_{x\to \infty}(\mu(x),\mathbf{J}(x))$ will be the rigged configuration associated {with} $\eta$. For example, the rigged configuration corresponding to the ball configuration $\left( \eta(x) \right)_{1 \le x \le 19}, \  \eta = 1100111011000110000\dots$ is given by { the partition $\mu = (4,2,2,1)$ and the rigging $\mathbf{J} = ( J_{1}, J_{2}, J_{4} ), J_{1} = (3), J_{2} = (-2, 1), J_{4} = (-4)$, where $J_{k}$ is the sequence of integers assigned to $k$'s in the partition $\mu$ ordered from the smallest to largest,} and it can be represented as follows.
\begin{equation}
	    \begin{tikzpicture}
	           \node[left] at (0,0) {$\yng(4,2,2,1)$};
	           \node[right] at (.1,.6) {$-4$};
	           \node[right] at (.2,.15) {$1$};
	           \node[right] at (.1,-.3) {$-2$};
	           \node[right] at (.2,-.75) {$3.$};
	    \end{tikzpicture}
	\end{equation} 
See also Figure \ref{fig:ex_rig} to see how $(\mu(x),\mathbf{J}(x))$ grows as $x$ changes. 

To state the claim, let $\mu(x)=(\mu_i(x))$ be the partitions obtained by the KKR bijection and $\lambda(x)=(\lambda_{k}(x))$ be the conjugate of $\mu(x)$, { i.e., $\lambda_{k}(x) := \left| \left\{ i \in \N \ ; \ \mu_{i}(x) \ge k \right\} \right|$}.
Also let $m_{k}(x) :=\lambda_{k}(x) - \lambda_{k+1}(x)$ be the multiplicity of $k$ in $\mu(x)$, $\mathbf{J}(x)=(J_k(x), k \in \N)$ be the rigging, and $p_{k}(x) := x - 2E_{k}(x)$ be the vacancy where $E_{k}(x)$ {is} the $k$-{th} energy {defined as \begin{align}
	        E_{k}(x) := \sum_{i \in \N} \min \left\{\mu_i(x),k \right\},
	    \end{align}  for any $k \in \N$ and $x \in \Z_{\ge 0}$. We note that the number of components in $J_{k}(x)$ is equal to $m_{k}(x)$.}
We also recall that the seat number configuration $\eta^{\sigma}_{k}$ and the function $m^{\sigma}_{k}$ are defined in \eqref{eq:kup_alt},\eqref{eq:kdown_alt} and \eqref{eq:updowndiff} for $k \in \N$ and $\sigma \in \left\{ \uparrow,\downarrow \right\}$. Then, we have the following relation between the quantities from the KKR bijection and those from the seat number configuration. 

\begin{proposition}[Seat-KKR]\label{prop:seat_KKR}
Suppose that $\eta \in \Omega$. For any $k \in \N$ and $x \in \Z_{\ge 0}$, we have
\begin{align}
\lambda_{k}(x) &= \sum_{y=1}^x\eta^{\uparrow}_{k}(y),  \\
p_{k}(x) &= x-2\sum_{\ell=1}^{k}\sum_{y=1}^x\eta^{\uparrow}_{\ell}(y), \label{eq:seat_energy}
\end{align}
and
\begin{align}
J_{k}(x) & =(J_{k,j}(x)=p_{k}(t_{k}(x,j)) \ ; \ j=1,\dots, m^{\uparrow}_{k}(x) ),
\end{align}
where 
\begin{align}
t_{k}(x,j):=\max\{1 \le y \le x \ ;  \ m^{\uparrow}_{k}(y)=j,\  \eta^{\uparrow}_{k}(y)=1\}.
\end{align}
\end{proposition}
In particular, from \eqref{eq:seat_energy} we see that $\sum_{\ell = 1}^{k}\eta^{\uparrow}_{\ell}(x)$ is the local {$k$-th} energy at $x \in \N$. Obviously, there is a direct relationship between $(k,\uparrow)$-seats and the {\it local energy function $H$} used in the crystal theory formulation of {the} BBS \cite{FOY}, see Remark \ref{rem:localenergy} for details. We emphasize that  since the rigged configuration is constructed sequentially, $\left(\mu(x), \mathbf{J}(x)\right)$ can always be defined for any $x \in \Z_{\ge 0}$ and Proposition \ref{prop:seat_KKR} is valid for any $\eta \in \Omega$, which is not necessarily in $\Omega_{<\infty}$.

In Section \ref{sec:slot}, we will establish the relation between the seat number configuration and the slot configuration. Compared to the KKR bijection, the slot configuration {defined via the algorithm in \cite{FNRW}}, denoted by $\nu(x)$, needs a {\it parallel} construction, that is, to define the value of $\nu(x)$, we need the entire information of $(\eta(y))_{y \in \N}$ (or at least $(\eta(y))_{y \in [1, x']}$ for some $x' >x$ in general). However, in this paper, we will prove that slot configuration can be constructed {\it sequentially}. In particular, we show that $\left(\eta^{\sigma}_{k}(x)\right)_{\sigma\in\{ \uparrow, \downarrow\}, k \in \N, x \in \N}$ can be considered as a sequential construction of the slot configuration. To describe the statement, let $\tilde{\xi}_{k}(x)$ be the number of $k$-slots in $[1,x]$, and $(\tilde{\zeta}_{k})_{k \in \N}=(\tilde{\zeta}_{k}(i))_{k \in \N, i \in \Z_{\ge 0}}$ be the slot configuration. For example, the slot configuration corresponding to the ball configuration $\eta = 1100111011000110000\dots$ is given as follows.
\begin{figure}[H]
        \footnotesize
        \setlength{\tabcolsep}{5pt}
        \begin{center}
        \renewcommand{\arraystretch}{2}
        \begin{tabular}{rccccccccccccccccccc}
            $x$ & 1 &  2 &  3 &  4 &  5 &  6 & 7 &  8 & 9 & 10 & 11 & 12 & 13 & 14 & 15 & 16 & 17 & 18 & 19 \\
         \hline\hline $\eta(x)$ & 1 & 1 & 0 & 0 & 1 & 1 & 1 & 0 & 1 & 1 & 0 & 0 & 0 & 1 & 1 & 0 & 0 & 0 & 0  \\
           \hline
            $\nu(x)$  & 0 & 1 & 0 & 1 & 0 & 1 & 2 & 0 & 0 & 3 & 0 & 1 & 2 & 0 & 1 & 0 & 1 & 3 & $\infty$ \\
            \hline
   \hline
   $\eta^{\uparrow}_{1}(x)$ &  1 & 0 & 0 & 0 & 1 & 0 & 0 & 0 & 1 & 0 & 0 & 0 & 0 & 1 & 0 & 0 & 0 & 0& 0\\
    \hline
    $\eta^{\downarrow}_{1}(x)$ &  0 & 0& 1 & 0 & 0 & 0 & 0 & 1 & 0 & 0 & 1 & 0 & 0 & 0 & 0 & 1 & 0 & 0& 0\\
    \hline
    $\eta^{\uparrow}_{2}(x)$ &  0 & 1 & 0 & 0 & 0 & 1 & 0 & 0 & 0 & 0 & 0 & 0 & 0 & 0 & 1 & 0 & 0 & 0& 0\\
    \hline
    $\eta^{\downarrow}_{2}(x)$ &  0 & 0 & 0 & 1 & 0 & 0 & 0 & 0 & 0 & 0 & 0 & 1 & 0 & 0 & 0 & 0 & 1 & 0& 0\\
    \hline
    $\eta^{\uparrow}_{3}(x)$ &  0 & 0 & 0 & 0 & 0 & 0 & 1 & 0 & 0 & 0 & 0 & 0 & 0 & 0 & 0 & 0 & 0 & 0& 0\\
    \hline
    $\eta^{\downarrow}_{3}(x)$ &  0 & 0 & 0 & 0 & 0 & 0 & 0 & 0 & 0 & 0 & 0 & 0 & 1 & 0 & 0 & 0 & 0 & 0& 0\\
    \hline
    $\eta^{\uparrow}_{4}(x)$ &  0 & 0 & 0 & 0 & 0 & 0 & 0 & 0 & 0 & 1 & 0 & 0 & 0 & 0 & 0 & 0 & 0 & 0& 0\\
    \hline
    $\eta^{\downarrow}_{4}(x)$ &  0 & 0 & 0 & 0 & 0 & 0 & 0 & 0 & 0 & 0 & 0 & 0 & 0 & 0 & 0 & 0 & 0 & 1 & 0 \\
           \end{tabular}
        \end{center}
        \caption{The slot configuration and the seat number configuration.}\label{fig:slot_seat}
    \end{figure}
Precise definitions of these quantities are given in Section \ref{sec:slot}. Then, we have the following relation between the quantities from the slot configuration and those from the seat number configuration. 

\begin{proposition}[Seat-Slot]\label{prop:seat_slot}
Suppose that $\eta \in \Omega_{< \infty}$. Then for any $k \in \N$ and $x \in \N$, we have the following equivalence : 
\begin{align}\label{eq:slot=seat}
    \eta^{\uparrow}_{k}(x) + \eta^{\downarrow}_{k}(x) = 1 \ \text{if and only if} \ \nu(x) = k - 1.
\end{align}
In particular, for any $k \in \N$, $i \in \Z_{\ge 0}$ and $x \in \Z_{\ge 0}$, we have $\tilde{\xi}_{k}(x) = \xi_{k}(x)$ and $\tilde{\zeta}_{k}(i) = \zeta_{k}(i)$.
\end{proposition}
The reader can check the relation \eqref{eq:slot=seat} for the ball configuration $\eta = 1100111011000110000\dots$ from Figure \ref{fig:slot_seat}. Since the construction of the slot configuration requires the TS algorithm, $\nu(x)$ cannot be defined for general $\eta \in \Omega$ due to the existence of infinitely many balls, and so the above proposition is also restricted to $\Omega_{<\infty}$. In particular, if the number of records in $\eta \in \Omega$ is finite, then we may not be able to identify solitons in $\eta$, see the description of the TS algorithm given in Appendix for details. However, the seat number configuration can be defined for any $\eta \in \Omega$, and thus it can be considered as a generalization of the slot configuration.  

{Using} Propositions \ref{prop:seat_KKR} and \ref{prop:seat_slot}, {we find for the first time that the relationship between local energy and slots can be understood via the seat number configuration.} We {highlight} that KKR-bijection does not distinguish roles of $0$'s in $\eta$, but the seat number configuration and slot configuration do so via the $(k,\downarrow)$-seats and $k$-slots, respectively. In other words, the seat number configuration and slot configuration also give energy to $0$'s. On the other hand, the slot configuration does not distinguish $1$'s and $0$'s if they are both $k$-slots, while the seat number configuration distinguishes them as $(k,\uparrow)$ and $(k,\downarrow)$. By introducing such a distinction, we obtain the nontrivial relation between the dynamics of $(k,\uparrow)$ and $(k,\downarrow)$ configurations, see Proposition \ref{prop:seat_flip}. See also \cite[Proposition 1.3]{FNRW} for an equivalent claim as that of Proposition \ref{prop:seat_flip} via the language of the slots. 

From the above propositions, we have an explicit relation between the riggings of KKR-bijection and the slot configuration. In the next theorem we denote $\mathbf{J}=\displaystyle \lim_{x \to \infty}\mathbf{J}(x)$ the rigging for a configuration $\eta \in \Omega_{< \infty}$, which is well defined since $\mathbf{J}(x)$ becomes constant in $x$ eventually. The indexes of the rigging $\mathbf{J}$ will be $J_{k,j}$ for $k,j$ in a suitable range.

\begin{theorem}[KKR-slot]\label{thm:2}
Suppose that $\eta \in \Omega_{< \infty}$. Then for any $k \in \N$ and $i \in \Z_{\ge 0}$, we have
\[
\tilde{\zeta}_{k}(i)=|\{j \in \N  \ ; \ J_{k,j}=i-k\}|.
\]
\end{theorem}
Theorem \ref{thm:2} means that the elements of $J_{k}$ are the effective positions of $\tau_{k}(\cdot)$ shifted by $k$, and the slot decomposition counts the total number of $\tau_{k}(\cdot)$ at the same effective position. As a direct consequence of Theorem \ref{thm:2} and the result in \cite{KOSTY} quoted as Theorem \ref{thm:KOSTY} in Section \ref{sec:KKR}, we see that the slot configuration linearizes the BBS with finite capacity BBS($\ell$). More precisely, we obtain the following theorem, which is a generalization of \cite[Theorem 1.4]{FNRW} for the case $\ell < \infty$.
\begin{theorem}\label{thm:slot_finite}
    Suppose that $\eta \in \Omega_{< \infty}$. For any $k \in \N$, $i \in \Z_{\ge 0}$ and $l \in \N \cup \{ \infty \}$, we have
    \begin{align}
        (T_{\ell}\tilde{\zeta})_{k}(i)= \tilde{\zeta}_{k}(i-(k\wedge \ell)).
    \end{align}
\end{theorem}

We mention some possible extensions of Proposition \ref{prop:seat_KKR} and Theorem \ref{thm:2}. In literature various extensions of the BBS have been defined and studied \cite{HHIKTT,HKT,IKT,KMP2,KOY,Takahashi_93,TTM}. One such generalization is given by the {\it multi-color} BBS with finite/infinite carrier capacity and it is known that such model can also be linearized by the KKR bijection. Nevertheless, in {the} colored setting such linearization techniques do not allow to study general hydrodynamic properties of the model in a rigorous way. To attack such probabilistic questions a linearization method more close in spirit to that of slot configurations seems to be required, yet no such result is, at this moment, available. We expect that Proposition \ref{prop:seat_KKR} and Theorem \ref{thm:2} might give a blueprint to generalize the idea of the slot/seat number configuration for multi-color BBS, and hence to carry out hydrodynamic studies of these generalized models.

Finally we note an application of Theorem \ref{thm:slot_finite} to the derivation of the generalized hydrodynamic limit (GHD limit) for the BBS($\ell$), $\ell < \infty$. In \cite{CS} the GHD limit for the BBS with infinite carrier capacity $(\ell = \infty)$ is rigorously derived, and the use of the slot decomposition is crucial in their strategy of the proof. However, the assumption $\ell = \infty$ is not necessary for most of the proof, and is only needed to use \cite[Theorem 1.4]{FNRW}, the linearization property of the slot decomposition. Therefore, combining Theorem \ref{thm:slot_finite} and the strategy in \cite{CS}, the GHD limit for the BBS($\ell$) can be also derived in a rigorous way.

{
\begin{remark}\label{rem:infinite}
    Since the main purpose of this paper is to investigate the relationships between the KKR bijection and the slot configuration, and the KKR bijection is only defined for semi-infinite sequences, we consider the BBS on $\{ 0, 1 \}^{\N}$. On the other hand, the slot configuration and the seat number configuration can be also defined for $\eta \in \left\{ 0 , 1 \right\}^{\Z}$ satisfying
        \begin{align}
            \lim_{x \to \infty} \frac{1}{x} \sum_{y = 1}^{x} \eta(y) < \frac{1}{2},  \  \lim_{x \to \infty} \frac{1}{x} \sum_{y = -x}^{-1} \eta(y) < \frac{1}{2},
        \end{align}
    and the relation \eqref{eq:slot=seat} also holds. That is, for the whole line case, the seat number configuration is also a generalization of the slot configuration, see \cite[Section 4]{S} for the construction of the seat number configuration on the whole line and the proof of an analogue of Proposition \ref{prop:seat_slot}. 
    We note that for the whole line case, since there are seats in both directions, the function $\xi_{k}$ may take values in $(-\infty, \infty)$, and an ambiguity arises as to where to assign the value of $0$ for $\xi_{k}$. 
    As a result, to describe an analogue of Theorem \ref{thm:seat_ISM}, we need an offset, which is also the case for the slot configuration, see \cite[Theorem 3.1]{FNRW} and \cite[Theorem 4.1]{S} for details.
\end{remark}
}

\section{Linearization property of the seat number configuration}\label{sec:seat}

In this section, we first state some simple observations obtained by the definition of the seat number configuration. Then, we prove Theorem \ref{thm:seat_ISM}. 

\subsection{Basic properties of the seat number configuration}

\begin{lemma}\label{lem:1}
For any $\eta \in \Omega$, the {following statements} hold : 
\begin{enumerate}
   \item[\rm{(i)}]  For any $k \in \N$, $\eta^{\uparrow}_{k}(x)=1$ implies $\sum_{y=1}^x(\eta^{\uparrow}_{\ell}(y) -\eta^{\downarrow}_{\ell}(y))=1$ for any $1 \le \ell \le k$. 
  \item[\rm{(ii)}] For any $k \in \N$, $\eta^{\downarrow}_{k}(x)=1$ implies $\sum_{y=1}^x(\eta^{\uparrow}_{\ell}(y) -\eta^{\downarrow}_{\ell}(y))=0$ for any $1 \le \ell \le k$.
   \item[\rm{(iii)}] $r(x)=1$ implies $\sum_{y=1}^x(\eta^{\uparrow}_{k}(y) -\eta^{\downarrow}_{k}(y))=0$ for any $k \in \N$.
   \end{enumerate}
    \end{lemma}
\begin{proof}
    From \eqref{eq:lem:1_1}, it is sufficient to show the {following statements} :
        \begin{enumerate}
   \item[\rm{(i)}']  For any $k \in \N$, $\eta^{\uparrow}_{k}(x)=1$ implies $\mathcal{W}_{\ell}(x) = 1$ for any $1 \le \ell \le k$. 
  \item[\rm{(ii)}']  For any $k \in \N$, $\eta^{\downarrow}_{k}(x)=1$ implies $\mathcal{W}_{\ell}(x) = 0$ for any $1 \le \ell \le k$.
   \item[\rm{(iii)}'] $r(x)=1$ implies $\mathcal{W}_{k}(x) = 0$ for any $k \in \N$.
   \end{enumerate}
   We prove them one by one.
    \begin{enumerate}
        \item[\rm{(i)}'] Assume that $\eta^{\uparrow}_{k}(x)=1$. Then, from the update rule of $\mathcal{W}(\cdot)$, the seats of No.$\ell$ for $1\le \ell \le k$ are all occupied at $x$. In {formulas}, from \eqref{eq:kup_alt} we have 
            \begin{align}
                 \eta(x) = 1, \quad \mathcal{W}_{k}(x-1) = 0, \quad  \prod_{\ell = 1}^{k - 1}\mathcal{W}_{\ell}(x-1) = 1, 
            \end{align}
        and thus from \eqref{eq:dynamics} we obtain $\mathcal{W}_{\ell}(x) = 1$ for any $1 \le \ell \le k$. 
        \item[\rm{(ii)}'] Assume that $\eta^{\downarrow}_{k}(x)=1$. Then, from the update rule of $\mathcal{W}(\cdot)$, the seats of No.$\ell$ for $1\le \ell \le k$ are all empty at $x$. In {formulas}, from \eqref{eq:kdown_alt} we have 
            \begin{align}
                \eta(x) = 0, \quad \mathcal{W}_{k}(x-1) = 1, \quad \prod_{\ell=1}^{k-1}(1-\mathcal{W}_{\ell}(x-1)) = 1, 
            \end{align}
        and thus from \eqref{eq:dynamics} we obtain $\mathcal{W}_{\ell}(x) = 0$ for any $1 \le \ell \le k$.
        \item[\rm{(iii)}'] Assume that $r(x)=1$. Then the seat of No.$k$ for $k \in \N$ are all empty at $x$. In {formulas}, $r(x)=1$ if and only if $W_{\infty}(x-1)=W_{\infty}(x)=0$ where $W_{\infty}(x)=\sum_{k \in \N}\mathcal{W}_k(x)$, and thus we have $\mathcal{W}_{k}(x) = 0$ for any $k \in \N$.
        
    \end{enumerate}
\end{proof}

The next proposition is crucial {for understanding} the dynamics of the BBS. 
\begin{proposition}\label{prop:seat_flip}
For any $\eta \in \Omega$, $x \in \N$ and $k \in \N$,
\begin{align}\label{eq:flip}
    \eta^{\downarrow}_{k}(x) &= T\eta^{\uparrow}_{k}(x).
\end{align}
In addition, if $\eta^{\uparrow}_{k}(x) = 1$ then we have
\begin{align}\label{eq:flip2}
    \sum_{\ell \ge k}  T\eta^{\downarrow}_{\ell}(x) + Tr(x) = 1.
\end{align}
\end{proposition}

The proof of Proposition \ref{prop:seat_flip} is in the next subsection. 

\begin{remark}\label{rem:conserv}

    By Lemma \ref{lem:1} (iii) and \eqref{eq:flip}, we see that if $r(x) = 1$ then we have
        \begin{align}\label{eq:conservation}
            \sum_{y = 1}^{x} \eta^{\uparrow}_{k}(y) = \sum_{y = 1}^{x} \eta^{\downarrow}_{k}(y) = \sum_{y = 1}^{x} T\eta^{\uparrow}_{k}(y), 
        \end{align}
    for any $k \in \N$. In particular, under the assumption $\eta \in \Omega_{< \infty}$, we have 
        \begin{align}
            \sum_{x \in \N} \eta^{\uparrow}_{k}(x) = \sum_{x \in \N} \eta^{\downarrow}_{k}(x) = \sum_{x \in \N} T\eta^{\uparrow}_{k}(x) = \sum_{x \in \N} T\eta^{\downarrow}_{k}(x)
        \end{align}
    since $x$ must be a record of $\eta$ and $T\eta$ if $x$ is sufficiently large. Hence, the total number of $(k,\sigma)$-seats is conserved in time for each $k \in \N$ and $\sigma\in\{ \uparrow, \downarrow\}$. When $\sum_{x \in \N} \eta(x) = \infty$, the above conservation law does not necessarily hold. 
    
\end{remark}

\begin{remark}
  Relation \eqref{eq:flip} is essentially equivalent to Proposition 1.3 of \cite{FNRW}, but generalized to configurations with infinitely many balls.
\end{remark}

     \subsection{Proof of Proposition \ref{prop:seat_flip}}
     First note that if $(\mathcal{W}_k)_k$ is the carrier with seat numbers for the configuration $\eta$, then $(1-\mathcal{W}_k)_k$ is the carrier with seat numbers for the configuration $1-\eta$, namely $(1-\mathcal{W}_k)_k$ satisfies the equation (\ref{eq:dynamics}) for $1-\eta$, but with the boundary condition $(1-\mathcal{W}_k)(0)=1$ for all $k \in \N$. Now, let $\tilde{\eta}=1-T\eta$ and $\tilde{\mathcal{W}}_{k}=1-T\mathcal{W}_{k}$. Then,  $\tilde{\mathcal{W}}= (\tilde{\mathcal{W}}_{k}) $ is the carrier with seat numbers for the configuration $\tilde{\eta}$ with the boundary condition $\tilde{\mathcal{W}}_{k}(0)=1$ for all $k \in \N$. More precisely,  $\tilde{\mathcal{W}}= (\tilde{\mathcal{W}}_{k}) $ satisfies the equation (\ref{eq:dynamics}) for $\tilde{\eta}$. Moreover, {from \eqref{eq:rec_T},}
     \begin{align}
    \tilde{\eta}(x)=     \begin{cases}
    \eta(x) \quad &\text{if} \quad r(x)=0, \\
    1-\eta(x) \quad &\text{if} \quad r(x)=1.
    \end{cases}
     \end{align}
   Then, \eqref{eq:flip} is equivalent to the claim that
   \begin{equation} \label{eq:W_difference}
       \tilde{\mathcal{W}}_{k}(x)-\tilde{\mathcal{W}}_{k}(x-1) =-1
   \end{equation}
   if and only if $\eta^{\downarrow}_{k}(x)=1$. To prove this, we first prove that $\tilde{\mathcal{W}}_{k}$ dominates $\mathcal{W}_{k}$.
   
   \begin{lemma}\label{lem:monotone}
  For any $x \in \Z_{\ge 0}$ and $k \in \N$,    \[\tilde{\mathcal{W}}_{k}(x) \ge \mathcal{W}_{k}(x).
  \]
   \end{lemma}
   \begin{proof}
        We prove it by induction on $x$. For $x=0$, the inequality holds since $\tilde{\mathcal{W}}_{k}(0)=1$ and $\mathcal{W}_{k}(0)=0$ for $k \in \N$.
        Suppose 
        \[\tilde{\mathcal{W}}_{k}(x-1) \ge \mathcal{W}_{k}(x-1),
\]
for all $k \in \N$. If $r(x)=1$, $\mathcal{W}_{k}(x-1)=\mathcal{W}_{k}(x)=0$ for all $k \in \N$, so 
\[\tilde{\mathcal{W}}_{k}(x) \ge \mathcal{W}_{k}(x)
\]
holds for all $k \in \N$. If $r(x)=0$, then $\eta(x)=\tilde{\eta}(x)$. If $\eta(x)=\tilde{\eta}(x)=1$, then $\eta^{\uparrow}_{k^*}(x)=1$ for some $k^* \in \N$. Therefore, $\mathcal{W}_{k}(x-1)=1$ for all $1 \le k < k^*$ and so $\tilde{\mathcal{W}}_{k}(x-1)=1$ by the induction assumption. This implies that  $\tilde{\mathcal{W}}_{k^*}(x)=1$ holds for both cases $\tilde{\mathcal{W}}_{k^*}(x-1)=0$ or $1$. Hence, 
\[
\tilde{\mathcal{W}}_{k^*}(x) =\mathcal{W}_{k^*}(x)=1
\]
and for $k \neq k^*$,
\[
\tilde{\mathcal{W}}_{k}(x) \ge \tilde{\mathcal{W}}_{k}(x-1) \ge \mathcal{W}_{k}(x-1)=\mathcal{W}_{k}(x).
\]
Similarly, if $\eta(x)=\tilde{\eta}(x)=0$, then there exists $k^* \in \N$ such that $\tilde{\mathcal{W}}_{k^*}(x) -\tilde{\mathcal{W}}_{k^*}(x-1) =-1$. Then, $\tilde{\mathcal{W}}_{k}(x-1)=0$ for all $1 \le k < k^*$ and so $\mathcal{W}_{k}(x-1)=0$ by the induction assumption. Hence, $\mathcal{W}_{k^*}(x)=0$ holds for both cases $\mathcal{W}_{k^*}(x-1)=0$ or $1$. Hence, 
\[
\tilde{\mathcal{W}}_{k^*}(x) =\mathcal{W}_{k^*}(x)=0
\]
and for $k \neq k^*$,
\[
\tilde{\mathcal{W}}_{k}(x) = \tilde{\mathcal{W}}_{k}(x-1) \ge \mathcal{W}_{k}(x-1) \ge \mathcal{W}_{k}(x),
\]
which completes the inductive step.
\end{proof}
        Next, we prove that $\tilde{\mathcal{W}}_{k}$ and $\mathcal{W}_{k}$ coincide on sufficiently large intervals.
 \begin{lemma}\label{lem:coincide}
 Suppose $x' < x $, $\eta^{\uparrow}_{k}(x')=1$ and $r(y)=0$ for all $x' <  y \le x$. Then, 
 \[
\mathcal{W}_{\ell}(y) = \tilde{\mathcal{W}}_{\ell}(y) 
 \]
 for any $x' \le  y \le x$ and $1 \le \ell \le k$.
 \end{lemma}
 \begin{proof}
 {Since} $\eta^{\uparrow}_{k}(x')=1$ implies $\mathcal{W}_{\ell}(x')=1$ for $1 \le \ell \le k${,} by Lemma \ref{lem:monotone}, {we have} $\tilde{\mathcal{W}}_{\ell}(x')=1$ for $1 \le \ell \le k$. In particular, $\mathcal{W}_{\ell}(x')=\tilde{\mathcal{W}}_{\ell}(x')$ for $1 \le \ell \le k$. Also, $r(y)=0$ for $x' <  y \le x$ implies $\eta(y)=\tilde{\eta}(y)$ for $x' <  y \le x$. Then, since $\{\mathcal{W}_{\ell}(y)\}_{ x' < y \le x, 1 \le \ell \le k}$ (resp. $\{\tilde{\mathcal{W}}_{\ell}(y)\}_{x' < y \le x, 1 \le \ell \le k}$) is determined by $\{\mathcal{W}_{\ell}(x')\}_{1 \le \ell \le k}$ and $\{\eta(y)\}_{ x' < y \le x}$ ( resp. $\{\tilde{\mathcal{W}}_{\ell}(x')\}_{1 \le \ell \le k}$ and $\{\tilde{\eta}(y)\}_{ x' < y \le x}$) through the recursive equation (\ref{eq:dynamics}), we conclude that $\mathcal{W}_{\ell}(y)=\tilde{\mathcal{W}}_{\ell}(y)$ for $x' \le y \le x$ and $1 \le \ell \le k$.
 \end{proof}
 
 \begin{proof}[Proof of Proposition \ref{prop:seat_flip}]
 We first show that $\eta^{\downarrow}_{k}(x)=1$ implies \eqref{eq:W_difference}, that is $\tilde{\mathcal{W}}_{k}(x)-\tilde{\mathcal{W}}_{k}(x-1)=-1$. Then we will prove the opposite implication.
 Suppose $\eta^{\downarrow}_{k}(x)=1$. This means 
 \[
 \mathcal{W}_{k}(x-1)=1, \quad  \mathcal{W}_{k}(x)=0.
 \]
 Let
 \[
 x':=\max\{ y \in \N \ ; \ \eta^{\uparrow}_{k}(y)=1, y <x\},
 \]
 that is the rightmost site to the left of $x$ where a ball is picked up and seated at No.$k$ seat. 
 We can also characterize $x'$ as
 \[
 x'=\min\{ y \in \N \ ; \ \mathcal{W}_{k}(z)=1 \ \text{for all} \ y \le z \le x-1\}.
 \]
 Then, it is obvious that $\eta^{\uparrow}_{k}(x')=1$, $x' <x$ and $r(y)=0$ for all $x' < y \le x$, since $r(y)=1$ implies $\mathcal{W}_{k}(y)=\mathcal{W}_{k}(y-1)=0$. Then, by Lemma \ref{lem:coincide}, \[
 \mathcal{W}_{k}(x-1)=\tilde{\mathcal{W}}_{k}(x-1)=1
 \]
 and
 \[
 \mathcal{W}_{k}(x)=\tilde{\mathcal{W}}_{k}(x)=0
 \]
 hold. In particular, \eqref{eq:W_difference} holds.
 Next, we assume \eqref{eq:W_difference} holds and prove $\eta^{\downarrow}_{k}(x)=1$. Since the relation $\tilde{\mathcal{W}}_{k}(x)-\tilde{\mathcal{W}}_{k}(x-1)=-1$ implies {$T\eta(x) = 1$ and} $\tilde{\eta}(x)=0$, we have $r(x)=0$ and $\eta(x)=0$. Hence, there exists $k^* \ge 1$ such that $\eta^{\downarrow}_{k^*}(x)=1$. Then, by the first part of this proof, this implies 
 \[
 \tilde{\mathcal{W}}_{k^*}(x)-\tilde{\mathcal{W}}_{k^*}(x-1)=-1,
 \]
 which means $k=k^*$, and so $\eta^{\downarrow}_{k}(x)=1$. 
 
Finally, we prove \eqref{eq:flip2}. {If $\eta^{\uparrow}_{k}(x) = 1$, then $T\eta(x) = 0$ and thus we have 
    \begin{align}
        \sum_{\ell \in \N} T\eta^{\downarrow}_{\ell}(x) + Tr(x) = 1.
    \end{align}
    Hence, it is sufficient to show that $\eta^{\uparrow}_{k}(x) = 1$ implies $T\eta^{\downarrow}_{\ell}(x) = 0$ for $1 \le \ell \le k - 1$. We observe that $T\mathcal{W}_{\ell}(x - 1)= T\mathcal{W}_{\ell}(x) = 0$ implies $T\eta^{\downarrow}_{\ell}(x) = 0$. Since $T\mathcal{W}_{\ell}(x) = 1 - \tilde{\mathcal{W}}_{\ell}(x)$, $T\mathcal{W}_{\ell}(x - 1)= T\mathcal{W}_{\ell}(x) = 0$ is equivalent to $\tilde{\mathcal{W}}_{\ell}(x - 1) = \tilde{\mathcal{W}}_{\ell}(x) = 1$. On the other hand, $\eta^{\uparrow}_{k}(x) = 1$ implies $\mathcal{W}_{\ell}(x - 1) = \mathcal{W}_{\ell}(x) = 1$ for $1 \le \ell \le k - 1$, and thus from Lemma \ref{lem:monotone} we have $\tilde{\mathcal{W}}_{\ell}(x - 1) = \tilde{\mathcal{W}}_{\ell}(x) = 1$. Therefore, $\eta^{\uparrow}_{k}(x) = 1$ implies $T\eta^{\downarrow}_{\ell}(x) = 0$ for $1 \le \ell \le k  -1$.}
 \end{proof}

\subsection{Proof of Proposition \ref{prop:match}}\label{proof:match} 

    In this subsection, we will show Proposition \ref{prop:match}. First we define for any $i\in \Z_{\ge 0}$ and $k \in \N$
    \begin{equation} \label{eq:s_k(i)}
        s_{k}(i) := \min \left\{ x \in \mathbb{Z}_{\ge 0} \ ; \ \xi_{k}(x) = i \right\},
    \end{equation}
with the convention that $\min \emptyset = \infty$. Since $\xi_{k}(x+1) -\xi_{k}(x) \in \{0,1\}$, the equivalence
\begin{equation}\label{eq:xi-s}
\xi_{k}(x) = i \ \text{if and only if} \  s_{k}(i) \le x < s_{k}(i+1)
\end{equation}
holds, where $s_{k}(i+1)$ can be infinite. 
    
    Since $s_{k}(i)$ is a $(\ell,\sigma)$-seat for some $\ell > k$ and $\sigma \in \{\uparrow,\downarrow\}$ or a record, by using \eqref{eq:diff1} and Lemma \ref{lem:1}, the following result is straightforward.
        \begin{lemma}\label{lem:pm_eq}
            Suppose that $s_{k}(i) < \infty$ for some $k \in \N$ and $i \in \Z_{\ge 0}$. Then we have
                \begin{align}
                    m^{\uparrow}_{k}\left( s_{k}(i) \right) = m^{\downarrow}_{k}\left( s_{k}(i) \right). 
                \end{align}
        \end{lemma}
    Next, we show that the sequence $ \left(m^{\sigma}_{k}\left( s_{k}(i) \right)\right)_{i \in \N}$, $\sigma \in \{\uparrow, \downarrow\}$ is non-decreasing.   
        \begin{lemma}\label{lem:nonnegative}
            Suppose that $s_{k}(i+1) < \infty$ for some $k \in \N$ and $i \in \Z_{\ge 0}$. Then for each $\sigma\in\{ \uparrow, \downarrow\}$, we have
                \begin{align}\label{eq:lem:nonnegative}
                    m^{\sigma}_{k}\left( s_{k}(i+1) \right) - m^{\sigma}_{k}\left( s_{k}(i) \right) \ge 0.
                \end{align}
        \end{lemma}
        \begin{proof}
            From Lemma \ref{lem:pm_eq}, it is sufficient to prove the case $\sigma=\uparrow$. Observe that 
            \begin{align}
            m^{\uparrow}_{k}\left( s_{k}(i+1) \right) - m^{\uparrow}_{k}\left( s_{k}(i) \right) & = \sum_{y = s_{k}(i) + 1}^{s_{k}(i+1)}  (\eta^{\uparrow}_{k}(y) - \eta^{\uparrow}_{k+1}\left(y\right))\\ 
            & =\sum_{y = s_{k}(i) + 1}^{s_{k}(i+1)}  \eta^{\uparrow}_{k}(y) - \eta^{\uparrow}_{k+1}\left(s_{k}(i+1)\right).
             \end{align}
            From the above expression, \eqref{eq:lem:nonnegative} clearly holds for the case $\eta^{\uparrow}_{k+1}\left(s_{k}(i+1)\right) = 0$. From now on we will consider the case $\eta^{\uparrow}_{k+1}\left(s_{k} (i+1)\right) = 1$. Then, to show \eqref{eq:lem:nonnegative} it is sufficient to show 
                \begin{align}
                    \sum_{y = s_{k}(i) + 1}^{s_{k}(i+1)}  \eta^{\uparrow}_{k}(y) \ge 1.
                \end{align}
            From Lemma \ref{lem:pm_eq}, we have 
             \begin{align}
                    \sum_{y = s_{k}(i) + 1}^{s_{k}(i+1)} \left( \eta^{\uparrow}_{k}(y) - \eta^{\downarrow}_{k}(y) \right) &= \sum_{y = s_{k}(i) + 1}^{s_{k}(i+1)} \left( \eta^{\uparrow}_{k+1}(y) - \eta^{\downarrow}_{k+1}(y) \right)  \\
                    &= \eta^{\uparrow}_{k+1}\left(s_{k}(i+1)\right) - \eta^{\downarrow}_{k+1}\left(s_{k}(i+1)\right).
                \end{align}
            Hence, we obtain
                \begin{align}
                    \sum_{y = s_{k}(i) + 1}^{s_{k}(i+1)}  \eta^{\uparrow}_{k}(y) &\ge \eta^{\uparrow}_{k+1}\left(s_{k}(i+1)\right) - \eta^{\downarrow}_{k+1}\left(s_{k}(i+1)\right) \\
                    &\ge 1,
                \end{align}
        and this completes the proof.
        \end{proof}

    \begin{proof}[Proof of Proposition \ref{prop:match}] 
    From the definition of $\tau_{k}(\cdot)$ given by \eqref{eq:defoftau}, it is sufficient to show that $x \ge \tau_{k}(j)$ implies $m^{\sigma}_{k}(x) \ge j$ for each $\sigma \in \left\{\uparrow,\downarrow \right\}$. Since $m^{\sigma}_{k}$ decreases only at $(k+1,{\sigma})$-seats, it suffices to prove the following claim : for any $x \ge \tau_{k}(j)$, 
        \begin{align}
            \eta^{\uparrow}_{k+1}(x) + \eta^{\downarrow}_{k+1}(x) = 1 \ \text{implies} \ m^{\sigma}_{k}(x) \ge j \ \text{for each} \ \sigma \in \left\{\uparrow,\downarrow \right\}.
        \end{align}
    Define $x' := \min \left\{ y \ge \tau_{k}(j) \ ; \ y = s_{k}(i) ~ \text{for some} ~ i \in \N \right\}$. Note that $x' <\infty$ and in particular $x' \le x$ since $\eta_{k+1}^{\uparrow}(x)+\eta_{k+1}^{\downarrow}(x)=1$ and so $x=s_k(i')$ for some $i'$. Then, again by using the fact that $m^{\sigma}_{k}$ decreases only at $(k+1,{\sigma})$-seats, we see that either $m^{\uparrow}_{k}(x') \ge j$ or $m^{\downarrow}_{k}(x') \ge j$ holds. Thus by using Lemma \ref{lem:pm_eq} we have $m^{\uparrow}_{k}(x') = m^{\downarrow}_{k}(x') \ge j$. Then from Lemma \ref{lem:nonnegative}, we obtain $m^{\sigma}_{k}(x) \ge m^{\sigma}_{k}(x') \ge j$ for $\sigma \in \{\uparrow, \downarrow\}$. Therefore, Proposition \ref{prop:match} is proved.  
    \end{proof}
    
    We conclude this subsection by pointing out that similar argument used above yields other representations of $\zeta_{k}(i)$ defined in \eqref{def:zeta} as follows. 
        \begin{lemma}\label{lem:repofzeta}
            Suppose that $s_{k}(i+1) < \infty$ for some $k \in \N$ and $i \in \Z_{\ge 0}$. Then for each $\sigma\in\{ \uparrow, \downarrow\}$ we have
                \begin{align}
                    \zeta_{k}(i) &= m^{\sigma}_{k}\left( s_{k}(i+1) \right) - m^{\sigma}_{k}\left( s_{k}(i) \right) \\
                    &= \left|\left\{x \in \N  \ ; \  \eta^{\sigma}_{k}(x)=1, \xi_{k}(x)=i\right\}\right|  - \left|\left\{x \in \N  \ ; \ \eta^{\sigma}_{k+1}(x)=1, \xi_{k}(x)=i+1\right\}\right| \label{eq:repofzeta}.
                \end{align}
        \end{lemma}
        \begin{proof}[Proof of Lemma \ref{lem:repofzeta}]
            
            Let $j^*=\max\{j \in \N ; \ \tau_k(j) < s_k(i+1)\}$ with the convention that $\max \emptyset =0$. {Then, we have 
            \begin{align}
                j^* &= \sum_{j = 0}^{i} \left| \left\{ h \in \N \ ; \ \tau_k(h) \in \left\{ x \in \Z_{\ge 0} \ ; \ \xi_{k}(x) = j \right\} \right\} \right| = \sum_{j = 0}^{i} \zeta_{k}(j).
            \end{align}
            On the other hand, since $\tau_k(j^*) < s_k(i+1) < \tau_k(j^*+1)$, from \eqref{eq:defoftau} and from Proposition \ref{prop:match}, we get  
                \begin{align}
                    \min\left\{ m^{\uparrow}_{k}\left( s_{k}(i+1) \right), m^{\downarrow}_{k}\left( s_{k}(i+1) \right) \right\} = j^*. 
                \end{align}}
            Hence, {from the above and} Lemma \ref{lem:pm_eq}, for each $\sigma \in \{\uparrow, \downarrow\}$ we have
                \begin{align}
                    m^{\sigma}_{k}\left( s_{k}(i+1) \right) =  \min\left\{ m^{\uparrow}_{k}\left( s_{k}(i+1) \right), m^{\downarrow}_{k}\left( s_{k}(i+1) \right) \right\} = j^* =\sum_{j = 0}^{i} \zeta_{k}(j),
                \end{align}
            and thus we obtain 
                \begin{align}
                    \zeta_{k}(i) &= m^{\sigma}_{k}\left( s_{k}(i+1) \right) - m^{\sigma}_{k}\left( s_{k}(i) \right) \\
                    &= \sum_{x = s_{k}(i) + 1}^{s_{k}(i+1)} \eta^{\sigma}_{k}(x) - \eta^{\sigma}_{k+1}\left(s_{k}(i+1)\right) \\
                    &= \left|\left\{x \in \N \ ; \ \eta^{\sigma}_{k}(x)=1, \xi_{k}(x)=i\right\}\right|  - \left|\left\{x \in \N \ ; \ \eta^{\sigma}_{k+1}(x)=1, \xi_{k}(x)=i+1\right\}\right|.
                \end{align}
        \end{proof}

\subsection{Proof of Theorem \ref{thm:seat_ISM}}

In this subsection, we give the proof of Theorem \ref{thm:seat_ISM}. 
First, we prove that the difference between $\xi_{k}$ and $T\xi_{k}$ is constant under a certain condition.

\begin{lemma}\label{lem:Txi}
    For any $k \in \N$ and $x \in \Z_{\ge 0}$, we have
        \begin{align}\label{eq:Txi_xi}
            T\xi_{k}(x) - \xi_{k}(x) \ge 0.
        \end{align}
    In addition, if $\eta^{\downarrow}_{\ell}(x)=1$ and $\ell \ge k$, then 
    \begin{align}\label{eq:Txi}
    T\xi_{k}(x)-\xi_{k}(x)=k.
    \end{align}
\end{lemma}
\begin{proof}
From \eqref{eq:cap_seat}, \eqref{eq:lem:1_1} and \eqref{eq:flip}, we have
        \begin{align}
            T\xi_{k}(x) - \xi_{k}(x) &= \sum_{\ell = 1}^{k} \sum_{y = 1}^{x} \left( \eta^{\uparrow}_{\ell}(y) - T\eta^{\uparrow}_{\ell}(y) \right) + \sum_{\ell = 1}^{k} \sum_{y = 1}^{x} \left( \eta^{\downarrow}_{\ell}(y) - T\eta^{\downarrow}_{\ell}(y) \right) \\
            &=  \sum_{\ell = 1}^{k} \sum_{y = 1}^{x} \left( \eta^{\uparrow}_{\ell}(y) - \eta^{\downarrow}_{\ell}(y) \right) + \sum_{\ell = 1}^{k} \sum_{y = 1}^{x} \left( T\eta^{\uparrow}_{\ell}(y) - T\eta^{\downarrow}_{\ell}(y) \right) \\
            &= W_{k}(x) + TW_{k}(x) \ge 0.
        \end{align}
Suppose $\eta^{\downarrow}_{\ell}(x)=1$ and $\ell \ge k$. By \eqref{eq:cap_seat}, \eqref{eq:lem:1_1} and Lemma \ref{lem:1}(ii), $ W_{k}(x) = 0$. Also,
    by \eqref{eq:flip}, $T\eta^{\uparrow}_{\ell}(x)=1$ and so by \eqref{eq:cap_seat}, \eqref{eq:lem:1_1} and Lemma \ref{lem:1}(i), $ TW_{k}(x) = k$.
    Hence,
    \begin{align}
        T\xi_{k}(x)-\xi_{k}(x) = W_{k}(x) + TW_{k}(x) = k.
    \end{align}
\end{proof}

Now we give the proof of Theorem \ref{thm:seat_ISM}.

\begin{proof}[Proof of Theorem \ref{thm:seat_ISM}]

    First, we note that $0 < \left| \left\{ x \in \Z_{\ge 0} \ ; \  \xi_{k}(x) = i \right\} \right| < \infty$ if and only if $s_{k}(i+1) < \infty$. In addition, 
    from \eqref{eq:Txi_xi}, we have $Ts_{k}(i) \le s_{k}(i)$ for any $k \in \N$ and $i \in \Z_{\ge 0}$. Hence, we see that $ 0 < \left| \left\{ x \in \Z_{\ge 0} \ ; \  \xi_{k}(x) = i \right\} \right| < \infty$ implies $0 < \left| \left\{ x \in \Z_{\ge 0} \ ; \  T\xi_{k}(x) = i \right\} \right| < \infty$.
    
    Next, by Lemma \ref{lem:repofzeta} and \eqref{eq:flip}, 
   \begin{align}
       (T\zeta)_{k}(i) & = |\{x \in \N \ ; \ T\eta^{\uparrow}_{k}(x)=1, T\xi_{k}(x)=i\}|   - |\{x \in \N \ ; \ T\eta^{\uparrow}_{k+1}(x)=1, T\xi_{k}(x)=i+1\}| \\
        & = |\{x \in \N \ ; \ \eta^{\downarrow}_{k}(x)=1, T\xi_{k}(x)=i\}| - |\{x \in \N \ ; \ \eta^{\downarrow}_{k+1}(x)=1, T\xi_{k}(x)=i+1\}|.
   \end{align}
 Then, by \eqref{eq:Txi},
   \begin{align}
 |\{x \in \N \ ; \ \eta^{\downarrow}_{k}(x)=1, T\xi_{k}(x)=i\}| &= |\{x \in \N \ ; \ \eta^{\downarrow}_{k}(x)=1, \xi_{k}(x) = i - k\}|,
   \end{align}
 and similarly, 
  \begin{align}
 |\{x \in \N \ ; \ \eta^{\downarrow}_{k+1}(x)=1, T\xi_{k}(x)=i+1\}| & = |\{x \in \N \ ; \ \eta^{\downarrow}_{k+1}(x)=1, \xi_{k}(x) = i - k + 1\}|.
    \end{align}
 Hence, by Lemma \ref{lem:repofzeta},
\begin{align}
 (T\zeta)_{k}(i) &= |\{x \in \N \ ;  \ \eta^{\downarrow}_{k}(x)=1, \xi_{k}(x)=i-k\}| \\ & \quad - |\{x \in \N \ ; \ \eta^{\downarrow}_{k+1}(x)=1, \xi_{k}(x)=i - k + 1\}| \\ &=\zeta_{k}(i-k).
\end{align}
\end{proof}

\section{Relation to KKR-bijection}\label{sec:KKR}

    \subsection{Definition of the rigged configuration.}
    
    Let us recall some basic notions. A \emph{partition} $\mu =(\mu_1\ge \mu_2 \ge \cdots \ge 0)$ is a weakly decreasing sequence of non-negative integers that eventually becomes zero. Partitions are naturally represented by their Young diagrams and often the two notions are interchanged. The \emph{conjugate partition} of $\mu$, denoted by $\lambda$, is the partition defined by $\lambda_k = |\{i \in \N \ : \ \mu_i \ge k \}|$, for $k \in \N$. For any $k$, the multiplicity of $k$ in $\mu$ is $m_k(\mu) := \lambda_k-\lambda_{k+1}$;  we will often suppress the dependence from $\mu$ to lighten the notation. A \emph{rigging} of a partition $\mu$ is a collection of arrays $\mathbf{J} = \left\{ J_{k} \ : \ 1 \le k \le \mu_1 \right\}$ such that
    \begin{equation}
        J_{k} = (J_{k,1},\dots, J_{k,m_k}),
        \qquad
        \text{with}
        \qquad
        -k \le J_{k,1} \le \cdots \le J_{k,m_k}
    \end{equation}
    and $J_{k} = \varnothing$ if $m_k=0$. A pair $(\mu,\mathbf{J})$ consisting of a partition and its rigging is called a (rank one) \emph{rigged configuration} and we denote the set of them by $RC$.
    \begin{figure}[h]
    \centering
            \begin{tikzpicture}
            \node at (0,0.5) {$\mu = \left( 4,2,2,1 \right)$};
	           \node at (0,-0.5){$\lambda = (4,3,1,1)$};
	           \node[left] at (4,0) {$\yng(4,2,2,1)$};
	           \node[right] at (4.5,.6) {$J_{4,1}$};
	           \node[right] at (4.5,.15) {$J_{2,2}$};
	           \node[right] at (4.5,-.3) {$J_{2,1}$};
	           \node[right] at (4.5,-.75) {$J_{1,1}$};
	       \end{tikzpicture}
	       \caption{A partition $\mu$ and its conjugate on the left and a rigged configuration $(\mu,\mathbf{J})$ on the right.}
	       \label{fig:partition_and_rigging}
	\end{figure}
	
	Rigged configurations are in bijections with ball configurations $\eta \in \Omega_{< \infty}$. In {the} literature{,} such {a} bijection takes the name of {the} \emph{Kerov-Kirillov-Reshetikhin (KKR) bijection} \cite{KKR,KOSTY} and we recall it below. Starting from $\eta$ we will sequentially construct rigged configurations $(\mu(x),\mathbf{J}(x))$ for $x=0,1,\dots$ in such a way that $(\mu(x),\mathbf{J}(x))$ is a function of $\eta(1),\dots,\eta(x)$ and $(\mu,\mathbf{J}) = \lim_{x\to \infty}(\mu(x),\mathbf{J}(x))$ will be the rigged configuration associated {with} $\eta$. Abusing notation{,} we will denote the arrays in rigging $\mathbf{J}(x)$ by $J_{k}(x)$. For any $k \in \N$ and $x\ge 0$ define the \emph{$k$-th vacancy} at $x$ as
	\begin{equation}
	    p_k(x) = x - 2 E_{k}(x),
	\end{equation}
	where $E_{k}(x)$ is called the \emph{$k$-th energy} defined as
	    \begin{align}
	        E_{k}(x) := \sum_{i \in \N} \min \left\{\mu_i(x),k \right\}.
	    \end{align}
	In case $\mu_i(x)=k$ and $p_k(x) = J_{k,{m_{k}}}(x)$ for some $i$ and $k$, we say that $(\mu_i(x),J_{k,{m_{k}}}(x))$ is a \emph{singular row} of length $k$ of $(\mu(x),\mathbf{J}(x))$. {For examples of singular rows, see Figure \ref{ex:singular_3}. Each partition corresponds to $\mu(x)$, and the number on the left-hand (resp. right-hand) side of the row with length $k$ is the value of $p_{k}(x)$ (resp. $J_{k}(x)$) at some $x$. The leftmost rigged configuration has two singular rows, and the middle one has one singular row, but the rightmost one has no singular row. 
    \begin{figure}
        \centering
        \begin{equation}
	    \begin{tikzpicture}
	           \node[left] at (0,0) {$\yng(3,2)$};
	           \node[right] at (.1,.2) {$\mathbf{-3}$};
	           \node[right] at (.1,-.25) {$\mathbf{-1}$};
                \node[left] at (-1.5,.2) {$-3$};
                \node[left] at (-1.5,-.25) {$-1$};
                \node[left] at (4,0) {$\yng(3,2)$};
	           \node[right] at (4.1,.2) {$-3$};
	           \node[right] at (4.1,-.25) {$\mathbf{1}$};
                \node[left] at (2.5,.2) {$-1$};
                \node[left] at (2.5,-.25) {$1$};
                \node[left] at (8,0) {$\yng(3,2)$};
	           \node[right] at (8.1,.2) {$-3$};
	           \node[right] at (8.1,-.25) {$-1$};
                \node[left] at (6.5,.2) {$0$};
                \node[left] at (6.5,-.25) {$2$};
	    \end{tikzpicture}
	   \end{equation}
        \caption{Some examples of partitions with riggings and vacancies. We note that the leftmost rigged configuration corresponds to $\eta = 1110011$, the middle one to $\eta = 111000011$ and the rightmost one to $\eta = 1110001100$. We highlighted singular rows writing the corresponding rigging in boldface.}
        \label{ex:singular_3}
    \end{figure}}
	
	To construct the sequence $(\mu(x),\mathbf{J}(x))$ we set $\mu(0)= \varnothing$ and $\mathbf{J}(0) = \varnothing$. Assuming that we have determined $(\mu(x),\mathbf{J}(x))$, we will construct $(\mu(x+1),\mathbf{J}(x+1))$ as a function of $\eta(x+1)$. If $\eta(x+1) = 0$, then we set $(\mu(x+1),\mathbf{J}(x+1)) = (\mu(x),\mathbf{J}(x))$. On the other hand, if $\eta(x+1)=1$, we look for the singular row $(\mu(x),J_{k,{m_{k}}}(x))$ of $(\mu(x),\mathbf{J}(x))$ of maximal length $k$. Then we replace such row with a singular row of length $k+1$. If there are no singular rows, then we simply create a singular row of length 1. Since we assume that $\eta$ has only finitely many balls, it is clear that from a certain $x$ onward $(\mu(x), \mathbf{J}(x))$ stabilizes and the result is the desired rigged configuration $(\mu, \mathbf{J})$. Note that even for general $\eta \in \Omega$, $(\mu(x), \mathbf{J}(x))$ is well-defined for any $x \in \Z_{\ge 0}$, but it may not stabilize. One can easily see that, starting from a rigged configuration $(\mu, \mathbf{J})$, it is possible to compute the algorithm just described in reverse and associate uniquely a ball configuration $\eta \in \Omega_{< \infty}$.  
	
	In Fig. \ref{fig:ex_rig} we show the computation of the KKR bijection relating the ball configuration $\eta = 11001110110001100000\dots$ 
	with the rigged configuration
	\begin{equation}
	    \begin{tikzpicture}
	           \node[left] at (0,0) {$\yng(4,2,2,1)$};
	           \node[right] at (.1,.6) {$-4$};
	           \node[right] at (.2,.15) {$1$};
	           \node[right] at (.1,-.3) {$-2$};
	           \node[right] at (.2,-.75) {$3.$};
	    \end{tikzpicture}
	\end{equation}
	For further examples and for generalizations of the KKR bijection we invite the reader to consult \cite{IKT} and references therein.
	
	\medskip
 
    \begin{figure}[h]
        \includegraphics[width=\linewidth]{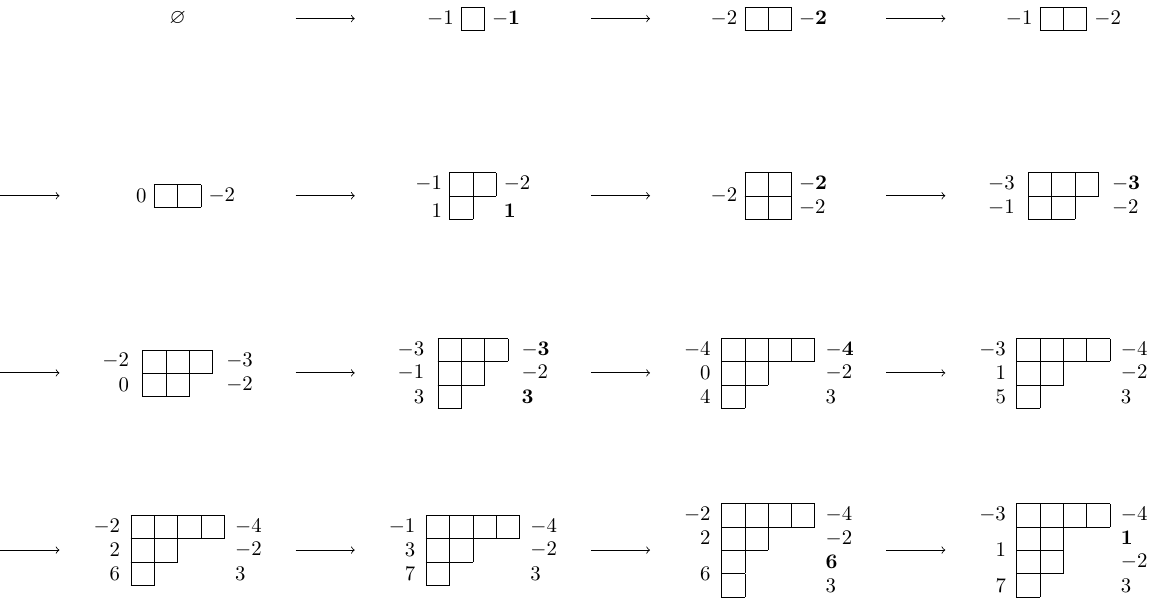}
        \caption{The computation of the KKR bijection corresponding to the configuration $\eta = 110011101100011000\dots$. Integers at the left of each partition represent the vacancies $p_k(x)$ associated {with} each group of rows of the same length $k$, while the integers at the right of each diagram represent the rigging $\mathbf{J}$. We have highlighted singular rows by writing the corresponding component of the rigging in boldface.} \label{fig:ex_rig}
    \end{figure} 	

    The KKR bijection is extremely important in the study of the BBS as in the rigged configuration, the dynamics becomes linear. The following proposition recalls a result from \cite{KOSTY}.
    
    \begin{theorem}[\cite{KOSTY}]\label{thm:KOSTY}
    Let $\eta \in \Omega_{< \infty}$ be a configuration associated with the rigged configuration $(\mu,\mathbf{J})$ under {the} KKR bijection. Then for any $k \in \N$ and $\ell \in \N \cup \left\{ \infty \right\}$, we have {$T\mu = \mu$, and}
    \begin{align}
        (T_{\ell}\mathbf{J})_{k} = \left( J_{k,j} + k \wedge \ell ; 1 \le j \le m_k \right).
	\end{align} 
	\end{theorem}
	
	\begin{remark}\label{rem:localenergy}
        In \cite{IKT} relations between rigged configurations and BBS were explained through the formalism of the theory of crystals. In this formalism the energy function $E_k$ can be expressed as a certain sum over a more refined quantity called \emph{local energy} which is function of a tensor product of two crystals \cite{FOY}.
        For our model, such local energy can be given in term of the function $\tilde{H} : \{0,1 \} \times B \to \{0,1 \}$, $B := \{(k,\ell) ; k \in \N, 0 \le \ell \le k \}$, given by
            \begin{align}
                \tilde{H}(a,(k,\ell)) := \min\left\{ a, k - \ell  \right\},
            \end{align}    
        using which we can represent $E_k$ as
            \begin{align}
                E_{k}(x) = \sum_{y = 1}^{x} \tilde{H}\left(\eta(y),(k,W_{k}(y-1)) \right).
            \end{align}
        Here recall that $W_k$ is the carrier with capacity $k$. There is a direct relation between $\tilde{H}$ and seat numbers. Actually, from the above representation of $E_{k}$ and \eqref{eq:seat_energy} we obtain
            \begin{align}
                \tilde{H}\left(\eta(x),(k,W_{k}(x-1)) \right) = \sum_{\ell = 1}^{k} \eta^{\uparrow}_{\ell}(x).
            \end{align}
        From this relation, we can also deduce that values {of} $\eta^{\uparrow}_{k}(x)$, or rather their sums, represent the local energy of the BBS.  
    \end{remark}

\subsection{A pair of interlacing Young diagrams with riggings}\label{sec:algo}  
In this subsection, we introduce a new algorithm to obtain a sequence of pairs of Young diagrams from a ball configuration $\eta \in \Omega$. Then, we prove that the rigged configuration obtained by the KKR bijection is understood as a projection {from} the pair of Young diagrams to the first component. 

For a pair of Young diagrams $(\mu^{\uparrow}, \mu^{\downarrow})$, we say that the pair is interlacing if 
\[
\mu^{\uparrow}_1 \ge \mu^{\downarrow}_1 \ge \mu^{\uparrow}_2 \ge \mu^{\downarrow}_2 \ge \dots
\]
holds, where we use the convention that for a given partition $\mu$, $\mu_{i} = 0$ for any $i > \lambda_{1}$. The interlacing condition is equivalent to
\begin{equation}\label{eq:cond-alter}
    \lambda^{\uparrow}_k - \lambda^{\downarrow}_k \in \{0,1\},
\end{equation}
for any $k \ge 1$ where $\lambda^{\sigma}$ is the conjugate of $\mu^{\sigma}$ for $\sigma \in \{ \uparrow, \downarrow \}$. 

We introduce 
\begin{align}
    k^{\uparrow}& :=\sup \{ k \ge 1  \ ;  \ \lambda^{\uparrow}_{\ell} - \lambda^{\downarrow}_{\ell}=1, \ 1 \le \forall \ell \le k \} \\
    & =\sup \{ k \ge 1  \ ; \ \sum_{\ell=1}^k(\lambda^{\uparrow}_{\ell} - \lambda^{\downarrow}_{\ell})=k \}
\end{align}
and 
\begin{align}
    k^{\downarrow}& :=\sup \{ k \ge 1  \ ;  \ \lambda^{\uparrow}_{\ell} - \lambda^{\downarrow}_{\ell}=0, \ 1 \le \forall \ell \le k \} \\
    & =\sup \{ k \ge 1  \ ; \ \sum_{\ell=1}^k(\lambda^{\uparrow}_{\ell} - \lambda^{\downarrow}_{\ell})=0 \}
\end{align}
with convention $\sup \emptyset = 0$. Note that $ k^{\uparrow} < \infty$ for any pair of interlacing Young diagrams $(\mu^{\uparrow}, \mu^{\downarrow})$ but $k^{\downarrow}=\infty$ if and only if $\mu^{\uparrow} = \mu^{\downarrow}$.

The next lemma is rather straightforward but we state it for clarity.

\begin{lemma}\label{lem:alt2}
Take a pair of interlacing Young diagrams $(\mu^{\uparrow}, \mu^{\downarrow})$. For $\sigma \in \{ \uparrow, \downarrow \}$, let $\tilde{\mu}^{\sigma}$ be the Young diagram obtained from $\mu^{\sigma}$ by adding a box to one of the row(s) satisfying $\mu_i^{\sigma}=k^{\sigma}$. In other words, we replace a row with length $k^{\sigma}$ by a row with length $k^{\sigma}+1$. Here, if $k^{\sigma}=0$, then we simply add a row with length one to $\mu^{\sigma}$. Then, the pair $(\tilde{\mu}^{\uparrow}, \mu^{\downarrow})$ is still a pair of interlacing Young diagrams. Also, if $\mu^{\uparrow} \neq \mu^{\downarrow}$, the pair $(\mu^{\uparrow}, \tilde{\mu}^{\downarrow})$ is still a pair of interlacing Young diagrams. 
\end{lemma}
\begin{proof}
    Since $\tilde{\lambda}^{\sigma}_{k}=\lambda^{\sigma}_k$ for $k \neq k^{\sigma}+1$ and $\tilde{\lambda}^{\sigma}_{k^{\sigma}+1}=\lambda^{\sigma}_{k^{\sigma}+1} +1$, the condition \eqref{eq:cond-alter} is satisfied by the definition of $k^{\sigma}$. 
\end{proof}

Now, for a given $\eta \in \Omega$, we construct a growing sequence of pairs of interlacing Young diagrams $(\mu^{\uparrow}(x), \mu^{\downarrow}(x))_{x \in \Z_{\ge 0}}$. 

Set $\mu^{\sigma}(0)=\emptyset$ for $\sigma=\uparrow, \downarrow$. For $x \ge 0$, we construct  $\mu^{\uparrow}(x+1), \mu^{\downarrow}(x+1) $ as a function of $\mu^{\uparrow}(x), \mu^{\downarrow}(x)$ and $\eta(x+1)$ by the algorithm explained below.

\begin{enumerate}
    \item If $\eta(x+1)=1$, then $\mu^{\downarrow}(x+1)=\mu^{\downarrow}(x)$ and $\mu^{\uparrow}(x+1)$ is obtained by adding a box to $\mu^{\uparrow}(x)$ at one of the row(s) satisfying $\mu_i^{\uparrow}(x)=k^{\uparrow}(x)$. In other words,
\[
\lambda_k^{\uparrow}(x+1)=\lambda_k^{\uparrow}(x)
\]
for any $k \neq k^{\uparrow}(x)+1$ and 
\[
\lambda_{k^{\uparrow}(x)+1}^{\uparrow}(x+1)=\lambda^{\uparrow}_{k^\uparrow(x)+1}(x)+1.
\]
Then, by Lemma \ref{lem:alt2}, $(\mu^{\uparrow}(x+1), \mu^{\downarrow}(x+1))$ is also a pair of interlacing Young diagrams. 
    
    \item If $\eta(x+1) =0$ and $\mu^{\uparrow}(x) \neq \mu^{\downarrow}(x)$, then $\mu^{\uparrow}(x+1)=\mu^{\uparrow}(x)$ and $\mu^{\downarrow}(x+1)$ is obtained by adding a box to $\mu^{\downarrow}(x)$ at one of the row(s) satisfying $\mu_i^{\downarrow}(x)=k^{\downarrow}(x)$ as for the first case. Then, by Lemma \ref{lem:alt2} again, $(\mu^{\uparrow}(x+1), \mu^{\downarrow}(x+1))$ is also a pair of interlacing Young diagrams.
    
    \item If $\eta(x+1) =0$ and $\mu^{\uparrow}(x)=\mu^{\downarrow}(x)$, then we set $(\mu^{\uparrow}(x+1), \mu^{\downarrow}(x+1))= (\mu^{\uparrow}(x), \mu^{\downarrow}(x))$.  
\end{enumerate}

\begin{figure}
    \centering
    \begin{tikzpicture}[scale=0.6]
        \node at (0,-1) {$\emptyset$};
        \draw[->] (2,-1)--(3,-1);
        \node at (5,-1) {$\young(\uparrow)$};
        \draw[->] (7,-1)--(8,-1);
        \node at (10,-1) {$\young(\uparrow\uparrow)$};
        \draw[->] (12,-1)--(13,-1);
        \node at (15,-1) {$\young(\updownarrows\uparrow)$};
        \draw[->] (-3,-4)--(-2,-4);
        \node at (0,-4) {$\young(\updownarrows\updownarrows)$};
        \draw[->] (2,-4)--(3,-4);
        \node at (5,-4) {$\young(\updownarrows\updownarrows,\uparrow)$};
        \draw[->] (7,-4)--(8,-4);
        \node at (10,-4) {$\young(\updownarrows\updownarrows,\uparrow\uparrow)$};
        \draw[->] (12,-4)--(13,-4);
        \node at (15,-4) {$\young(\updownarrows\updownarrows\uparrow,\uparrow\uparrow)$};
        \draw[->] (-3,-8)--(-2,-8);
        \node at (0,-8) {$\young(\updownarrows\updownarrows\uparrow,\updownarrows\uparrow)$};
        \draw[->] (2,-8)--(3,-8);
        \node at (5,-8) {$\young(\updownarrows\updownarrows\uparrow,\updownarrows\uparrow,\uparrow)$};
        \draw[->] (7,-8)--(8,-8);
        \node at (10,-8) {$\young(\updownarrows\updownarrows\uparrow\uparrow,\updownarrows\uparrow,\uparrow)$};
        \draw[->] (12,-8)--(13,-8);
        \node at (15,-8) {$\young(\updownarrows\updownarrows\uparrow\uparrow,\updownarrows\uparrow,\updownarrows)$};
        \draw[->] (-3,-12)--(-2,-12);
        \node at (0,-12) {$\young(\updownarrows\updownarrows\uparrow\uparrow,\updownarrows\updownarrows,\updownarrows)$};
        \draw[->] (2,-12)--(3,-12);
        \node at (5,-12) {$\young(\updownarrows\updownarrows\updownarrows\uparrow,\updownarrows\updownarrows,\updownarrows)$};
        \draw[->] (7,-12)--(8,-12);
        \node at (10,-12) {$\young(\updownarrows\updownarrows\updownarrows\uparrow,\updownarrows\updownarrows,\updownarrows,\uparrow)$};
        \draw[->] (12,-12)--(13,-12);
        \node at (15,-12) {$\young(\updownarrows\updownarrows\updownarrows\uparrow,\updownarrows\updownarrows,\updownarrows\uparrow,\uparrow)$};
        \draw[->] (-3,-16)--(-2,-16);
        \node at (0,-16) {$\young(\updownarrows\updownarrows\updownarrows\uparrow,\updownarrows\updownarrows,\updownarrows\uparrow,\updownarrows)$};
        \draw[->] (2,-16)--(3,-16);
        \node at (5,-16) {$\young(\updownarrows\updownarrows\updownarrows\uparrow,\updownarrows\updownarrows,\updownarrows\updownarrows,\updownarrows)$};
        \draw[->] (7,-16)--(8,-16);
        \node at (10,-16) {$\young(\updownarrows\updownarrows\updownarrows\updownarrows,\updownarrows\updownarrows,\updownarrows\updownarrows,\updownarrows)$};
        \draw[->] (12,-16)--(13,-16);
        \node at (15,-16) {$\dots\dots$};
    \end{tikzpicture}
    \caption{Construction of the pair of interlacing Young diagrams from the ball configuration $\eta = 110011101100011000\dots$. In this figure, $(\mu^{\uparrow}(x), \mu^{\downarrow}(x))$ are superimposed. The symbols $\uparrow$ and $\downarrow$ indicate the shape of Young diagram $\mu^{\uparrow}(x)$ and $\mu^{\downarrow}(x)$, respectively, and $\updownarrows$ indicates the overlapped area.}
    \label{fig:ex_altyoung}
\end{figure}

In Figure \ref{fig:ex_altyoung}, an example of the process of construction of $(\mu^{\uparrow}(x), \mu^{\downarrow}(x))$ is shown. Now we observe that by comparing Figures \ref{fig:seatnumber_conf} and \ref{fig:ex_altyoung}, the relation $\lambda^{\uparrow}_{k}(x) - \lambda^{\downarrow}_{k}(x) = \mathcal{W}_{k}(x)$ holds for our working example. Actually, we can prove the same relation for any ball configuration $\eta \in \Omega$ as follows. 

\begin{proposition}\label{prop:alt}
Suppose that $\eta \in \Omega$. Then, the following relation between the seat-numbers and the pair of interlacing Young diagrams
\begin{align}\label{eq:lambda_seat}
    \lambda_k^{\sigma}(x)=\sum_{y=1}^x \eta^{\sigma}_{k}(y),    
\end{align}
holds for any $\sigma=\uparrow, \downarrow$, $x \ge 0$ and $k \in \N$.
In particular, we have
    \begin{align}\label{eq:lambda_seat2}
        \mathcal{W}_{\ell}(x) = \lambda^{\uparrow}_{\ell}(x)-\lambda^{\downarrow}_{\ell}(x),
    \end{align}
and
    \begin{align}\label{eq:lambda_seat3}
        m^{\sigma}_{k}(x) = \lambda_k^{\sigma}(x) - \lambda_{k+1}^{\sigma}(x),
    \end{align}
    where $m^{\sigma}_{k}(x)$ is defined in \eqref{eq:updowndiff}. In addition, 
    \begin{align}
        k^{\uparrow}(x) &= \sup \{k \ge 1 \ ; \ \mathcal{W}_{\ell}(x)=1, \text{ for all} \ 1 \le \ell \le k\}, \label{eq:lambda_seat4} \\
        k^{\downarrow}(x) &= \sup \{k \ge 1 \ ; \ \mathcal{W}_{\ell}(x)=0, \textrm{ for all} \ 1 \le  \ell \le k\}.
    \end{align}
\end{proposition}
\begin{proof}
    We prove this by induction on $x$. The statement clearly holds if $x=0$. Then, suppose \eqref{eq:lambda_seat} holds for some $x \in \Z_{\ge 0}$ for any $\sigma=\uparrow, \downarrow$ and $k \in \N$. Then, 
    \[
    \lambda_k^{\uparrow}(x) - \lambda_k^{\downarrow}(x)  =\sum_{y=1}^x\eta^{\uparrow}_{k}(y) - \sum_{y=1}^x\eta^{\downarrow}_{k}(y) = \mathcal{W}_k(x).
    \]
   Now suppose $\eta(x+1)=1$. Then, by this characterization and the definition of the seat-number configuration, $\eta^{\uparrow}_{k^{\uparrow}(x)+1}(x+1)=1$. This implies
   \[
 \sum_{y=1}^{x+1}\eta^{\sigma}_{k}(y) -\sum_{y=1}^{x}\eta^{\sigma}_{k}(y) = \mathbf{1}_{\{k= k^{\uparrow}(x)+1, \sigma=\uparrow\}}.
\]
On the other hand, by construction of the sequence of interlacing Young diagrams,
\begin{equation} \label{eq:difference_lambda_x+1_x}
    \lambda_k^{\sigma}(x+1) -\lambda_k^{\sigma}(x) = \mathbf{1}_{\{k= k^{\uparrow}(x)+1, \sigma=\uparrow\}}.
\end{equation}
Hence, by combining \eqref{eq:difference_lambda_x+1_x} with the induction assumption, the equality
\[
\lambda_k^{\sigma}(x+1) = \sum_{y=1}^{x+1}\eta^{\sigma}_{k}(y)
\]
holds for any $\sigma$ and $k$ if $\eta(x+1)=1$. For the case $\eta(x+1)=0$, $r(x+1)=1$ if $W_{\infty}(x)=0$ and $\eta^{\downarrow}_{k^{\downarrow}(x)+1}(x+1)=1$ if $W_{\infty}(x) \neq 0$. Noting that $\mu^{\uparrow}(x)=\mu^{\downarrow}(x)$ is equivalent to $k^{\downarrow}(x)=\infty$, and so to $W_{\infty}(x)=0$, we can also prove the result in the same manner. 
\end{proof}

\begin{remark}\label{rem:pair_seat}
    Note that from \eqref{eq:lambda_seat}, if we obtain the seat number configuration of $\eta \in \Omega$, then the sequence of pairs of interlacing Young diagrams corresponding to $\eta$ can be obtained via the following simple rules. Assume that we have constructed $\left(\mu^{\uparrow}(x),\mu^{\downarrow}(x) \right)$. 
        \begin{enumerate}
            \item[(1')] If $\eta^{\uparrow}_{k}{(x+1)} = 1$ for some $k \in \N$, then $\mu^{\downarrow}(x+1) = \mu^{\downarrow}(x)$ and $\mu^{\uparrow}(x+1)$ is obtained by adding a box to $\mu^{\uparrow}(x)$ at one of the {rows} with length $k-1$. 
            \item[(2')] If $\eta^{\downarrow}_{k}{(x+1)} = 1$ for some $k \in \N$, then $\mu^{\uparrow}(x+1) = \mu^{\uparrow}(x)$ and $\mu^{\downarrow}(x+1)$ is obtained by adding a box to $\mu^{\downarrow}(x)$ at one of the {rows} with length $k-1$.
            \item[(3')] If $r{(x+1)} = 1$, then we set $\left(\mu^{\uparrow}(x+1),\mu^{\downarrow}(x+1) \right) = \left(\mu^{\uparrow}(x),\mu^{\downarrow}(x) \right)$.
        \end{enumerate}
    For example, by following the above rules, we obtain the same sequence of the pairs of the interlacing Young {diagrams} in Figure \ref{fig:ex_altyoung} from the seat number configuration of $\eta = 110011101100011000\dots$ shown in Figure \ref{fig:ex_altyoung}. 
\end{remark}

Next, to reveal the relation with the original KKR bijection, we introduce the $k$-th $\sigma$ energy $E^{\sigma}_k(x)$ and the $k$-th $\sigma$ vacancy $p^{\sigma}_k(x)$ as
\[
E^{\sigma}_k(x) := \sum_{i \in \N} \min\{ \mu^{\sigma}_{i}(x), k\}, \quad 
p^{\sigma}_k(x) := x - 2 E^{\sigma}_k(x).
\]
Since
\[
\sum_{i \in \N} \min\{ \mu^{\sigma}_i(x), k\} = \sum_{i \in \N} \sum_{\ell=1}^k \mathbf{1}_{\{\mu^{\sigma}_i(x) \ge \ell\}} = \sum_{\ell=1}^k \lambda_{\ell}^{\sigma}(x),
\]
by Proposition \ref{prop:alt}, we can also rewrite
\[
E^{\sigma}_k(x) =\sum_{\ell=1}^k  \sum_{y=1}^x\eta^{\sigma}_{k}(y)
\]
and
\begin{align}
    p^{\sigma}_k(x) & = x - 2 \sum_{\ell=1}^k  \sum_{y=1}^x\eta^{\sigma}_{k}(y) = x -  \sum_{\ell=1}^k  \sum_{y=1}^x(\eta^{\sigma}_{k}(y) + \eta^{\check{\sigma}}_{k}(y)) - \sum_{\ell=1}^k  \sum_{y=1}^x(\eta^{\sigma}_{k}(y) - \eta^{\check{\sigma}}_{k}(y))\\
     & =\xi_k(x) - \sum_{\ell=1}^k \mathcal{W}^{\sigma}_{\ell}(x)
\end{align}
where $\check{\sigma}$ is the opposite arrow to $\sigma$,  $\mathcal{W}^{\uparrow}_k=\mathcal{W}_k$ and $\mathcal{W}^{\downarrow}_k=-\mathcal{W}_k$. 

\medskip

The following property of the $k$-th $\sigma$ vacancy is useful {for understanding} the relation between the seat number configuration and the riggings.

\begin{lemma}\label{lem:alt3}
Consider a ball configuration $\eta \in \Omega$. Then, the following statements hold.

\begin{enumerate}
    \item Suppose $\eta^{\uparrow}_k(x)=1$. Then 
    \begin{equation} \label{eq:p_arrow_xi}
        p^{\uparrow}_k(x) =\xi_k(x)-k
    \end{equation}
    and for any $x' \ge x$, $p^{\uparrow}_k(x) \le p^{\uparrow}_k(x')$. Moreover, for $x' \ge x$, $p^{\uparrow}_k(x) = p^{\uparrow}_k(x')$ holds if and only if 
    \[
    \sum_{y=x+1}^{x'}\sum_{\ell \ge k+1}(\eta^{\uparrow}_{\ell}(y) + \eta^{\downarrow}_{\ell}(y)) +\sum_{y=x+1}^{x'}r(y) =0 \quad \text{and} \quad \sum_{\ell=1}^k \mathcal{W}_{\ell}(x')=k.
    \]
    
    \item Suppose $\eta^{\downarrow}_k(x)=1$. Then 
    \[
        p^{\downarrow}_k(x) =\xi_k(x)
    \]
    and for any $x' \ge x$, $p^{\downarrow}_k(x) \le p^{\downarrow}_k(x')$.
    Moreover, for $x' \ge x$, $p^{\downarrow}_k(x) = p^{\downarrow}_k(x')$ holds if and only if 
    \[
        \sum_{y=x+1}^{x'}\sum_{\ell \ge k+1}(\eta^{\uparrow}_{\ell}(y) + \eta^{\downarrow}_{\ell}(y)) +\sum_{y=x+1}^{x'}r(y) =0 \quad \text{and} \quad \sum_{\ell=1}^k \mathcal{W}_{\ell}(x')=0.
    \]
\end{enumerate}
\end{lemma}

\begin{proof}
Let us only show (1), as the proof (2) is completely analogous. Since $\eta^{\uparrow}_k(x)=1$ implies $\mathcal{W}_{\ell}(x)=1$ for all $1 \le \ell \le k$, \eqref{eq:p_arrow_xi} holds. Let $x' \ge x$. Then,
\[
p^{\uparrow}_k(x') - p^{\uparrow}_k(x) = \xi_k(x') -\xi_k(x) + k-\sum_{\ell=1}^k\mathcal{W}_{\ell}(x')
\]
and $\xi_k(x') -\xi_k(x) \ge 0$, $k-\sum_{\ell=1}^k\mathcal{W}_{\ell}(x') \ge 0$ implies the inequality $p^{\uparrow}_k(x) \le p^{\uparrow}_k(x')$. The last condition in the statement is equivalent to $\xi_k(x') -\xi_k(x)=0$ and $k-\sum_{\ell=1}^k\mathcal{W}_{\ell}(x')=0$, which is obviously equivalent to $p^{\uparrow}_k(x) = p^{\uparrow}_k(x')$.
\end{proof}

We now introduce {\it refined} riggings $\mathbf{J}^{\sigma}(x)=(J_k^{\sigma}(x), \ 1 \le k  \le \mu^{\sigma}_1(x))$ such that $J^{\sigma}_k(x)=(J^{\sigma}_{k,j}(x), \ 1 \le j \le m^{\sigma}_k(x))$ where $m^{\sigma}_k(x)=\lambda^{\sigma}_k(x)-\lambda^{\sigma}_{k+1}(x) =|\{i \ ; \mu^{\sigma}_i(x) =k\}|$ from Proposition \ref{prop:alt}. We order them as 
\[
J^{\sigma}_{k,1}(x) \le J^{\sigma}_{k,2}(x) \dots \le J^{\sigma}_{k,m^{\sigma}_k(x)}(x)
\]
to make the notation simple in the {following} argument.
We define the riggings recursively, although later in the same subsection, we will show that they can be defined more directly in terms of seat numbers. 

Let $\mathbf{J}^{\sigma}(0)=\emptyset$ for $\sigma =\uparrow, \downarrow$. We will construct $\mathbf{J}^{\sigma}(x+1)$ from $\mathbf{J}^{\sigma}(x)$ by considering three cases separately. 

\underline{Case 1 : $\eta_k^{\sigma}(x+1)=0$ for all $k$.} If $\eta_k^{\sigma}(x+1)=0$ for all $k$, or equivalently $\mu^{\sigma}(x+1)=\mu^{\sigma}(x)$, then we also keep the rigging as $\mathbf{J}^{\sigma}(x+1)=\mathbf{J}^{\sigma}(x)$. 

\underline{Case 2 : $\eta_1^{\sigma}(x+1)=1$.} If $\eta_1^{\sigma}(x+1)=1$, or equivalently a row of length $1$ is added to obtain $\mu^{\sigma}(x+1)$ from $\mu^{\sigma}(x)$, we append the value $p^{\sigma}_{1}(x+1)$ to $J^{\sigma}_{1}(x)$ and obtain $J^{\sigma}_{1}(x+1)$. More precisely,
\[
J^{\sigma}_{\ell}(x+1)=J^{\sigma}_{\ell}(x) \quad \ell \neq 1
\]
and
\[
J^{\sigma}_{1}(x+1) = (J^{\sigma}_{1,1}(x), \dots J^{\sigma}_{1,m^{\sigma}_{1}(x)}(x) , p^{\sigma}_{1}(x+1)).
\]

\underline{Case 3 : $\eta_{k+1}^{\sigma}(x+1)=1$ for some $k \ge 1$. } If $\eta_{k+1}^{\sigma}(x+1)=1$ for some $k \ge 1$, or equivalently a row of length $k$ is replaced by one of length $k+1$ to obtain $\mu^{\sigma}(x+1)$ from $\mu^{\sigma}(x)$, we remove the largest entry from $J^{\sigma}_k(x)$ and append $p^{\sigma}_{k+1}(x+1)$ to $J^{\sigma}_{k+1}(x)$ to obtain rigging $\mathbf{J}^{\sigma}(x+1)$. More precisely,
\[
J^{\sigma}_{\ell}(x+1)=J^{\sigma}_{\ell}(x) \quad \ell \neq k, k+1,
\]
\[
J^{\sigma}_{k}(x+1)= (J^{\sigma}_{k,1}(x), \dots, J^{\sigma}_{k,m^{\sigma}_k(x)-1}(x))
\]
and
\[
J^{\sigma}_{k+1}(x+1) =(J^{\sigma}_{k+1,j}(x), \dots, J^{\sigma}_{k+1,m^{\sigma}_{k+1}(x)}(x),  p^{\sigma}_{k+1}(x+1)).
\]

\medskip

Now we analyze properties of this newly defined rigging. By the way of construction, it is clear that for any  $J^{\sigma}_{k,j}(x)$ in the rigging $\mathbf{J}^{\sigma}(x)$, there exists $y \le x$ such that $\eta^{\sigma}_k(y)=1$ and $J^{\sigma}_{k,j}(x)=p^{\sigma}_k(y)$, which is not necessarily unique.  
In the next proposition, we give an explicit expression of one of such $y=y(x,k,j,\sigma)$.
For $1 \le j \le m^{\sigma}_k(x)$, let 
\[
t_k^{\sigma}(x,j):=\max \{ y \le x  \ ; \ m^{\sigma}_k(y)=j, \eta^{\sigma}_k(y)=1\}.
\]
Since $|m^{\sigma}_k(y+1)-m^{\sigma}_k(y)| \le 1$ for any $y$, $m^{\sigma}_k(0)=0$ and $m^{\sigma}_k(y)$ increases only when $\eta^{\sigma}_k(y)=1$, the above set is not empty, namely $1 \le t_k^{\sigma}(x,j) \le x$. Moreover, $t_k^{\sigma}(x,1) < t_k^{\sigma}(x,2 ) < \cdots < t_k^{\sigma}(x,m^{\sigma}_k(x))$.

\begin{proposition}\label{prop:alt2}
For any $x \in \Z_{\ge 0}, k \in \N$, $\sigma \in \{\uparrow,\downarrow\}$ and $1 \le j \le m^{\sigma}_k(x)$, we have
\begin{equation} \label{eq:J_and_p}
    J^{\sigma}_{k,j}(x)=p^{\sigma}_k(t_k^{\sigma}(x,j)).
\end{equation}
\end{proposition}
\begin{proof}
    We prove \eqref{eq:J_and_p} by induction on $x$. For $x=0$, the equality trivially holds as there is no $j$ satisfying $1 \le j \le m^{\sigma}_k(x)$. Next, suppose \eqref{eq:J_and_p} holds for some $x \in \Z_{\ge 0}$ and for any $k, \sigma$ and $1 \le j \le m^{\sigma}_k(x)$. We prove that the same holds for $x+1$ by considering three cases separately as above.
    
    \underline{Case 1 : $\eta_k^{\sigma}(x+1)=0$ for all $k$.} Then by definition, $m^{\sigma}_k(x+1)=m^{\sigma}_k(x)$ and  $t_k^{\sigma}(x+1,j)=t^{\sigma}_k(x,j)$ for any $k, \sigma$ and $1 \le j \le m^{\sigma}_k(x+1)$. Also, $\mathbf{J}^{\sigma}(x+1)=\mathbf{J}^{\sigma}(x)$. Hence,
    \[
    J^{\sigma}_{k,j}(x+1)=J^{\sigma}_{k,j}(x) = p^{\sigma}_k(t_k^{\sigma}(x,j)) = p^{\sigma}_k(t_k^{\sigma}(x+1,j))
    \]
    holds for any $k, \sigma$ and $1 \le j \le m^{\sigma}_k(x+1)$.
    
    \underline{Case 2 : $\eta_1^{\sigma}(x+1)=1$.} If $\eta^{\sigma}_1(x+1)=1$, then as in Case 1, for any $\ell \neq 1$ and $1 \le j \le m^{\sigma}_\ell(x+1)$, 
    \[
    J^{\sigma}_{\ell,j}(x+1)=J^{\sigma}_{\ell,j}(x) =p^{\sigma}_\ell (t_\ell^{\sigma}(x,j)) = p^{\sigma}_\ell (t_\ell^{\sigma}(x+1,j))
    \]
    holds. On the other hand, for $\ell= 1$,  $m^{\sigma}_{1}(x+1)=m^{\sigma}_1(x)+1$ and  $t_1^{\sigma}(x+1,m^{\sigma}_{1}(x+1))=x+1$. Moreover, by Lemma \ref{lem:alt3}, for any $1\le j \le m^{\sigma}_1(x)$,
    \[
    J^{\sigma}_{1,j}(x) = p^{\sigma}_1(t_1^{\sigma}(x,j)) \le p^{\sigma}_1(x+1)
    \]
    since $\eta^{\sigma}_1(t_1^{\sigma}(x,j))=1$.
    Hence, as we order 
    \[
    J^{\sigma}_{k,1}(x+1) \le J^{\sigma}_{k,2}(x+1) \le \dots \le J^{\sigma}_{k,m^{\sigma}_1(x+1)}(x+1), 
    \]
    we have $J^{\sigma}_{1,j}(x+1)=J^{\sigma}_{1,j}(x)$ for any $1\le j \le m^{\sigma}_1(x)$ and 
    $J^{\sigma}_{1,m^{\sigma}_{1}(x+1)}=p^{\sigma}_1(x+1)$.
    Hence, for $j=m^{\sigma}_{1}(x+1)$,
       \[
    J^{\sigma}_{1,j}(x+1)=p^{\sigma}_1(x+1)= p^{\sigma}_1(t_1^{\sigma}(x+1,m^{\sigma}_{1}(x+1))) = p^{\sigma}_1(t_1^{\sigma}(x+1,j))
    \]
    holds. Also, for $1 \le j \le m^{\sigma}_1(x)$, by the definition of $t_1^{\sigma}(x,j)$, it is obvious that $t_1^{\sigma}(x+1,j)=t_1^{\sigma}(x,j)$ and so
    \[
    J^{\sigma}_{1,j}(x+1)=p^{\sigma}_1(t_1^{\sigma}(x+1,j))
    \]
    holds as well. 
    
    \underline{Case 3 : $\eta_{k+1}^{\sigma}(x+1)=1$ for some $k \ge 1$.} If $\eta^{\sigma}_{k+1}(x+1)=1$, then as in Case 1, for any $\ell \neq k, k+1$ and $1 \le j \le m^{\sigma}_\ell(x+1)$, 
    \[
    J^{\sigma}_{\ell,j}(x+1)= p^{\sigma}_\ell (t_\ell^{\sigma}(x+1,j))
    \]
    holds. Also, for $\ell=k+1$, the same relation holds by exactly the same argument as in Case 2. Finally, for $\ell=k$, $m^{\sigma}_{k}(x+1)=m^{\sigma}_k(x)-1$ and for any $1 \le j \le m^{\sigma}_{k}(x+1)$, we have  $t_k^{\sigma}(x+1,j)=t_k^{\sigma}(x,j)$, $J^{\sigma}_{k,j}(x+1)=J^{\sigma}_{k,j}(x)$ by definition. Hence, \[
    J^{\sigma}_{k,j}(x+1)=p^{\sigma}_k(t_k^{\sigma}(x+1,j))
    \]
    holds for any $1 \le j \le m^{\sigma}_{k}(x+1)$.
    This completes the proof. 
\end{proof}
Finally, we prove that $k^{\uparrow}$ and $k^{\downarrow}$ can be characterized in terms of the rigging and the traditional singular condition. We say $\mu_i^{\sigma}(x)$ is a singular row of $(\mu^{\sigma}(x), \mathbf{J}^{\sigma}(x))$ if $p^{\sigma}_k(x)=J^{\sigma}_{k,m^{\sigma}_{k}(x)}(x)$, where $k=\mu_i^{\sigma}(x)$. 
\begin{proposition}\label{prop:alt3}
Assume the convention that $\max\emptyset=0$. Then,
    \[
        k^{\uparrow}(x)= \max_i \{\mu^{\uparrow}_i(x) \ ; \ \mu^{\uparrow}_i(x)  \text{ is singular} \}.
    \]
Also, if $\mu^{\uparrow}(x) \neq \mu^{\downarrow}(x)$, then
\[
        k^{\downarrow}(x)= \max_i \{\mu^{\downarrow}_i(x) \ ; \ \mu^{\downarrow}_i(x)  \text{ is singular} \}.  
    \]
\end{proposition}
\begin{proof}
    If $\mu_i^{\sigma}(x)$ is singular, then for $k=\mu_i^{\sigma}(x)$, we have
    \[
    p^{\sigma}_k(x)=J^{\sigma}_{k,m^{\sigma}_{k}(x)}(x) = p^{\sigma}_k\left(t_k^{\sigma}\left(x,m^{\sigma}_{k}(x)\right)\right),
    \]
    where we apply Proposition \ref{prop:alt2} for the second equality. Then, since $x \ge t_k^{\sigma}\left(x,m^{\sigma}_{k}(x)\right)$ and 
    
    \noindent $\eta^{\sigma}_k\left(t_k^{\sigma}(x,m^{\sigma}_{k}(x)) \right)=1$, by Lemma \ref{lem:alt3}, we have $\sum_{\ell=1}^k \mathcal{W}_{\ell}(x)=k$ if $\sigma=\uparrow$ and $\sum_{\ell=1}^k \mathcal{W}_{\ell}(x)=0$ if $\sigma=\downarrow$. Thus from \eqref{eq:lambda_seat4}, we obtain 
    \[
        k^{\sigma}(x) \ge \max_i \{\mu^{\sigma}_i(x) \ ; \ \mu^{\sigma}_i(x)  \text{ is singular} \}.
    \]
    Hence, it is sufficient to prove that if  $1 \le  k^{\sigma}(x) < \infty$, then the row satisfying $\mu^{\sigma}_i(x)=k^{\sigma}(x)$ exists in $\mu^{\sigma}(x)$ and it is singular. In the rest of the proof, we prove this assertion. 
    
    Suppose $1 \le  k^{\sigma}(x) < \infty$. Observe that at least one row with length $k^{\sigma}$ exists in $\mu^{\sigma}(x)$. To simplify the notation, denote $k^{\sigma}(x)$ by $k^*$. Since $\lambda^{\sigma}_{k^*} (x) \ge 1$ and $\lambda^{\sigma}_{k^*} (x)=\sum_{y=1}^{x}\eta^{\sigma}_{k^*}(y)$, we define $x^*$ as the maximal $y$ satisfying $y \le x$ and $\eta^{\sigma}_{k^*}(y)=1$, or in formula
    \[
    x^* :=\max\{y \ ; \ y \le x \ , \ \eta^{\sigma}_{k^*}(y)=1 \}.
    \]
    From now on, we consider the case $\sigma=\uparrow$. Then, from Lemma \ref{lem:1} (i), for any $1 \le \ell \le k^*$, we have $\mathcal{W}_{\ell}(x^*)=1$. Also, from \eqref{eq:lambda_seat4},
    for any $1 \le \ell \le k^*$, $
    \mathcal{W}_{\ell}(x)=1$
         and $\mathcal{W}_{k^*+1}(x)=0$. Hence, we have
     \[
      \mathcal{W}_{k^*}(x) -  \mathcal{W}_{k^*}(x^*) =\sum_{y=x^*+1}^x (\eta_{k^*}^{\uparrow}(y) -\eta_{k^*}^{\downarrow}(y) )=0.
     \]
     On the other hand, by the construction of $x^*$,
        \begin{align}\label{eq:prop3_0}
            \sum_{y=x^*+1}^x \eta_{k^*}^{\uparrow}(y) =0.
        \end{align}
     Hence, $\sum_{y=x^*+1}^x \eta_{k^*}^{\uparrow}(y)= \sum_{y=x^*+1}^x \eta_{k^*}^{\downarrow}(y)=0$. This implies that the seat $k^*$ is occupied for any $y \in [x^*, x]$ and therefore
        \begin{align}\label{eq:prop3_1}
            \sum_{\ell \ge k^*+1}\sum_{y=x^*+1}^x \eta_{\ell}^{\downarrow}(y)=0,
        \end{align}
     since if a ball leaves a seat $\ell \ge k^*+1$, then the ball at seat $k^*$ must have already left. Then, from \eqref{eq:lem:1_1}, \eqref{eq:prop3_1} and $\mathcal{W}_{k^*+1}(x)=0$, we also have $\sum_{y=x^*+1}^x \eta_{k^*+1}^{\uparrow}(y)=0$. Namely, the seat $k^* +1$ is empty for any $y \in [x^*, x]$. This also implies that 
        \begin{align}\label{eq:prop3_2}
            \sum_{\ell \ge k^*+1}\sum_{y=x^*+1}^x \eta_{\ell}^{\uparrow}(y)=0.
        \end{align}
    Finally, since the seat $k^*$ is occupied for any $y \in [x^*, x]$, it is obvious that
        \begin{align}\label{eq:prop3_3}
            \sum_{y=x^*+1}^x r(y)=0.
        \end{align}
    Combining \eqref{eq:prop3_1},\eqref{eq:prop3_2} and \eqref{eq:prop3_3} we have
    \[
    \sum_{\ell \ge k^*+1}\sum_{y=x^*+1}^x( \eta_{\ell}^{\uparrow}(y) + \eta_{\ell}^{\downarrow}(y)) + \sum_{y=x^*+1}^x r(y)=0.
    \]
 Then, from \eqref{eq:lambda_seat4} and Lemma \ref{lem:alt3}, we have
    \[
    p_{k^*}^{\uparrow}(x^*)= p_{k^*}^{\uparrow}(x).
    \]
    Finally, we check that $x^*=t^{\uparrow}_{k^*}(x,m_{k^*}^{\uparrow}(x))$. For this, we only need to prove that $m^{\uparrow}_{k^*}(x^*)=m_{k^*}^{\uparrow}(x)$ and this is equivalent to $\sum_{y=x^*+1}^x \eta^{\uparrow}_{k^*}(y)= \sum_{y=x^*+1}^x \eta^{\uparrow}_{k^* +1}(y)$, which is true as this is $0$ as shown in \eqref{eq:prop3_0} and \eqref{eq:prop3_2}. Consequently, we have $J^{\uparrow}_{k^*,m_{k^*}^{\uparrow}(x)}(x)= p_{k^*}^{\uparrow}(t^{\uparrow}_{k^*}(x,m_{k^*}^{\uparrow}(x))) =  p_{k^*}^{\uparrow}(x)$, and so there exists at least one singular row with length $k^{*}$. 
    
    For the case $\sigma=\downarrow$, by using $\mu^{\uparrow}(x) \neq \mu^{\downarrow}(x)$, the exactly same argument works.
\end{proof}

\begin{remark}
    Proposition \ref{prop:alt3} allows us to give an intuitive meaning to the term of ``singular" for rigged configurations by means of the seat number configuration. Combining the above with Remark \ref{rem:pair_seat} and Proposition \ref{prop:seat_KKR}, we obtain an interpretation of the KKR bijection, which was a purely combinatorial object, in terms of the seat number configuration. 
\end{remark}

\subsection{Proof of Proposition \ref{prop:seat_KKR}}

In the last subsection, we have constructed the sequence of rigged Young diagrams $(\mu^{\uparrow}(x),\mathbf{J}^{\uparrow}(x))$ satisfying all the properties claimed in Proposition \ref{prop:seat_KKR} if we replace $(\mu(x),\mathbf{J}(x))$ by $(\mu^{\uparrow}(x),\mathbf{J}^{\uparrow}(x))$. Hence, we only need to prove that $(\mu(x),\mathbf{J}(x))=(\mu^{\uparrow}(x),\mathbf{J}^{\uparrow}(x))$. By Proposition \ref{prop:alt3}, we can construct $(\mu^{\uparrow}(x),\mathbf{J}^{\uparrow}(x))$ by the algorithm \textit{without using information {from}} $(\mu^{\downarrow}(x))_{x}$ to update as follows: Let $\mu^{\uparrow}(0)=\emptyset$ and $\mathbf{J}^{\uparrow}(0)=\emptyset$. Once $(\mu^{\uparrow}(x),\mathbf{J}^{\uparrow}(x))$ is given, we construct $(\mu^{\uparrow}(x+1),\mathbf{J}^{\uparrow}(x+1))$ as follows. If $\eta(x+1)=0$, we set $(\mu^{\uparrow}(x+1),\mathbf{J}^{\uparrow}(x+1))= (\mu^{\uparrow}(x),\mathbf{J}^{\uparrow}(x))$. If $\eta(x+1)=1$, then let $k:= \max\{\mu^{\uparrow}_i(x) \ : \ \mu^{\uparrow}_i(x) \text{ is singular} \ \}$ with convention $\max\emptyset=0$. If $k=0$, then add a row of length $1$ to $ \mu^{\uparrow}(x)$ and also add $p^{\uparrow}_1(x+1)$ to $J^{\uparrow}_{1}(x)$. If $k \ge 1$, then replace a row of length $k$ by that of $k+1$ and remove $J^{\uparrow}_{k,m_k^{\uparrow}(x)}(x)=p^{\uparrow}_k(x)$ from $J_{k}(x)$ and add $p^{\uparrow}_{k+1}(x+1)$ to $J^{\uparrow}_{k+1}(x)$. Note that {in} the original algorithm for the construction of $(\mu^{\uparrow}(x+1),\mu^{\downarrow}(x+1))$ from $\eta(x+1)$ and $(\mu^{\uparrow}(x),\mu^{\downarrow}(x))$, we used information {from} both Young diagrams, but instead we did not use the rigging. Here, we emphasize that the functions $m^{\uparrow}_k(x)$ and $p^{\uparrow}_k(x)=x-2E^{\uparrow}_k(x)$ can be obtained from $\mu^{\uparrow}(x)$ alone without information {from} $\mu^{\downarrow}(x)$, and this is also the case for the rigging $\mathbf{J}^{\uparrow}(x)$. Moreover, the last algorithm is exactly same as the one to construct $(\mu(x),\mathbf{J}(x))$ from $\eta$ by KKR bijection, which {confirms} that $(\mu(x),\mathbf{J}(x))=(\mu^{\uparrow}(x),\mathbf{J}^{\uparrow}(x))$ as desired. 

\begin{remark}
   Interestingly, the update algorithm is closed for $\sigma=\uparrow$, namely we can obtain $(\mu^{\uparrow}(x+1),\mathbf{J}^{\uparrow}(x+1))$ from the data $\eta(x+1)$ and $(\mu^{\uparrow}(x),\mathbf{J}^{\uparrow}(x))$, but it is not the case for $\sigma=\downarrow$. This is because, when $\eta(x+1)=0$, we should distinguish whether $\mu^{\uparrow}(x)=\mu^{\downarrow}(x)$ or not, or in other words if $r(x)=1$ or not, and this information cannot be derived from the data $(\mu^{\downarrow}(x),\mathbf{J}^{\downarrow}(x))$ only. If we also include an additional information, such as the total number of balls up to site $x$, that is $N(x):=\sum_{y=1}^x \eta(y)$, namely we consider the sequence $(\mu^{\downarrow}(x),\mathbf{J}^{\downarrow}(x),N(x) )$, then we can construct a local update algorithm which is closed. We may be able to show that $\mathbf{J}^{\downarrow}=\lim_{x \to \infty}\mathbf{J}^{\downarrow}(x) $ is also linearized under the BBS dynamics without using any relation to other linearizations. 
\end{remark}

\section{Relation to the slot decomposition}\label{sec:slot}

In this section, we first briefly recall the definition of the slot configuration and the corresponding slot decomposition introduced in \cite{FNRW}. Then, we give proofs of Proposition \ref{prop:seat_slot} and Theorem \ref{thm:2}. 
Note that for simplicity we will only consider finite ball configurations, but one can easily extend the definitions and results presented in this section to the configurations with an infinite number of records, {see, for example, \cite{S} for such an extension.}

    \subsection{Definition of the slot decomposition.}
    
    The notion of {\it slots} {was} originally introduced {in} \cite{FNRW}. Before defining the slots, we recall the fact that{, by the Takahashi-Satsuma algorithm (TS algorithm, see \cite{TS} or Appendix of this article),} any site of a given ball configuration $\eta \in \Omega_{< \infty}$ is either a record or a component of a soliton. Any $k$-soliton $\gamma \subset \N$ has the form $\gamma = \left\{ z(\gamma)_{1} < ... < z(\gamma)_{2k} \right\}$, where the coordinates are again identified by the TS algorithm. Then, the {\it slot configuration} $\nu : \N \to \Z_{\ge 0} \cup \left\{ \infty \right\} $ is defined as
	            \begin{align}
	                \nu(x) := \begin{dcases} l - 1 \quad & x = z(\gamma)_{l}, z(\gamma)_{l + k}  ~ \text{for some $k$-soliton $\gamma$ in $\eta$ with $k>\ell$}, \\
	                \infty \quad & x ~ \text{is a record},\end{dcases}
	            \end{align}
	for any $x \in \N$. For $k \in \N$, a site $x$ is called a \textbf{$k$-slot} if $\nu(x) \ge k$. Observe that a $k$-slot is also a $j$-slot for any $1 \le j \le k$, and a record is a $k$-slot for any $k \in \N$. Intuitively, a $k$-slot is a place where another soliton can be added without 
 modifying the structure of existing solitons in the configuration. To explain this better we define a way to ``append a soliton to a $k$-slot". 
	
	First, we define the function $\tilde{\xi}_{k} : \Z_{\ge 0} \to \Z_{\ge 0}$ as
	    \begin{align}\label{eq:defoftilde1}
	        \tilde{\xi}_{k}(x) &:= \sum_{y = 1}^x \mathbf{1}_{\left\{ \nu(y) \ge k \right\}}, \\
	        \tilde{\xi}_{k}(0) &:= 0
	    \end{align}
	for any $k \in \N$ and $x \in \Z_{\ge 0}$, which counts the number of $k$-slots in $[1,x]$. We number $k$-slots from left to right with the origin $x = 0$ as the $0$-th $k$-slot and call $\tilde{s}_{k}(i) := \min \left\{ x \in \Z_{\ge 0} \ ; \ \tilde{\xi}_{k}(x) = i \right\}$ the position of $i$-th $k$-slot. We say that a $k$-soliton $\gamma$ is appended to $\tilde{s}_{k}(i)$ if $\gamma \subset \left[\tilde{s}_{k}(i), \tilde{s}_{k}(i+1) - 1  \right]$. Note that several solitons can be appended to the same slot. By using this notion, for any $k \in \N$, we define the {\it slot decomposition} $\tilde{\zeta}_{k} : \Z_{\ge 0} \to \Z_{\ge 0}$ as 
	    \begin{align}
	        \tilde{\zeta}_{k}(i) &:= \left| \left\{ \gamma : k\text{-soliton in $\eta$ } ; \gamma ~ \text{is appended to the $i$-th $k$-slot} \right\} \right| \label{eq:defoftilde2}.
	    \end{align}
    \begin{figure}[H]
        \footnotesize
        \setlength{\tabcolsep}{5pt}
        \begin{center}
        \renewcommand{\arraystretch}{2}
        \begin{tabular}{rccccccccccccccccccc}
            $x$ & 1 &  2 &  3 &  4 &  5 &  6 & 7 &  8 & 9 & 10 & 11 & 12 & 13 & 14 & 15 & 16 & 17 & 18 & 19 \\
         \hline\hline $\eta(x)$ & \color{red}{1} & \color{red}{1} & \color{red}{0} & \color{red}{0} & \color{brown}{1} & \color{brown}{1} & \color{brown}{1} & \color{blue}{0} & \color{blue}{1} & \color{brown}{1} & \color{brown}{0} & \color{brown}{0} & \color{brown}{0} & \color{green}{1} & \color{green}{1} & \color{green}{0} & \color{green}{0} & \color{brown}{0} & 0  \\
           \hline
           \hline
            $\nu(x)$  & 0 & 1 & 0 & 1 & 0 & 1 & 2 & 0 & 0 & 3 & 0 & 1 & 2 & 0 & 1 & 0 & 1 & 3 & $\infty$ \\
           \end{tabular}
        \end{center}
        \caption{Slot configuration of $\eta = 1100111011000110000...$.}\label{fig:slotconf}
    \end{figure}
         In Figure \ref{fig:slotconf} we see an example of a slot configuration. For that particular ball configuration $\eta$ we have that $x=0$ is a record, $x=2$ is a 1-slot, $x=7$ is a 2-slot, etc. In the same example, solitons are added to slots as follows.
	            \begin{itemize}
	                \item A $4$-soliton is added to the $0$-th $4$-slot.
	                \item Two $2$-solitons are included. One is added to the $0$-th $2$-slot, and the other is added to the $3$-rd $2$-slot.
	                \item A $1$-soliton is added to the $4$-th $1$-slot.
	            \end{itemize}
	       Hence, the slot decomposition of $\eta$ is given by
	            \begin{align}
	                \tilde{\zeta}(\eta)_{k}(i) = \begin{dcases} 1 \quad &(k,i) = (1,4),(2,0),(2,3), (4,0) \\
	                0 \quad &\text{otherwise.}\end{dcases}
	            \end{align}
	        
	        \begin{remark}\label{rem:bijective}
	            Consider the following configuration spaces
	            \begin{align}
                \Omega_{r} &:= \left\{ \eta \in \Omega \ ; \ \sum_{x \in \N} r(x) = \infty \right\}, \\
                \tilde{\Omega}_{r} &:= \left\{ \tilde{\zeta} = \left(\tilde{\zeta}_{k}(i)\right)_{k \in \N, i \in \Z_{\ge 0}} \in \Z_{\ge 0}^{\N \times \Z_{\ge 0}} \  ;  \  \max\left\{ k \in \N \ : \ \tilde{\zeta}_{k}(i) > 0 \right\} < \infty \ \text{for any} \ i \right\}.
            \end{align}
	            It is known that the map $\eta \mapsto \tilde{\zeta}(\eta)$ is a bijection between $\Omega_{r}$ and $\tilde{\Omega}_{r}$ via the explicit reconstruction algorithm from $\tilde{\zeta}(\eta)$ to $\eta$ \cite{CS, FNRW}. By combining this fact with Proposition \ref{prop:seat_slot}, one can also reconstruct $\eta$ from $\left(\zeta_{k}(i)\right)_{k,i}$ by using the same algorithm. 
	        \end{remark}
	        
	        The dynamics of the BBS is linearized by the slot decomposition \cite{FNRW}. Actually, the slot decomposition makes the dynamics a mere spatial shift as described by the following theorem. 
	            \begin{theorem}[Theorem 1.4 in \cite{FNRW}]
        	        Suppose that $\eta \in \Omega_{< \infty}$. Then we have
        	            \begin{align}
        	                T\tilde{\zeta}_{k}(i) = \tilde{\zeta}_{k}(i - k)
        	            \end{align}
        	        for any $k \in \N$ and $i \in \Z_{\ge 0}$ where $\tilde{\zeta}_k (i)=0$ if $i <0$ by convention. 
	            \end{theorem}
            
    \subsection{Proof of Proposition \ref{prop:seat_slot}}
   In this subsection we prove Proposition \ref{prop:seat_slot}, which establishes the equivalence between the seat number configuration and the slot configuration. First we introduce an alternative formula for the slot decomposition : 
    \begin{lemma}\label{lem:repoftilde}
        Suppose that $\eta \in \Omega_{< \infty}$. Then for any $k \in \N$ and $i \in \Z_{\ge 0}$, we have
            \begin{align}
                \tilde{\zeta}_{k}(i) &= \frac{1}{2} \sum_{y = \tilde{s}_{k}(i) + 1}^{\tilde{s}_{k}(i+1) - 1} \mathbf{1}_{\{\nu(y) = k - 1\}} - \mathbf{1}_{\{\nu\left( \tilde{s}_{k}(i+1)\right) = k\}}  \\
	        &= \frac{1}{2} \sum_{y = \tilde{s}_{k}(i) + 1}^{\tilde{s}_{k}(i+1)} \left(\mathbf{1}_{\{\nu(y) = k - 1\}} - \mathbf{1}_{\{\nu(y) = k\}} \right).
            \end{align}
    \end{lemma}
    \begin{proof}[Proof of Lemma \ref{lem:repoftilde}]
        First we consider the case $\nu\left( \tilde{s}_{k}(i+1)\right) > k$. In this case, each {$(k-1)$-slot} in $\left(\tilde{s}_{k}(i), \tilde{s}_{k}(i+1)\right)$ is a component of a  $k$-soliton in $\left(\tilde{s}_{k}(i), \tilde{s}_{k}(i+1)\right)$, because if a {$(k - 1)$}-slot {was} a component of some $\ell$-soliton for $\ell>k$, then from the definition of $\nu$ and the TS algorithm, we should find a $k$-slot $x_{k} \in \left(\tilde{s}_{k}(i), \tilde{s}_{k}(i+1)\right)$ and thus we would have $\tilde{\xi}_{k}(x_{k}) = i + 1$, which contradicts the definition of $\tilde{s}_{k}(i+1)$. Hence,
        the number of $k$-solitons appended to {the} $i$-th $k$-slot is half the number of {$(k-1)$}-slots in $\left(\tilde{s}_{k}(i), \tilde{s}_{k}(i+1)\right)$. 
        
        Next we consider the case $\nu\left( \tilde{s}_{k}(i+1)\right) = k$. In this case, we will show that the rightmost $k-1$-slot in $\left(\tilde{s}_{k}(i), \tilde{s}_{k}(i+1)\right)$, denoted by $x_{k-1}$, and $\tilde{s}_{k}(i+1)$ are components of some $\ell$-soliton for $\ell>k$, denoted by $\gamma_{\ell}$. From the TS algorithm, there exists $y_{k-1} \in \left(\tilde{s}_{k}(i), \tilde{s}_{k}(i+1)\right)$ such that $\nu(y_{k-1}) = k - 1, \ \eta(y_{k-1}) = \eta(\tilde{s}_{k}(i+1))$, and $y_{k-1}$ is a component of $\gamma_{\ell}$. Then, again by the TS algorithm, we see that $x_{k-1} = y_{k-1}$, because if $y_{k - 1} < x_{k - 1}$, then $y_{k-1}$ is not a component of $\gamma_{\ell}$ but a component of $k$-soliton in $\left(\tilde{s}_{k}(i), \tilde{s}_{k}(i+1)\right)$, and this contradicts the definition of $y_{k-1}$. Therefore, the number of $k$-solitons appended to $i$-th $k$-slots is the same as the half of the number of $k-1$-slots in $\left(\tilde{s}_{k}(i), \tilde{s}_{k}(i+1)\right)$ minus one. 
    \end{proof}
    
        \begin{proof}[Proof of Proposition \ref{prop:seat_slot}]
            We will represent ball configurations $\eta \in \Omega_{< \infty}$ using the notation 
            \begin{align}
                \eta = 0^{\otimes \mathsf{m}_0} 1^{\otimes \mathsf{n}_1} \dots 0^{\mathsf{m}_{L}}1^{\otimes \mathsf{n}_{L+1}}0^{\otimes \mathsf{m}_{L+1}},
            \end{align}
        by which we mean that the first $\mathsf{m}_0$ entries of $\eta$ are 0's, the following $\mathsf{n}_1$ entries are 1's and so on. Notice that since the configuration has finitely many balls we have $\mathsf{m}_{L+1}=\infty$ and moreover \begin{align}
                \sum_{i = 1}^{L+1} \mathsf{n}_i = \sum_{x \in \N} \eta(x).
            \end{align}
        From the TS algorithm, 
        the first solitons that are identified by the algorithm are of the form
            \begin{align}
                0^{\otimes \mathsf{m}_{i}} 1^{\otimes \mathsf{m}_{i} } ~ \text{or} ~ 1^{\otimes \mathsf{n}_{i}} 0^{\otimes \mathsf{n}_{i} }.
            \end{align}
        for some $\mathsf{m}_i, \mathsf{n}_{i}$ such that $\mathsf{m}_i \le \mathsf{n}_{i+1}$, $i \neq 0$ or $\mathsf{n}_{i} \le  \mathsf{m}_i$ respectively. In the rest of this subsection, we call such solitons \emph{connected solitons}.
        
        Now we claim that 
        \begin{enumerate}
            \setlength{\leftskip}{10mm}
            \item[ Claim(A)] it is sufficient to show \eqref{eq:slot=seat} for sites that consist of connected solitons.
        \end{enumerate}
         To verify Claim(A) we will show that after removing a connected soliton, the seat number configuration for the remaining sites is ``invariant" in the following sense. Observe that after the removal of a connected soliton of the form $0^{\otimes \mathsf{m}_{i}} 1^{\otimes \mathsf{m}_{i} }$ or $1^{\otimes \mathsf{n}_{i}} 0^{\otimes \mathsf{n}_{i} }$, following the TS algorithm, we obtain the configuration
            \begin{align}
                \eta' = 0^{\otimes \mathsf{m}_0} 1^{\otimes \mathsf{n}_1}...0^{\otimes \mathsf{m}_{i - 1}}1^{\otimes \left( \mathsf{n}_{i} + \mathsf{n}_{i+1} - \mathsf{m}_{i} \right)}0^{\otimes \mathsf{m}_{i + 1}} ...0^{\otimes \mathsf{m}_{L}}1^{\otimes \mathsf{n}_{L+1}}0^{\otimes \mathsf{m}_{L+1}},
            \end{align}
        or 
            \begin{align}
                \eta'' = 0^{\otimes \mathsf{m}_0} 1^{\otimes \mathsf{n}_1}...1^{\otimes \mathsf{n}_{i-1}}0^{\otimes \left( \mathsf{m}_{i - 1} + \mathsf{m}_{i} - \mathsf{n}_{i} \right)}1^{\otimes  \mathsf{n}_{i+1}}0^{\otimes \mathsf{m}_{i + 1}} ...0^{\otimes \mathsf{m}_{L}}1^{\otimes \mathsf{n}_{L+1}}0^{\otimes \mathsf{m}_{L+1}},
            \end{align}
        respectively. In addition, after the removal of a soliton, the seat numbers given to other sites do not change, that is,
            \begin{align}
                \begin{dcases}
                    \mathcal{W}_{k}\left(x' \right) = \mathcal{W}_{k}\left(x' + 2\mathsf{m}_{i}\right) \ & \text{if} ~ 0^{\otimes \mathsf{m}_{i}} 1^{\otimes \mathsf{m}_{i} } ~ \text{ is removed}, \\
                    \mathcal{W}_{k}\left(x'' \right) = \mathcal{W}_{k}\left(x'' + 2\mathsf{n}_{i}\right) \ & \text{if} ~ 1^{\otimes \mathsf{n}_{i}} 0^{\otimes \mathsf{n}_{i} } ~ \text{ is removed},
                \end{dcases}
            \end{align}
        for any $k \in \N$, where $x', x''$ are defined as 
            \begin{align}
                x' &:= \sum_{j = 1}^{i} \left(\mathsf{m}_{j - 1} + \mathsf{n}_{j} \right) - 1, \\
                x'' &:= \sum_{j = 1}^{i} \left(\mathsf{m}_{j - 1} + \mathsf{n}_{j - 1} \right) - 1
            \end{align}
            with convention that $\mathsf{n}_0=0$, because from the rule of the TS algorithm, 
                \begin{align}
                    \begin{dcases}
                    W_{\mathsf{m}_{i}}\left(x' \right) = \mathsf{m}_i \ & \text{if} \ 0^{\otimes \mathsf{m}_{i}} 1^{\otimes \mathsf{m}_{i} } ~ \text{ is removed}, \\
                    W_{n_{i}}\left(x'' \right) = 0 \ & \text{if} \ 1^{\otimes \mathsf{n}_{i}} 0^{\otimes \mathsf{n}_{i} } ~ \text{ is removed}.
                    \end{dcases}
                \end{align}
        Thus we see that if a soliton $0^{\otimes \mathsf{m}_{i}} 1^{\otimes \mathsf{m}_{i} }$ is removed, then for any $x \in [1, x'] \cap \N$,
            \begin{align}
                x \in [1, x'] \cap \N ~ \text{is a} ~ (k,\sigma) ~ \text{seat in} \ \eta \ \text{if and only if} \ x \in [1, x'] \cap \N ~ \text{is a} ~ (k,\sigma) ~ \text{seat in} ~ \eta', 
            \end{align}
        while for any $y \in  [ x' + 2\mathsf{m}_{i} + 1, \infty ) \cap \N$,
            \begin{align}
                y \ \text{is a} ~ (k,\sigma) ~ \text{seat in} ~ \eta 
                 \ \text{if and only if} \  (y - 2\mathsf{m}_{i}) ~ \text{is a} ~ (k,\sigma) ~ \text{seat in} ~ \eta',
            \end{align}
        for any $k \in \N$ and $\sigma \in \{\uparrow, \downarrow \}$.
        Similarly, if a soliton $1^{\otimes \mathsf{n}_{i}} 0^{\otimes \mathsf{n}_{i} }$ is removed, then for any $x \in [1, x''] \cap \N$, 
            \begin{align}
                x \in [1, x''] \cap \N ~ \text{is a} ~ (k,\sigma) ~ \text{seat in} ~ \eta ~  \text{if and only if} ~ x \in [1, x''] \cap \N ~ \text{is a} ~ (k,\sigma) ~ \text{seat in} ~ \eta'', 
            \end{align}
        while for any $y \in  [ x'' + 2\mathsf{n}_{i} + 1, \infty ) \cap \N$,
            \begin{align}
                y ~ \text{is a} ~ (k,\sigma) ~ \text{seat in} ~ \eta ~  \text{if and only if} ~ (y - 2\mathsf{n}_{i}) ~ \text{is a} ~ (k,\sigma) ~ \text{seat in} ~ \eta'',
            \end{align}
        for any $k \in \N$ and $\sigma \in \{\uparrow, \downarrow\}$. 
        Hence, by considering multiple iterations of the TS algorithm and its inverse, we see that if the seat number configuration is determined for each connected soliton, the seat number configuration of the original ball configuration is completely determined. Also, by considering multiple iterations of the TS algorithm and its inverse, the slot configuration of the original ball configuration is also determined. Therefore Claim(A) is proved.
        
        Now we show \eqref{eq:slot=seat} for the case when $x$ belongs to a connected soliton. For this purpose, we divide the cases as follows.
            \begin{itemize}
                \item If a soliton $ 0^{\otimes \mathsf{m}_{i}} 1^{\otimes  \mathsf{m}_{i} }$ is detected by the TS algorithm, then we have $\mathsf{n}_{j} > \mathsf{m}_{i}$ for $j = i, i + 1$. Observing that
                    \begin{align}
                        \mathcal{W}_{\ell}(x') = 1, 
                    \end{align}
                for any $l \le \mathsf{n}_{i}$, we obtain 
                    \begin{align}
                        \begin{dcases}\eta^{\downarrow}_{\ell}(x' + l) = 1 \\ \eta^{\uparrow}_{\ell}(x' + \mathsf{m}_{i} + l) = 1
                        \end{dcases}
                    \end{align}
                for any $1 \le l \le \mathsf{m}_{i}$. On the other hand, from the definition of the slot configuration, we get 
                    \begin{align}
                        \begin{dcases}
                            \nu(x' + l) = l - 1 \\
                            \nu(x' + \mathsf{m}_{i} + l) = l - 1
                        \end{dcases}
                    \end{align}
                for any $1 \le l \le \mathsf{m}_{i}$. Therefore in this case \eqref{eq:slot=seat} holds. 
                \item If a soliton $ 1^{\otimes \mathsf{n}_{i}} 0^{\otimes \mathsf{n}_{i} }$ is detected by the TS algorithm, then we have $\mathsf{m}_{j} > \mathsf{n}_{i}$ for $j = i - 1, i$. Observing that
                    \begin{align}
                        \mathcal{W}_{\ell}(x'') = 0, 
                    \end{align}
                for any $l \le \mathsf{m}_{i - 1}$,  we obtain 
                    \begin{align}
                        \begin{dcases}\eta^{\uparrow}_{\ell}(x'' + l) = 1 \\ \eta^{\downarrow}_{\ell}(x'' + \mathsf{n}_{i} + l) = 1
                        \end{dcases}
                    \end{align}
                for any $1 \le l \le \mathsf{n}_{i}$. On the other hand, from the definition of the slot configuration, we get 
                    \begin{align}
                        \begin{dcases}
                            \nu(x'' + l) = l - 1 \\
                            \nu(x'' + \mathsf{n}_{i} + l) = l - 1
                        \end{dcases}
                    \end{align}
                for any $1 \le l \le \mathsf{n}_{i}$. Therefore in this case \eqref{eq:slot=seat} holds. 
            \end{itemize}
        Hence by combining the above with Claim(A), \eqref{eq:slot=seat}
        is shown for any $k \in \N$ and $x \in \Z_{\ge 0}$. 
        
        We now compute formulas for $\tilde{\xi}(\cdot)$ and $\tilde{\zeta}(\cdot)$, which were defined in \eqref{eq:defoftilde1} and \eqref{eq:defoftilde2}, in terms of the seat number configuration. By using \eqref{eq:slot=seat}, we have  
            \begin{equation}
                \begin{split}
                    \tilde{\xi}_{k}(x) &= x -  \sum_{y = 1}^x \mathbf{1}_{\left\{ \nu(y) \le k \right\}} \\
                &= x - \sum_{\ell = 1}^{k} \sum_{y = 1}^x \left( \eta^{\uparrow}_{\ell}(y) + \eta^{\downarrow}_{\ell}(y) \right) \\
                &= \xi_{k}(x). \label{eq:slot=seat_xi}
                \end{split}
            \end{equation}
        A direct consequence of \eqref{eq:slot=seat_xi} is $\tilde{s}_{k}(\cdot) = s_{k}(\cdot)$. In addition, by using $\tilde{s}_{k}(\cdot) = s_{k}(\cdot)$, Lemmas \ref{lem:pm_eq} and \ref{lem:repoftilde}, $\tilde{\zeta}_{k}(\cdot)$ can be represented as
            \begin{align}
                \tilde{\zeta}_{k}(i) &= \frac{1}{2} \sum_{\sigma \in \{\uparrow , \downarrow\}} \sum_{y = \tilde{s}_{k}(i) + 1}^{\tilde{s}_{k}(i+1)} \left(\eta^{\sigma}_{k}(y) - \eta^{\sigma}_{k+1}(y) \right) \\
                &= \frac{1}{2}\sum_{\sigma \in \{\uparrow , \downarrow\}} \sum_{y = s_{k}(i) + 1}^{s_{k}(i+1)} \left(\eta^{\sigma}_{k}(y) - \eta^{\sigma}_{k+1}(y) \right) \\
                &= \zeta_{k}(i).
            \end{align}
        This concludes the proof.
        \end{proof}

        We conclude this subsection by describing the relationship between solitons and $\tau_{k}(\cdot)$, and the characterization of the slots via the carrier processes.

            \begin{proposition}\label{prop:tau_soliton}
                Let $\eta \in \Omega_{<\infty}$ and $k \in \N$. Then, $x \in \N$ is a rightmost component of a $k$-soliton if and only if $x = \tau_{k}(j)$ for some $j \in \N$.
            \end{proposition}
            
            \begin{proposition}\label{prop:char_slot}
    	        Let $\eta \in \Omega_{<\infty}$. A site $x \in \N$ is $k$-slot if and only if one of the {following statements} hold : 
    	            \begin{itemize}
    	                \item $\eta(x) = 1$ and $\min \left\{\ell \in \N  \ ; \ \ell -  W_{\ell}(x - 1) \ge 1 \right\} \ge k+1$. 
    	                \item $\eta(x) = 0$ and $\min \left\{\ell \in \N \ ; \ W_{\ell}(x - 1) \ge 1 \right\} \ge k + 1$.
    	            \end{itemize}

    	    \end{proposition}

            \begin{proof}[Proof of Proposition \ref{prop:tau_soliton}] 

                From Proposition \ref{prop:seat_slot}, we see that $x$ is a rightmost component of a $k$-soliton if and only if $\eta^{\uparrow}_{k}(x) + \eta^{\downarrow}_{k}(x) = 1$, $x \in \left( s_{k}(i), s_{k}(i + 1) \right)$ for some $i \in \Z_{\ge 0}$ and 
                    \begin{align}
                        \sum_{y = s_{k}(i) + 1}^{x} \mathbf{1}_{\{\nu\left(y\right) = k - 1\}} = \sum_{\sigma \in \left\{\uparrow, \downarrow \right\}} \sum_{y = s_{k}(i) + 1}^{x} \eta^{\sigma}_{k}(y) = 2n
                    \end{align}
                for some $n \in \N$. On the other hand, from Lemmas \ref{lem:1} and \ref{lem:pm_eq}, we also see that $x = \tau_{k}(j)$ for some $j \in \Z_{\ge 0}$ if and only if $\eta^{\uparrow}_{k}(x) + \eta^{\downarrow}_{k}(x) = 1$, $x \in \left( s_{k}(i), s_{k}(i + 1) \right)$ for some $i \in \Z_{\ge 0}$ and 
                    \begin{align}
                        \sum_{\sigma \in \left\{\uparrow, \downarrow \right\}} \sum_{y = s_{k}(i) + 1}^{x} \eta^{\sigma}_{k}(y) = 2n
                    \end{align}
                for some $n \in \N$. By comparing the above two equivalences, this proposition is proved.

            \end{proof}
         
	        \begin{proof}[Proof of Proposition \ref{prop:char_slot}]
	            From \eqref{eq:cap_seat}, we have
            	    \begin{align}
            	        \eta(x) = 1 ~ \text{and} ~ \min \left\{\ell \in \N \ ; \ \ell -  W_{\ell}(x - 1) \ge 1 \right\} = k \ &\text{if and only if} \ \mathcal{W}_{k}(x-1) = 0, ~ \mathcal{W}_{k}(x) = 1,
            	    \end{align}
                and
                    \begin{align}
            	        \eta(x) = 0 ~ \text{and} ~ \min \left\{\ell \in \N \ ; \ W_{\ell}(x - 1) \ge 1 \right\} = k \ &\text{if and only if} \ \mathcal{W}_{k}(x-1) = 1, ~ \mathcal{W}_{k}(x) = 0,
            	    \end{align}
                for any $x \in \N$. Therefore, from Proposition \ref{prop:seat_slot}, the assertion of this proposition holds. 

	        \end{proof}
    
    \subsection{Proofs of Theorem \ref{thm:2} and Theorem \ref{thm:slot_finite}}
    We finally come to the proof of Theorem \ref{thm:2}, providing an explicit relation between the KKR bijection and the slot configuration. Then, by using Theorem \ref{thm:2}, we show Theorem \ref{thm:slot_finite}. 
        \begin{proof}[Proof of Theorem \ref{thm:2}]
            First we note that since $\eta \in \Omega_{\infty}$, from Proposition \ref{prop:seat_KKR} the rigging $\mathbf{J} = (J_{k})$ associated {with} $\eta$ is given by 
                \begin{align}
                    J_{k} =\left( p_{k}\left( t_{k}(j) \right) \ ; \ j = 1, \dots,  m_{k} \right),
                \end{align}
            where 
                \begin{align}
                    m_{k} &:= \lim_{x \to \infty} m_{k}(x), \\
                    t_{k}(j) &:= \lim_{x \to \infty} t_{k}(x,j) = \max \left\{ y \in \N \ ; \  m_{k}(y) = j, \ \eta^{\uparrow}_{k}(y) = 1 \right\}.
                \end{align}
            In addition, since $t_{k}(j)$ is a $(k,\uparrow)$-seat, from Lemma \ref{lem:1} we obtain
                \begin{align}
                    \xi_{k}\left(t_{k}(j) \right) - p_{k}\left(t_{k}(j) \right) &= \sum_{y = 1}^{t_k(j)} \sum_{\ell = 1}^{k} \left( \eta^{\uparrow}_{\ell}(y) - \eta^{\downarrow}_{\ell}(y) \right) \\
                    &= k.
                \end{align}
            Thus we have
                \begin{align}
                    \left| \left\{ j \in \N \ ; \ J_{k,j} = i - k \right\} \right| = \left| \left\{ j \in \N \ ;\  s_{k}(i) \le t_{k}(j) < s_{k}(i+1) \right\} \right|.
                \end{align}
            On the other hand, from Proposition \ref{prop:match} and the definitions of $t_{k}(\cdot)$ and $\tau_{k}(\cdot)$, for any $j \in \Z_{\ge 0}$ we have
                \begin{align}
                    \tau_{k}(j) < t_{k}(j + 1) \le \tau_{k}(j + 1).
                \end{align} 
            In addition, if $\tau_{k}(j), \tau_{k}(j+1)$ satisfies 
                \begin{align}
                    \tau_{k}(j) < s_{k}(i+1) < \tau_{k}(j+1),
                \end{align}
            for some $i$, then from Proposition \ref{prop:seat_KKR} and Lemma \ref{lem:pm_eq} we obtain $m\left( s_{k}(i+1) \right) = m^{\uparrow}\left( s_{k}(i+1) \right) =  j$, and thus we have
                \begin{align}
                    s_{k}(i+1) < t_{k}(j+1).
                \end{align}
            From the above, we have 
            \[
            \tau_{k}(j) < s_{k}(i+1) < \tau_{k}(j+1) \ \text{if and only if} \ t_{k}(j) < s_{k}(i+1) < t_{k}(j+1).
            \]
            Since $s_k(i) \neq t_k(j)$ and $ s_k(i) \neq \tau_k(j)$ for any $i,j$, combining with Proposition \ref{prop:seat_slot}, we have
                \begin{align}
                    \left| \left\{ j \in \N \ ; \ J_{k,j} = i - k \right\} \right| &= \left| \left\{ j \in \N \ ; \ s_{k}(i) \le t_{k}(j) < s_{k}(i+1) \right\} \right| \\
                    &= \left| \left\{ j \in \N \ ; \ s_{k}(i) \le \tau_{k}(j) < s_{k}(i+1) \right\} \right| \\
                    &= \zeta_{k}(i) \\ &= \tilde{\zeta}_{k}(i).
                \end{align}
            
        \end{proof}
 
        \begin{proof}[Proof of Theorem \ref{thm:slot_finite}]
        From Theorem \ref{thm:2} and Theorem \ref{thm:KOSTY}, we have
            \begin{align}
                T_{\ell} \tilde{\zeta}_{k}(i) &= \left| \left\{ j \in \N \ ; \ T_{\ell}J_{k,j} = i - k \right\} \right| \\
                &= \left| \left\{ j \in \N \ ; \ J_{k,j} + (k \wedge \ell) = i - k \right\} \right| \\
                &= \tilde{\zeta}_{k}\left( i - (k \wedge \ell) \right).
            \end{align}
        \end{proof}
        
\appendix
        
    \section{Takahashi-Satsuma Algorithm}\label{app:TS}
    
        Given a configuration $\eta$, we can decompose it into $k$-solitons, for $k\geq 1$, which are certain substrings of $\eta$ consisting of $k$ ``$1$"s and $k$ ``$0$"s. Such {a} decomposition is produced by the Takahashi-Satsuma algorithm \cite{TS} described below. The procedure consists in iteratively scanning $\eta$ identifying and crossing out $k$-solitons at each iteration. We call a \emph{run} of $\eta$ a maximal substring of consecutive equal letters.

\begin{algorithm}[H]
    	             Start with a configuration $\eta$
    	                \\
                        \While{there are still uncrossed 1's in $\eta$} {Considering only uncrossed elements of $\eta$, select the leftmost run whose length is at least as long as the length (denote it by $k$) of the run preceding it \\ 
                        Identify a soliton of size $k$, or simply $k$-soliton, consisting of the first $k$ letters of this run and the last $k$ letters of the run preceding it
                        \\
                        Cross out these $2k$ letters from $\eta$}
                \end{algorithm}
An example of applying the above algorithm to $\eta = 11001110110001100000\dots$ is shown in Figure \ref{ex:TSforex}. Then we see that in $\eta$, there are one $4$-soliton, two $2$-solitons and one $1$-soliton.
\begin{figure}[H]
                    \centering
                    \begin{equation}
                        \begin{matrix}
                        \eta = & 1 & 1 & 0 & 0 & 1 & 1 & 1 & 0 & 1 & 1 & 0 & 0 & 0 & 1 & 1 & 0 & 0 & 0 & 0 & 0 & \cdots
                        \\
                        & \xc{red}{1} & \xc{red}{1} & \xc{red}{0} & \xc{red}{0} & 1 & 1 & 1 & 0 & 1 & 1 & 0 & 0 & 0 & 1 & 1 & 0 & 0 & 0 & 0 & 0 & \cdots
                        \\
                        & \xc{red}{1} & \xc{red}{1} & \xc{red}{0} & \xc{red}{0} & 1 & 1 & 1 & \xc{blue}{0} & \xc{blue}{1} & 1 & 0 & 0 & 0 & 1 & 1 & 0 & 0 & 0 & 0 & 0 & \cdots
                        \\
                        & \xc{red}{1} & \xc{red}{1} & \xc{red}{0} & \xc{red}{0} & 1 & 1 & 1 & \xc{blue}{0} & \xc{blue}{1} & 1 & 0 & 0 & 0 & \xc{green}{1} & \xc{green}{1} & \xc{green}{0} & \xc{green}{0} & 0 & 0 & 0 & \cdots
                        \\
                        & \xc{red}{1} & \xc{red}{1} & \xc{red}{0} & \xc{red}{0} & \xc{brown}{1} & \xc{brown}{1} & \xc{brown}{1} & \xc{blue}{0} & \xc{blue}{1} & \xc{brown}{1} & \xc{brown}{0} & \xc{brown}{0} & \xc{brown}{0} & \xc{green}{1} & \xc{green}{1} & \xc{green}{0} & \xc{green}{0} & \xc{brown}{0} & 0 & 0 & \cdots
                        \\
                        & \color{red}{1} & \color{red}{1} & \color{red}{0} & \color{red}{0} & \color{brown}{1} & \color{brown}{1} & \color{brown}{1} & \color{blue}{0} & \color{blue}{1} & \color{brown}{1} & \color{brown}{0} & \color{brown}{0} & \color{brown}{0} & \color{green}{1} & \color{green}{1} & \color{green}{0} & \color{green}{0} & \color{brown}{0} & 0 & 0 & \cdots
                        \end{matrix}
                    \end{equation}
                    \caption{Identifying solitons in $\eta$ by the TS Algorithm}\label{ex:TSforex}
                \end{figure}


\paragraph{Funding statement}
The work of MM has been supported by the European Union’s Horizon 2020 research and innovation programme under the Marie Sk\l odowska-Curie grant agreement No. 101030938.
The work of MS has been supported by JSPS KAKENHI Grant Nos. JP18H03672, JP19H01792, JP22H01143.
The work of TS has been supported by JSPS KAKENHI Grant Nos. JP18H03672, JP19K03665, JP21H04432, JP22H01143.
The work of HS has been supported by JST CREST Grant No. JPMJCR1913 and
JSPS KAKENHI Grant No. JP21K20332.


\end{document}